\documentclass[11pt,a4paper]{article}

\usepackage{tikz}
\usepackage{amsthm}
\usepackage{amsmath}
\usepackage{amssymb}
\usepackage{mathrsfs}
\usepackage{geometry}
\usepackage{graphicx}
\usepackage{amsfonts}
\usepackage{epstopdf}
\usepackage{enumerate}
\usepackage{hyperref}
\usepackage{subfigure}

\newcommand{\vphi}{\varphi}
\newcommand{\veps}{\varepsilon}

\newcommand{\md}{\mathrm{d}}
\newcommand{\loc}{{\mathrm{loc}}}

\newcommand{\R}{\mathbb{R}}

\newcommand{\rmnum}[1]{\romannumeral #1} 
\newcommand{\Rmnum}[1]{\uppercase\expandafter{\romannumeral#1}} 

\newcommand{\Ee}{\mathcal{E}}

\newcommand{\Ff}{\mathcal{F}}

\newcommand{\calE}{\mathcal{E}}
\newcommand{\calF}{\mathcal{F}}
\newcommand{\calV}{\mathcal{V}}
\newcommand{\calW}{\mathcal{W}}

\newcommand{\frakE}{\mathfrak{E}}
\newcommand{\frakD}{\mathfrak{D}}
\newcommand{\frakR}{\mathfrak{R}}

\newcommand{\myset}[1]{\left\{#1\right\}}
\newcommand{\mybar}[1]{\overline{#1}}

\newtheorem{mythm}{Theorem}[section]
\newtheorem{myprop}[mythm]{Proposition}
\newtheorem{mylem}[mythm]{Lemma}
\newtheorem{mycor}[mythm]{Corollary}
\newtheorem{myrmk}[mythm]{Remark}

\newtheorem{mydef}[mythm]{Definition}

\AtEndDocument{
\bigskip{\footnotesize%
  \textsc{Fakult\"{a}t f\"{u}r Mathematik, Universit\"{a}t Bielefeld, Postfach 100131, 33501 Bielefeld, Germany.} \par  
  \textit{E-mail address}: \texttt{grigor@math.uni-bielefeld.de} \par
}
\bigskip{\footnotesize%
  \textsc{Fakult\"{a}t f\"{u}r Mathematik, Universit\"{a}t Bielefeld, Postfach 100131, 33501 Bielefeld, Germany.} \par  
  \textit{E-mail address}: \texttt{ymeng@math.uni-bielefeld.de} \par
  and \par
  \addvspace{\medskipamount}
  \textsc{Department of Mathematical Sciences, Tsinghua University, Beijing 100084, China.} \par
  \textit{E-mail address}: \texttt{meng-yang13@mails.tsinghua.edu.cn}
}}

\begin{document}

\title{Local and Non-Local Dirichlet Forms on the Sierpi\'nski Carpet}
\author{Alexander Grigor'yan and Meng Yang}
\date{}

\maketitle

\abstract{We give a purely analytic construction of a self-similar local regular Dirichlet form on the Sierpi\'nski carpet using approximation of stable-like non-local closed forms which gives an answer to an open problem in analysis on fractals.}

\footnote{\textsl{Date}: \today}
\footnote{The authors were supported by SFB701 and SFB1283 of the German Research Council (DFG). The second author is very grateful to Dr. Qingsong Gu for very helpful discussions. Part of the work was carried out while the second author was visiting the Chinese University of Hong Kong, he is very grateful to Prof. Ka-Sing Lau for the arrangement of the visit.}

\section{Introduction}

Sierpi\'nski carpet (SC) is a typical example of non p.c.f. (post critically finite) self-similar sets. It was first introduced by Wac\l aw Sierpi\'nski in 1916 which is a generalization of Cantor set in two dimensions, see Figure \ref{fig_SC}.

\begin{figure}[ht]
\centering
\includegraphics[width=0.5\textwidth]{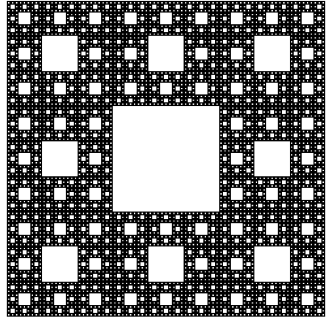}
\caption{Sierpi\'nski Carpet}\label{fig_SC}
\end{figure}

SC can be obtained as follows. Divide the unit square into nine congruent small squares, each with sides of length $1/3$, remove the central one. Divide each of the eight remaining small squares into nine congruent squares, each with sides of length $1/9$, remove the central ones, see Figure \ref{fig_construction}. Repeat above procedure infinitely many times, SC is the compact connected set $K$ that remains.

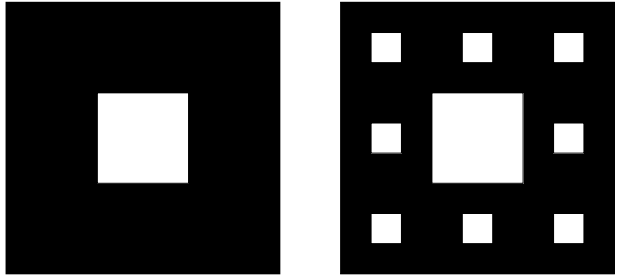
\begin{figure}[ht]
\centering
\begin{tikzpicture}[scale=0.4]
\draw[fill=black] (0,0)--(9,0)--(9,9)--(0,9)--cycle;
\draw[fill=white] (3,3)--(6,3)--(6,6)--(3,6)--cycle;

\draw[fill=black] (11,0)--(20,0)--(20,9)--(11,9)--cycle;
\draw[fill=white] (14,3)--(17,3)--(17,6)--(14,6)--cycle;
\draw[fill=white] (12,1)--(13,1)--(13,2)--(12,2)--cycle;
\draw[fill=white] (15,1)--(16,1)--(16,2)--(15,2)--cycle;
\draw[fill=white] (18,1)--(19,1)--(19,2)--(18,2)--cycle;
\draw[fill=white] (12,4)--(13,4)--(13,5)--(12,5)--cycle;
\draw[fill=white] (18,4)--(19,4)--(19,5)--(18,5)--cycle;
\draw[fill=white] (12,7)--(13,7)--(13,8)--(12,8)--cycle;
\draw[fill=white] (15,7)--(16,7)--(16,8)--(15,8)--cycle;
\draw[fill=white] (18,7)--(19,7)--(19,8)--(18,8)--cycle;

\end{tikzpicture}
\caption{The Construction of Sierpi\'nski Carpet}\label{fig_construction}
\end{figure}

In recent decades, self-similar sets have been regarded as underlying spaces for analysis and probability. Apart from classical Hausdorff measures, this approach requires the introduction of Dirichlet forms. Local regular Dirichlet forms or associated diffusions (also called Brownian motion (BM)) have been constructed in many fractals, see \cite{BP88,BB89,Lin90,KZ92,Kig93,Bar98,Kig01}. In p.c.f. self-similar sets including Sierpi\'nski gasket, this construction is relatively transparent, while similar construction on SC is much more involved.

For the first time, BM on SC was constructed by Barlow and Bass \cite{BB89} using \emph{extrinsic} approximation domains in $\R^2$ (see black domains in Figure \ref{fig_construction}) and time-changed reflected BMs in those domains. Technically, \cite{BB89} is based on the following two ingredients in approximation domains:
\begin{enumerate}[(a)]
\item\label{enum_a} Certain resistance estimates.
\item\label{enum_b} Uniform Harnack inequality for harmonic functions with Neumann boundary condition.
\end{enumerate}
For the proof of the uniform Harnack inequality, Barlow and Bass used certain probabilistic techniques based on Knight move argument (this argument was generalized later in \cite{BB99a} to deal also with similar problems in higher dimensions).

Subsequently, Kusuoka and Zhou \cite{KZ92} gave an alternative construction of BM on SC using \emph{intrinsic} approximation graphs and Markov chains in those graphs. However, in order to prove the convergence of Markov chains to a diffusion, they used the two aforementioned ingredients of \cite{BB89}, reformulated in terms of approximation graphs.

However, the problem of a purely analytic construction of a local regular Dirichlet form on SC (similar to that on p.c.f. self-similar sets) has been open until now and was explicitly raised by Hu \cite{Hu13}. The main result of this paper is a direct purely \emph{analytic} construction of a local regular Dirichlet form on SC.

The most essential ingredient of our construction is a certain resistance estimate in approximation graphs which is similar to the ingredient (\ref{enum_a}). We obtain the second ingredient---the uniform Harnack inequality in approximation graphs as a consequence of (\ref{enum_a}). A possibility of such an approach was mentioned in \cite{BCK05}. In fact, in order to prove a uniform Harnack inequality in approximation graphs, we extend resistance estimates from finite graphs to the infinite graphical SC (see Figure \ref{fig_graphSC}) and then deduce from them a uniform Harnack inequality-first on the infinite graph and then also on finite graphs. By this argument, we avoid the most difficult part of the proof in \cite{BB89}.

\begin{figure}[ht]
\centering
\begin{tikzpicture}[scale=0.3]

\foreach \x in {0,1,...,27}
\draw (\x,0)--(\x,28);

\foreach \y in {0,1,...,27}
\draw (0,\y)--(28,\y);

\foreach \x in {0,1,2}
\foreach \y in {0,1,2}
\draw[fill=white] (9*\x+3,9*\y+3)--(9*\x+6,9*\y+3)--(9*\x+6,9*\y+6)--(9*\x+3,9*\y+6)--cycle;

\draw[fill=white] (9,9)--(18,9)--(18,18)--(9,18)--cycle;

\foreach \x in {0,1,...,27}
\foreach \y in {0,0.5,1,...,27.5}
\draw[fill=black] (\x,\y) circle (0.08);

\foreach \y in {0,1,...,27}
\foreach \x in {0,0.5,1,...,27.5}
\draw[fill=black] (\x,\y) circle (0.08);

\draw[fill=white,draw=white] (9.25,9.25)--(17.75,9.25)--(17.75,17.75)--(9.25,17.75)--cycle;

\foreach \x in {0,1,2}
\foreach \y in {0,1,2}
\draw[fill=white,draw=white] (9*\x+3.25,9*\y+3.25)--(9*\x+5.75,9*\y+3.25)--(9*\x+5.75,9*\y+5.75)--(9*\x+3.25,9*\y+5.75)--cycle;

\end{tikzpicture}
\caption{The Infinite Graphical Sierpi\'nski Carpet}\label{fig_graphSC}
\end{figure}
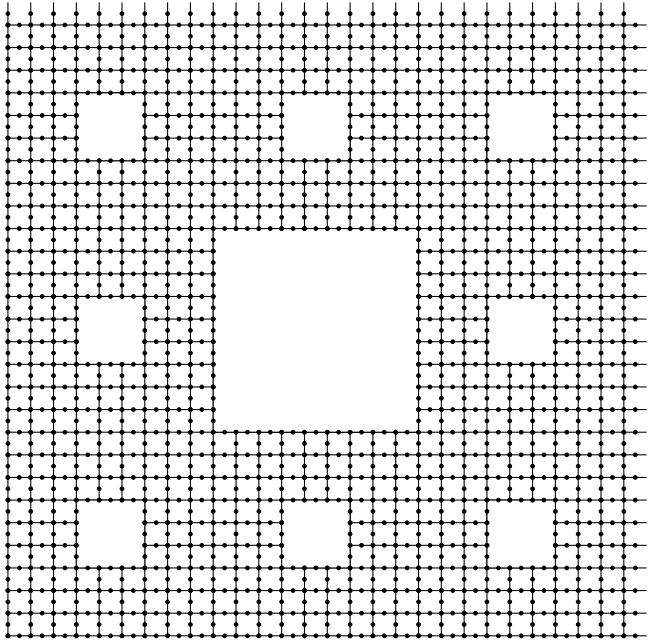

The self-similar local regular Dirichlet form $\calE_\loc$ on SC has the following self-similarity property. Let $f_0,\ldots,f_7$ be the contraction mappings generating SC. For all function $u$ in the domain $\calF_\loc$ of $\calE_\loc$ and for all $i=0,\ldots,7$, we have $u\circ f_i\in\calF_\loc$ and 
$$\calE_\loc(u,u)=\rho\sum_{i=0}^7\calE_\loc(u\circ f_i,u\circ f_i).$$
Here $\rho>1$ is a parameter from the aforementioned resistance estimates, whose exact value remains still unknown. Barlow, Bass and Sherwood \cite{BB90,BBS90} gave two bounds as follows:
\begin{itemize}
\item $\rho\in[7/6,3/2]$ based on shorting and cutting technique.
\item $\rho\in[1.25147,1.25149]$ based on numerical calculation.
\end{itemize}
McGillivray \cite{McG02} generalized above estimates to higher dimensions.

The heat semigroup associated with $\calE_\loc$ has a heat kernel $p_t(x,y)$ satisfying the following estimates: for all $x,y\in K,t\in(0,1)$
\begin{equation}\label{eqn_hk}
p_t(x,y)\asymp\frac{C}{t^{\alpha/\beta^*}}\exp\left(-c\left(\frac{|x-y|}{t^{1/\beta^*}}\right)^{\frac{\beta^*}{\beta^*-1}}\right),
\end{equation}
where $\alpha=\log8/\log3$ is the Hausdorff dimension of SC and
\begin{equation}\label{eqn_beta_up}
\beta^*:=\frac{\log(8\rho)}{\log3}.
\end{equation}
The parameter $\beta^*$ is called the \emph{walk dimension of BM} and is frequently denoted also by $d_w$. The estimates (\ref{eqn_hk}) were obtained by Barlow and Bass \cite{BB92,BB99a} and by Hambly, Kumagai, Kusuoka and Zhou \cite{HKKZ00}. Equivalent conditions of sub-Gaussian heat kernel estimates for local regular Dirichlet forms on metric measure spaces were explored by many authors, see Andres and Barlow \cite{AB15}, Grigor'yan and Hu \cite{GH14a,GH14b}, Grigor'yan, Hu and Lau \cite{GHL10,GHL15}, Grigor'yan and Telcs \cite{GT12}. We give an alternative proof of the estimates (\ref{eqn_hk}) based on the approach developed by the first author and others.

Consider the following stable-like non-local quadratic form
$$
\begin{aligned}
&\calE_\beta(u,u)=\int_K\int_K\frac{(u(x)-u(y))^2}{|x-y|^{\alpha+\beta}}\nu(\md x)\nu(\md y),\\
&\calF_\beta=\myset{u\in L^2(K;\nu):\calE_\beta(u,u)<+\infty},
\end{aligned}
$$
where $\alpha=\mathrm{dim}_{\mathcal{H}}K$ as above, $\nu$ is the normalized Hausdorff measure on $K$ of dimension $\alpha$, and $\beta>0$ is so far arbitrary. Then the \emph{walk dimension of SC} is defined as
\begin{equation}\label{eqn_beta_low}
\beta_*:=\sup\myset{\beta>0:(\Ee_\beta,\Ff_\beta)\text{ is a regular Dirichlet form on }L^2(K;\nu)}.
\end{equation}
Using the estimates (\ref{eqn_hk}) and subordination technique, it was proved in \cite{Pie00,GHL03} that $(\calE_\beta,\calF_\beta)$ is a regular Dirichlet form on $L^2(K;\nu)$ if $\beta\in(0,\beta^*)$ and that $\calF_\beta$ consists only of constant functions if $\beta>\beta^*$, which implies the identity
$$\beta_*=\beta^*.$$
In this paper, we give another proof of this identity without using the estimates (\ref{eqn_hk}), but using directly the definitions (\ref{eqn_beta_up}) and (\ref{eqn_beta_low}) of $\beta^*$ and $\beta_*$.

Barlow raised in \cite{Bar13} a problem of obtaining bounds of the walk dimension $\beta^*$ of BM without using directly $\calE_\loc$. We partially answer this problem by showing that
$$\beta_*\in\left[\frac{\log\left(8\cdot\frac{7}{6}\right)}{\log3},\frac{\log\left(8\cdot\frac{3}{2}\right)}{\log3}\right],$$
which gives then the same bound for $\beta^*$. However, the same bound for $\beta^*$ follows also from the estimate $\rho\in[7/6,3/2]$ mentioned above. We hope to be able to improve this approach in order to get better estimates of $\beta_*$ in the future.

Using the estimates (\ref{eqn_hk}) and subordination technique, it was proved in \cite{Pie08} that
\begin{equation}\label{eqn_approximation}
\varliminf_{\beta\uparrow\beta^*}(\beta^*-\beta)\calE_\beta(u,u)\asymp\calE_\loc(u,u)\asymp\varlimsup_{\beta\uparrow\beta^*}(\beta^*-\beta)\calE_\beta(u,u)
\end{equation}
for all $u\in\calF_\loc$.
This is similar to the following classical result
$$\lim_{\beta\uparrow2}(2-\beta)\int_{\R^n}\int_{\R^n}\frac{(u(x)-u(y))^2}{|x-y|^{n+\beta}}\md x\md y=C(n)\int_{\R^n}|\nabla u(x)|^2\md x,$$
for all $u\in W^{1,2}(\R^n)$, where $C(n)$ is some positive constant (see \cite[Example 1.4.1]{FOT11}). We reprove (\ref{eqn_approximation}) as a direct corollary of our construction without using the estimates (\ref{eqn_hk}).

The idea of our construction of $\calE_\loc$ is as follows. In the first step, we construct another quadratic form $E_\beta$ equivalent to $\calE_\beta$ and use it to prove the identity
\begin{equation}\label{eqn_walk}
\beta_*=\beta^*:=\frac{\log(8\rho)}{\log3}.
\end{equation}
It follows that $\calE_\beta$ is a regular Dirichlet form for all $\beta\in(\alpha,\beta^*)$. Then, we use another quadratic form $\frakE_\beta$, also equivalent to $\calE_\beta$, and define $\calE$ as a $\Gamma$-limit of a sequence $\myset{(\beta^*-\beta_n)\frakE_{\beta_n}}$ with $\beta_n\uparrow\beta^*$. We prove that $\calE$ is a regular closed form, where the main difficulty lies in the proof of the uniform density of the domain $\calF$ of $\calE$ in $C(K)$. However, $\calE$ is not necessarily Markovian, local or self-similar. In the last step, $\calE_\loc$ is constructed from $\calE$ by means of an argument from \cite{KZ92}. Then $\calE_\loc$ is a self-similar local regular Dirichlet form with a Kigami's like representation (\ref{eqn_Kigami}) which is similar to the representations in Kigami's construction on p.c.f. self-similar sets, see \cite{Kig01}. We use the latter in order to obtain certain resistance estimates for $\calE_\loc$, which imply the estimates (\ref{eqn_hk}) by \cite{GHL14,GH14a}.

Let us emphasize that the resistance estimates in approximation graphs and their conse\-qu\-ence---the uniform Harnack inequality, are mainly used in order to construct one \emph{good} function on $K$ with certain energy property and separation property, which is then used to prove the identity (\ref{eqn_walk}) and to ensure the non-triviality of $\calF$.

An important fact about the local regular Dirichlet form $\calE_\loc$ is that this Dirichlet form is a resistance form in the sense of Kigami whose existence gives many important corollaries, see \cite{Kig01,Kig03,Kig12}.

\section{Statement of the Main Results}
Consider the following points in $\R^2$:
$$p_0=(0,0),p_1=(\frac{1}{2},0),p_2=(1,0),p_3=(1,\frac{1}{2}),$$
$$p_4=(1,1),p_5=(\frac{1}{2},1),p_6=(0,1),p_7=(0,\frac{1}{2}).$$
Let $f_i(x)=(x+2p_i)/3$, $x\in\R^2$, $i=0,\ldots,7$. Then the Sierpi\'nski carpet (SC) is the unique non-empty compact set $K$ in $\R^2$ satisfying $K=\cup_{i=0}^7f_i(K)$.

Let $\nu$ be the normalized Hausdorff measure on $K$. Let $(\Ee_\beta,\Ff_\beta)$ be given by
$$
\begin{aligned}
&\Ee_\beta(u,u)=\int_K\int_K\frac{(u(x)-u(y))^2}{|x-y|^{\alpha+\beta}}\nu(\md x)\nu(\md y),\\
&\Ff_\beta=\myset{u\in L^2(K;\nu):\Ee_\beta(u,u)<+\infty},
\end{aligned}
$$
where $\alpha=\log8/\log3$ is Hausdorff dimension of SC, $\beta>0$ is so far arbitrary. Then $(\Ee_\beta,\Ff_\beta)$ is a quadratic form on $L^2(K;\nu)$ for all $\beta\in(0,+\infty)$. Note that $(\Ee_\beta,\Ff_\beta)$ is not necessary to be a regular Dirichlet form on $L^2(K;\nu)$ related to a stale-like jump process. The \emph{walk dimension} of SC is defined as
$$\beta_*:=\sup\myset{\beta>0:(\Ee_\beta,\Ff_\beta)\text{ is a regular Dirichlet form on }L^2(K;\nu)}.$$

Let
$$V_0=\myset{p_0,\ldots,p_7},V_{n+1}=\cup_{i=0}^7f_i(V_n)\text{ for all }n\ge0.$$
Then $\myset{V_n}$ is an increasing sequence of finite sets and $K$ is the closure of $\cup_{n=0}^\infty V_n$. Let $W_0=\myset{\emptyset}$ and 
$$W_n=\myset{w=w_1\ldots w_n:w_i=0,\ldots,7,i=1,\ldots,n}\text{ for all }n\ge1.$$
For all $w^{(1)}=w^{(1)}_1\ldots w^{(1)}_m\in W_m,w^{(2)}=w^{(2)}_1\ldots w^{(2)}_n\in W_n$, denote $w^{(1)}w^{(2)}$ as $w=w_1\ldots w_{m+n}\in W_{m+n}$ with $w_i=w^{(1)}_i$ for all $i=1,\ldots,m$ and $w_{m+i}=w^{(2)}_i$ for all $i=1,\ldots n$. For all $i=0,\ldots,7$, denote $i^n$ as $w=w_1\ldots w_n\in W_n$ with $w_k=i$ for all $k=1,\ldots,n$.

For all $w=w_1\ldots w_n\in W_n$, let
$$
\begin{aligned}
f_w&=f_{w_1}\circ\ldots\circ f_{w_n},\\
V_w&=f_{w_1}\circ\ldots\circ f_{w_n}(V_0),\\
K_w&=f_{w_1}\circ\ldots\circ f_{w_n}(K),\\
P_w&=f_{w_1}\circ\ldots\circ f_{w_{n-1}}(p_{w_n}),
\end{aligned}
$$
where $f_\emptyset=\mathrm{id}$ is the identity map.

Our semi-norm $E_\beta$ is given as follows.
$$E_\beta(u,u):=\sum_{n=1}^\infty3^{(\beta-\alpha)n}\sum_{w\in W_n}
{\sum_{\mbox{\tiny
$
\begin{subarray}{c}
p,q\in V_w\\
|p-q|=2^{-1}\cdot3^{-n}
\end{subarray}
$
}}}
(u(p)-u(q))^2.$$

Our first result is as follows.

\begin{mylem}\label{lem_equiv}
For all $\beta\in(\alpha,+\infty),u\in C(K)$, we have
$$E_\beta(u,u)\asymp\Ee_\beta(u,u).$$
\end{mylem}

The second author has established similar equivalence on Sierpi\'nski gasket (SG), see \cite[Theorem 1.1]{MY17}.

We use Lemma \ref{lem_equiv} to give bound of walk dimension as follows.

\begin{mythm}\label{thm_bound}
\begin{equation}\label{eqn_bound_beta}
\beta_*\in\left[\frac{\log\left(8\cdot\frac{7}{6}\right)}{\log3},\frac{\log\left(8\cdot\frac{3}{2}\right)}{\log3}\right].
\end{equation}
\end{mythm}
This estimate follows also from the results of \cite{BB90} and \cite{BBS90} where the same bound 
for $\beta^*$ was obtained by means of shorting and cutting techniques, while the identity $\beta_{*}=\beta^{*}$ 
follows from the sub-Gaussian heat kernel estimates by means of subordination technique. 
Here we prove the estimate (\ref{eqn_bound_beta}) of $\beta_*$ directly, without using heat kernel or subordination technique.

We give a direct proof of the following result.

\begin{mythm}\label{thm_walk}
$$\beta_*=\beta^*:=\frac{\log(8\rho)}{\log3},$$
where $\rho$ is some parameter in resistance estimates.
\end{mythm}

Hino and Kumagai \cite{HK06} established other equivalent semi-norms as follows. For all $n\ge1,u\in L^2(K;\nu)$, let
$$P_nu(w)=\frac{1}{\nu(K_w)}\int_{K_w}u(x)\nu(\md x),w\in W_n.$$
For all $w^{(1)},w^{(2)}\in W_n$, denote $w^{(1)}\sim_nw^{(2)}$ if $\mathrm{dim}_{\mathcal{H}}(K_{w^{(1)}}\cap K_{w^{(2)}})=1$. Let
$$\frakE_\beta(u,u):=\sum_{n=1}^\infty3^{(\beta-\alpha)n}
{\sum_{\mbox{\tiny
$
\begin{subarray}{c}
w^{(1)}\sim_nw^{(2)}
\end{subarray}
$
}}}
\left(P_nu(w^{(1)})-P_nu(w^{(2)})\right)^2.$$

\begin{mylem}\label{lem_equivHK}(\cite[Lemma 3.1]{HK06})
For all $\beta\in(0,+\infty),u\in L^2(K;\nu)$, we have
$$\frakE_\beta(u,u)\asymp\Ee_\beta(u,u).$$
\end{mylem}

We combine $E_\beta$ and $\frakE_\beta$ to construct a local regular Dirichlet form on $K$ using $\Gamma$-convergence technique as follows.

\begin{mythm}\label{thm_BM}
There exists a self-similar strongly local regular Dirichlet form $(\Ee_{\loc},\Ff_{\loc})$ on $L^2(K;\nu)$ satisfying
\begin{align}
&\Ee_{\loc}(u,u)\asymp\sup_{n\ge1}3^{(\beta^*-\alpha)n}\sum_{w\in W_n}
{\sum_{\mbox{\tiny
$
\begin{subarray}{c}
p,q\in V_w\\
|p-q|=2^{-1}\cdot3^{-n}
\end{subarray}
$
}}}
(u(p)-u(q))^2,\label{eqn_Kigami}\\
&\Ff_{\loc}=\myset{u\in C(K):\sup_{n\ge1}3^{(\beta^*-\alpha)n}\sum_{w\in W_n}
{\sum_{\mbox{\tiny
$
\begin{subarray}{c}
p,q\in V_w\\
|p-q|=2^{-1}\cdot3^{-n}
\end{subarray}
$
}}}
(u(p)-u(q))^2<+\infty}.\nonumber
\end{align}
\end{mythm}
By uniqueness result in \cite{BBKT10}, we have above local regular Dirichlet form coincides with that given by \cite{BB89} and \cite{KZ92}.

We have a direct corollary that non-local Dirichlet forms can approximate local Dirichlet form as follows.

\begin{mycor}\label{cor_approx}
There exists some positive constant $C$ such that for all $u\in\Ff_\loc$
$$\frac{1}{C}\Ee_\loc(u,u)\le\varliminf_{\beta\uparrow\beta^*}(\beta^*-\beta)\Ee_\beta(u,u)\le\varlimsup_{\beta\uparrow\beta^*}(\beta^*-\beta)\Ee_{\beta}(u,u)\le C\Ee_\loc(u,u).$$
\end{mycor}

Let us introduce the notion of Besov spaces. Let $(M,d,\mu)$ be a metric measure space and $\alpha,\beta>0$ two parameters. Let
$$
\begin{aligned}
&\left[u\right]_{B^{2,2}_{\alpha,\beta}(M)}&=\sum_{n=1}^\infty3^{(\alpha+\beta)n}\int\limits_M\int\limits_{d(x,y)<3^{-n}}(u(x)-u(y))^2\mu(\md y)\mu(\md x),\\
&\left[u\right]_{B^{2,\infty}_{\alpha,\beta}(M)}&=\sup_{n\ge1}3^{(\alpha+\beta)n}\int\limits_M\int\limits_{d(x,y)<3^{-n}}(u(x)-u(y))^2\mu(\md y)\mu(\md x),\\
\end{aligned}
$$
and
$$
\begin{aligned}
B_{\alpha,\beta}^{2,2}(M)&=\myset{u\in L^2(M;\mu):[u]_{B^{2,2}_{\alpha,\beta}(M)}<+\infty},\\
B_{\alpha,\beta}^{2,\infty}(M)&=\myset{u\in L^2(M;\mu):[u]_{B^{2,\infty}_{\alpha,\beta}(M)}<+\infty}.\\
\end{aligned}
$$
By the following Lemma \ref{lem_equiv1} and Lemma \ref{lem_holder}, we have $\Ff_\beta=B^{2,2}_{\alpha,\beta}(K)$ for all $\beta\in(\alpha,+\infty)$.

We characterize $(\Ee_\loc,\Ff_\loc)$ on $L^2(K;\nu)$ as follows.

\begin{mythm}\label{thm_Besov}
$\Ff_\loc=B_{\alpha,\beta^*}^{2,\infty}(K)$ and $\calE_\loc(u,u)\asymp[u]_{B^{2,\infty}_{\alpha,\beta^*}(K)}$ for all $u\in\Ff_\loc$.
\end{mythm}

We give a direct proof of this theorem using (\ref{eqn_Kigami}) and thus avoiding heat kernel estimates, while using some geometric properties of SC. Similar characterization of the domains of local regular Dirichlet forms was obtained in \cite{Jon96} for SG, \cite{Pie99} for simple nested fractals and \cite{HW06} for p.c.f. self-similar sets. In \cite{Pie00,GHL03,KS05}, the characterization of the domains of local regular Dirichlet forms was obtained in the setting of metric measure spaces assuming heat kernel estimates.

Finally, using (\ref{eqn_Kigami}) of Theorem \ref{thm_BM}, we give an alternative proof of sub-Gaussian heat kernel estimates as follows.

\begin{mythm}\label{thm_hk}
$(\calE_\loc,\calF_\loc)$ on $L^2(K;\nu)$ has a heat kernel $p_t(x,y)$ satisfying
$$p_t(x,y)\asymp\frac{C}{t^{\alpha/\beta^*}}\exp\left(-c\left(\frac{|x-y|}{t^{1/\beta^*}}\right)^{\frac{\beta^*}{\beta^*-1}}\right),$$
for all $x,y\in K,t\in(0,1)$.
\end{mythm}

This paper is organized as follows. In Section \ref{sec_equiv}, we prove Lemma \ref{lem_equiv}. In Section \ref{sec_bound}, we prove Theorem \ref{thm_bound}. In Section \ref{sec_resistance}, we give resistance estimates. In Section \ref{sec_harnack}, we give uniform Harnack inequality. In Section \ref{sec_monotone}, we give two weak monotonicity results. In Section \ref{sec_good}, we construct one good function. In Section \ref{sec_walk}, we prove Theorem \ref{thm_walk}. In Section \ref{sec_BM}, we prove Theorem \ref{thm_BM}. In Section \ref{sec_Besov}, we prove Theorem \ref{thm_Besov}. In Section \ref{sec_hk}, we prove Theorem \ref{thm_hk}.

NOTATION. The letters $c,C$ will always refer to some positive constants and may change at each occurrence. The sign $\asymp$ means that the ratio of the two sides is bounded from above and below by positive constants. The sign $\lesssim$ ($\gtrsim$) means that the LHS is bounded by positive constant times the RHS from above (below).

\section{Proof of Lemma \ref{lem_equiv}}\label{sec_equiv}

We need some preparation as follows.

\begin{mylem}\label{lem_equiv1}(\cite[Lemma 2.1]{MY17})
For all $u\in L^2(K;\nu)$, we have
$$\int_K\int_K\frac{(u(x)-u(y))^2}{|x-y|^{\alpha+\beta}}\nu(\md x)\nu(\md y)\asymp\sum_{n=0}^\infty3^{(\alpha+\beta)n}\int_K\int_{B(x,3^{-n})}(u(x)-u(y))^2\nu(\md y)\nu(\md x).$$
\end{mylem}

\begin{mycor}\label{cor_arbi}(\cite[Corollary 2.2]{MY17})
Fix arbitrary integer $N\ge0$ and real number $c>0$. For all $u\in L^2(K;\nu)$, we have
$$\int_K\int_K\frac{(u(x)-u(y))^2}{|x-y|^{\alpha+\beta}}\nu(\md x)\nu(\md y)\asymp\sum_{n=N}^\infty3^{(\alpha+\beta)n}\int_K\int_{B(x,c3^{-n})}(u(x)-u(y))^2\nu(\md y)\nu(\md x).$$
\end{mycor}

The proofs of above results are essentially the same as those in \cite{MY17} except that contraction ratio $1/2$ is replaced by $1/3$. We also need the fact that SC satisfies the chain condition, see \cite[Definition 3.4]{GHL03}.

The following result states that a Besov space can be embedded in some H\"older space.
\begin{mylem}\label{lem_holder}(\cite[Theorem 4.11 (\rmnum{3})]{GHL03})
Let $u\in L^2(K;\nu)$ and
$$E(u):=\int_K\int_K\frac{(u(x)-u(y))^2}{|x-y|^{\alpha+\beta}}\nu(\md x)\nu(\md y),$$
then
$$|u(x)-u(y)|^2\le cE(u)|x-y|^{\beta-\alpha}\text{ for }\nu\text{-almost every }x,y\in K,$$
where $c$ is some positive constant.
\end{mylem}
\begin{myrmk}
If $E(u)<+\infty$, then $u\in C^{\frac{\beta-\alpha}{2}}(K)$.
\end{myrmk}

Note that the proof of above lemma does not rely on heat kernel.

We divide Lemma \ref{lem_equiv} into the following Theorem \ref{thm_equiv1} and Theorem \ref{thm_equiv2}. The idea of the proofs of these theorems comes form \cite{Jon96}. But we do need to pay special attention to the difficulty brought by non p.c.f. property.

\begin{mythm}\label{thm_equiv1}
For all $u\in C(K)$, we have
$$\sum_{n=1}^\infty3^{(\beta-\alpha)n}\sum_{w\in W_n}
{\sum_{\mbox{\tiny
$
\begin{subarray}{c}
p,q\in V_w\\
|p-q|=2^{-1}\cdot3^{-n}
\end{subarray}
$
}}}
(u(p)-u(q))^2\lesssim\int_K\int_K\frac{(u(x)-u(y))^2}{|x-y|^{\alpha+\beta}}\nu(\md x)\nu(\md y).$$
\end{mythm}

\begin{proof}
First fix $n\ge1,w=w_1\ldots w_n\in W_n$, consider
$$
{\sum_{\mbox{\tiny
$
\begin{subarray}{c}
p,q\in V_w\\
|p-q|=2^{-1}\cdot3^{-n}
\end{subarray}
$
}}}
(u(p)-u(q))^2.$$
For all $x\in K_w$, we have
$$(u(p)-u(q))^2\le2(u(p)-u(x))^2+2(u(x)-u(q))^2.$$
Integrating with respect to $x\in K_w$ and dividing by $\nu(K_w)$, we have
$$(u(p)-u(q))^2\le\frac{2}{\nu(K_w)}\int_{K_w}(u(p)-u(x))^2\nu(\md x)+\frac{2}{\nu(K_w)}\int_{K_w}(u(x)-u(q))^2\nu(\md x),$$
hence
$$
{\sum_{\mbox{\tiny
$
\begin{subarray}{c}
p,q\in V_w\\
|p-q|=2^{-1}\cdot3^{-n}
\end{subarray}
$
}}}
(u(p)-u(q))^2\le2\cdot2\cdot2\sum_{p\in V_w}\frac{1}{\nu(K_w)}\int_{K_w}(u(p)-u(x))^2\nu(\md x).$$
Consider $(u(p)-u(x))^2$, $p\in V_w$, $x\in K_w$. There exists $w_{n+1}\in\myset{0,\ldots,7}$ such that $p=f_{w_1}\circ\ldots\circ f_{w_n}(p_{w_{n+1}})$. Let $k,l\ge1$ be integers to be determined, let
$$w^{(i)}=w_1\ldots w_nw_{n+1}\ldots w_{n+1}$$
with $ki$ terms of $w_{n+1}$, $i=0,\ldots,l$. For all $x^{(i)}\in K_{w^{(i)}}$, $i=0,\ldots,l$, we have
$$
\begin{aligned}
(u(p)-u(x^{(0)}))^2&\le2(u(p)-u(x^{(l)}))^2+2(u(x^{(0)})-u(x^{(l)}))^2\\
&\le2(u(p)-u(x^{(l)}))^2+2\left[2(u(x^{(0)})-u(x^{(1)}))^2+2(u(x^{(1)})-u(x^{(l)}))^2\right]\\
&=2(u(p)-u(x^{(l)}))^2+2^2(u(x^{(0)})-u(x^{(1)}))^2+2^2(u(x^{(1)})-u(x^{(l)}))^2\\
&\le\ldots\le2(u(p)-u(x^{(l)}))^2+2^2\sum_{i=0}^{l-1}2^i(u(x^{(i)})-u(x^{(i+1)}))^2.
\end{aligned}
$$
Integrating with respect to $x^{(0)}\in K_{w^{(0)}}$, \ldots, $x^{(l)}\in K_{w^{(l)}}$ and dividing by $\nu(K_{w^{(0)}})$, \ldots, $\nu(K_{w^{(l)}})$, we have
$$
\begin{aligned}
&\frac{1}{\nu(K_{w^{(0)}})}\int_{K_{w^{(0)}}}(u(p)-u(x^{(0)}))^2\nu(\md x^{(0)})\\
&\le\frac{2}{\nu(K_{w^{(l)}})}\int_{K_{w^{(l)}}}(u(p)-u(x^{(l)}))^2\nu(\md x^{(l)})\\
&+2^2\sum_{i=0}^{l-1}\frac{2^i}{\nu(K_{w^{(i)}})\nu(K_{w^{(i+1)}})}\int_{K_{w^{(i)}}}\int_{K_{w^{(i+1)}}}(u(x^{(i)})-u(x^{(i+1)}))^2\nu(\md x^{(i)})\nu(\md x^{(i+1)}).
\end{aligned}
$$
Now let us use $\nu(K_{w^{(i)}})=(1/8)^{n+ki}=3^{-\alpha(n+ki)}$. For the first term, by Lemma \ref{lem_holder}, we have
$$
\begin{aligned}
\frac{1}{\nu(K_{w^{(l)}})}\int_{K_{w^{(l)}}}(u(p)-u(x^{(l)}))^2\nu(\md x^{(l)})&\le \frac{cE(u)}{\nu(K_{w^{(l)}})}\int_{K_{w^{(l)}}}|p-x^{(l)}|^{\beta-\alpha}\nu(\md x^{(l)})\\
&\le{2}^{(\beta-\alpha)/2}cE(u){3}^{-(\beta-\alpha)(n+kl)}.
\end{aligned}
$$
For the second term, for all $x^{(i)}\in K_{w^{(i)}},x^{(i+1)}\in K_{w^{(i+1)}}$, we have
$$|x^{(i)}-x^{(i+1)}|\le\sqrt{2}\cdot3^{-(n+ki)},$$
hence
$$
\begin{aligned}
&\sum_{i=0}^{l-1}\frac{2^i}{\nu(K_{w^{(i)}})\nu(K_{w^{(i+1)}})}\int_{K_{w^{(i)}}}\int_{K_{w^{(i+1)}}}(u(x^{(i)})-u(x^{(i+1)}))^2\nu(\md x^{(i)})\nu(\md x^{(i+1)})\\
&\le\sum_{i=0}^{l-1}{2^{i}\cdot3^{\alpha k+2\alpha(n+ki)}}\int\limits_{K_{w^{(i)}}}\int\limits_{|x^{(i+1)}-x^{(i)}|\le\sqrt{2}\cdot3^{-(n+ki)}}(u(x^{(i)})-u(x^{(i+1)}))^2\nu(\md x^{(i)})\nu(\md x^{(i+1)}),
\end{aligned}
$$
and
$$
\begin{aligned}
&\frac{1}{\nu(K_w)}\int_{K_w}(u(p)-u(x))^2\nu(\md x)=\frac{1}{\nu(K_{w^{(0)}})}\int_{K_{w^{(0)}}}(u(p)-u(x^{(0)}))^2\nu(\md x^{(0)})\\
&\le 2\cdot{2}^{(\beta-\alpha)/2}cE(u)3^{-(\beta-\alpha)(n+kl)}\\
&+4\sum_{i=0}^{l-1}{2^{i}\cdot3^{\alpha k+2\alpha(n+ki)}}\int\limits_{K_{w^{(i)}}}\int\limits_{|x-y|\le\sqrt{2}\cdot3^{-(n+ki)}}(u(x)-u(y))^2\nu(\md x)\nu(\md y).
\end{aligned}
$$
Hence
$$
\begin{aligned}
&\sum_{w\in W_n}
{\sum_{\mbox{\tiny
$
\begin{subarray}{c}
p,q\in V_w\\
|p-q|=2^{-1}\cdot3^{-n}
\end{subarray}
$
}}}
(u(p)-u(q))^2\\
&\le8\sum_{w\in W_n}\sum_{p\in V_w}\frac{1}{\nu(K_w)}\int_{K_w}(u(p)-u(x))^2\nu(\md x)\\
&\le8\sum_{w\in W_n}\sum_{p\in V_w}\left(2\cdot{2}^{(\beta-\alpha)/2}cE(u)3^{-(\beta-\alpha)(n+kl)}\right.\\
&\left.+4\sum_{i=0}^{l-1}{2^{i}\cdot3^{\alpha k+2\alpha(n+ki)}}\int\limits_{K_{w^{(i)}}}\int\limits_{|x-y|\le\sqrt{2}\cdot3^{-(n+ki)}}(u(x)-u(y))^2\nu(\md x)\nu(\md y)\right).
\end{aligned}
$$
For the first term, we have
$$\sum_{w\in W_n}\sum_{p\in V_w}3^{-(\beta-\alpha)(n+kl)}=8\cdot8^n\cdot3^{-(\beta-\alpha)(n+kl)}=8\cdot3^{\alpha n-(\beta-\alpha)(n+kl)}.$$
For the second term, fix $i=0,\ldots,l-1$, different $p\in V_w$, $w\in W_n$ correspond to different $K_{w^{(i)}}$, hence
$$
\begin{aligned}
&\sum_{i=0}^{l-1}\sum_{w\in W_n}\sum_{p\in V_w}2^{i}\cdot3^{\alpha k+2\alpha(n+ki)}\int\limits_{K_{w^{(i)}}}\int\limits_{|x-y|\le\sqrt{2}\cdot3^{-(n+ki)}}(u(x)-u(y))^2\nu(\md x)\nu(\md y)\\
&\le\sum_{i=0}^{l-1}2^{i}\cdot3^{\alpha k+2\alpha(n+ki)}\int\limits_{K}\int\limits_{|x-y|\le\sqrt{2}\cdot3^{-(n+ki)}}(u(x)-u(y))^2\nu(\md x)\nu(\md y)\\
&=3^{\alpha k}\sum_{i=0}^{l-1}2^{i}\cdot3^{-(\beta-\alpha)(n+ki)}\left(3^{(\alpha+\beta)(n+ki)}\int\limits_{K}\int\limits_{|x-y|\le\sqrt{2}\cdot3^{-(n+ki)}}(u(x)-u(y))^2\nu(\md x)\nu(\md y)\right).
\end{aligned}
$$
For simplicity, denote
$$E_{n}(u)=3^{(\alpha+\beta)n}\int_{K}\int_{|x-y|\le\sqrt{2}\cdot3^{-n}}(u(x)-u(y))^2\nu(\md x)\nu(\md y).$$
We have
\begin{equation}\label{eqn_equiv1_1}
\begin{aligned}
&\sum_{w\in W_n}
{\sum_{\mbox{\tiny
$
\begin{subarray}{c}
p,q\in V_w\\
|p-q|=2^{-1}\cdot3^{-n}
\end{subarray}
$
}}}
(u(p)-u(q))^2\\
&\le128\cdot{2}^{(\beta-\alpha)/2}cE(u)3^{\alpha n-(\beta-\alpha)(n+kl)}+32\cdot3^{\alpha k}\sum_{i=0}^{l-1}2^{i}\cdot3^{-(\beta-\alpha)(n+ki)}E_{n+ki}(u).
\end{aligned}
\end{equation}
Hence
$$
\begin{aligned}
&\sum_{n=1}^\infty3^{(\beta-\alpha)n}\sum_{w\in W_n}
{\sum_{\mbox{\tiny
$
\begin{subarray}{c}
p,q\in V_w\\
|p-q|=2^{-1}\cdot3^{-n}
\end{subarray}
$
}}}
(u(p)-u(q))^2\\
&\le128\cdot{2}^{(\beta-\alpha)/2}cE(u)\sum_{n=1}^\infty3^{\beta n-(\beta-\alpha)(n+kl)}+32\cdot3^{\alpha k}\sum_{n=1}^\infty\sum_{i=0}^{l-1}2^{i}\cdot3^{-(\beta-\alpha)ki}E_{n+ki}(u).
\end{aligned}
$$
Take $l=n$, then
$$
\begin{aligned}
&\sum_{n=1}^\infty3^{(\beta-\alpha)n}\sum_{w\in W_n}
{\sum_{\mbox{\tiny
$
\begin{subarray}{c}
p,q\in V_w\\
|p-q|=2^{-1}\cdot3^{-n}
\end{subarray}
$
}}}
(u(p)-u(q))^2\\
&\le128\cdot{2}^{(\beta-\alpha)/2}cE(u)\sum_{n=1}^\infty3^{\left[\beta-(\beta-\alpha)(k+1)\right]n}+32\cdot3^{\alpha k}\sum_{n=1}^\infty\sum_{i=0}^{n-1}2^{i}\cdot3^{-(\beta-\alpha)ki}E_{n+ki}(u)\\
&=128\cdot{2}^{(\beta-\alpha)/2}cE(u)\sum_{n=1}^\infty3^{\left[\beta-(\beta-\alpha)(k+1)\right]n}+32\cdot3^{\alpha k}\sum_{i=0}^\infty2^{i}\cdot3^{-(\beta-\alpha)ki}\sum_{n=i+1}^{\infty}E_{n+ki}(u)\\
&\le128\cdot{2}^{(\beta-\alpha)/2}cE(u)\sum_{n=1}^\infty3^{\left[\beta-(\beta-\alpha)(k+1)\right]n}+32\cdot3^{\alpha k}\sum_{i=0}^\infty3^{\left[1-(\beta-\alpha)k\right]i}C_1E(u),
\end{aligned}
$$
where $C_1$ is some positive constant from Corollary \ref{cor_arbi}. Take $k\ge1$ sufficiently large such that $\beta-(\beta-\alpha)(k+1)<0$ and $1-(\beta-\alpha)k<0$, then above two series converge, hence
$$\sum_{n=1}^\infty3^{(\beta-\alpha)n}\sum_{w\in W_n}
{\sum_{\mbox{\tiny
$
\begin{subarray}{c}
p,q\in V_w\\
|p-q|=2^{-1}\cdot3^{-n}
\end{subarray}
$
}}}
(u(p)-u(q))^2\lesssim\int_K\int_K\frac{(u(x)-u(y))^2}{|x-y|^{\alpha+\beta}}\nu(\md x)\nu(\md y).$$
\end{proof}

\begin{mythm}\label{thm_equiv2}
For all $u\in C(K)$, we have
\begin{equation}\label{eqn_equiv2_1}
\int_K\int_K\frac{(u(x)-u(y))^2}{|x-y|^{\alpha+\beta}}\nu(\md x)\nu(\md y)\lesssim\sum_{n=1}^\infty3^{(\beta-\alpha)n}\sum_{w\in W_n}
{\sum_{\mbox{\tiny
$
\begin{subarray}{c}
p,q\in V_w\\
|p-q|=2^{-1}\cdot3^{-n}
\end{subarray}
$
}}}
(u(p)-u(q))^2,
\end{equation}
or equivalently for all $c\in(0,1)$
\begin{equation}\label{eqn_equiv2_2}
\begin{aligned}
&\sum_{n=2}^\infty3^{(\alpha+\beta)n}\int\limits_K\int\limits_{B(x,c3^{-n})}(u(x)-u(y))^2\nu(\md y)\nu(\md x)\\
&\lesssim\sum_{n=1}^\infty3^{(\beta-\alpha)n}\sum_{w\in W_n}
{\sum_{\mbox{\tiny
$
\begin{subarray}{c}
p,q\in V_w\\
|p-q|=2^{-1}\cdot3^{-n}
\end{subarray}
$
}}}
(u(p)-u(q))^2.
\end{aligned}
\end{equation}
\end{mythm}

\begin{proof}
Note $V_n=\cup_{w\in W_n}V_w$, it is obvious that its cardinal $\#V_n\asymp8^n=3^{\alpha n}$. Let $\nu_n$ be the measure on $V_n$ which assigns $1/\#V_n$ on each point of $V_n$, then $\nu_n$ converges weakly to $\nu$.

First, for $n\ge2,m>n$, we estimate
$$3^{(\alpha+\beta)n}\int_K\int_{B(x,c3^{-n})}(u(x)-u(y))^2\nu_m(\md y)\nu_m(\md x).$$
Note that
$$
\begin{aligned}
\int\limits_K\int\limits_{B(x,c3^{-n})}(u(x)-u(y))^2\nu_m(\md y)\nu_m(\md x)=\sum\limits_{w\in W_n}\int\limits_{K_w}\int\limits_{B(x,c3^{-n})}(u(x)-u(y))^2\nu_m(\md y)\nu_m(\md x).
\end{aligned}
$$
Fix $w\in W_n$, there exist at most nine $\tilde{w}\in W_n$ such that $K_{\tilde{w}}\cap K_w\ne\emptyset$, see Figure \ref{fig_Kw}.

\begin{figure}[ht]
\centering
\begin{tikzpicture}[scale=0.5]

\draw (0,0)--(6,0)--(6,6)--(0,6)--cycle;
\draw (2,0)--(2,6);
\draw (4,0)--(4,6);
\draw (0,2)--(6,2);
\draw (0,4)--(6,4);
\draw[very thick] (2,2)--(4,2)--(4,4)--(2,4)--cycle;

\draw (3,3) node {$K_w$};

\end{tikzpicture}
\caption{A Neighborhood of $K_w$}\label{fig_Kw}
\end{figure}
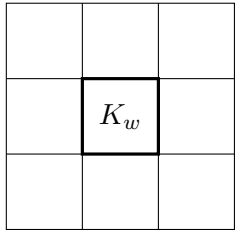

Let
$$K_w^*=
{\bigcup_{\mbox{\tiny
$
\begin{subarray}{c}
\tilde{w}\in W_n\\
K_{\tilde{w}}\cap K_w\ne\emptyset
\end{subarray}
$
}}}
K_{\tilde{w}}.$$
For all $x\in K_w$, $y\in B(x,c3^{-n})$, we have $y\in K_w^*$, hence
$$
\begin{aligned}
&\int_{K_w}\int_{B(x,c3^{-n})}(u(x)-u(y))^2\nu_m(\md y)\nu_m(\md x)\le\int_{K_w}\int_{K_w^*}(u(x)-u(y))^2\nu_m(\md y)\nu_m(\md x)\\
&=
{\sum_{\mbox{\tiny
$
\begin{subarray}{c}
\tilde{w}\in W_n\\
K_{\tilde{w}}\cap K_w\ne\emptyset
\end{subarray}
$
}}}
\int_{K_w}\int_{K_{\tilde{w}}}(u(x)-u(y))^2\nu_m(\md y)\nu_m(\md x).
\end{aligned}
$$
Note $\myset{P_w}=K_w\cap V_{n-1}$ for all $w\in W_n$. Fix $\tilde{w},w\in W_n$ with $K_{\tilde{w}}\cap K_w\ne\emptyset$. If $P_{\tilde{w}}\ne P_w$, then $|P_{\tilde{w}}-P_w|=2^{-1}\cdot3^{-(n-1)}$ or there exists a unique $z\in V_{n-1}$ such that

\begin{equation}\label{eqn_med}
\lvert P_{\tilde{w}}-z\rvert=\lvert P_w-z\rvert=2^{-1}\cdot3^{-(n-1)}.
\end{equation}
Let $z_1=P_{\tilde{w}}$, $z_3=P_w$ and
$$z_2=
\begin{cases}
P_{\tilde{w}}=P_w,&\text{if }P_{\tilde{w}}=P_w,\\
P_{\tilde{w}},&\text{if }|P_{\tilde{w}}-P_w|=2^{-1}\cdot3^{-(n-1)},\\
z,&\text{if }P_{\tilde{w}}\ne P_w\text{ and }z \text{ is given by Equation (\ref{eqn_med})}.
\end{cases}
$$
Then for all $x\in K_w$, $y\in K_{\tilde{w}}$, we have
$$
\begin{aligned}
&(u(x)-u(y))^2\\
&\le4\left[(u(y)-u(z_1))^2+(u(z_1)-u(z_2))^2+(u(z_2)-u(z_3))^2+(u(z_3)-u(x))^2\right].
\end{aligned}
$$
For $i=1,2$, we have
$$
\begin{aligned}
&\int_{K_w}\int_{K_{\tilde{w}}}(u(z_i)-u(z_{i+1}))^2\nu_m(\md y)\nu_m(\md x)=(u(z_i)-u(z_{i+1}))^2\left(\frac{\#(K_w\cap V_m)}{\#V_m}\right)^2\\
&\asymp(u(z_i)-u(z_{i+1}))^2\left(\frac{8^{m-n}}{8^m}\right)^2=3^{-2\alpha n}(u(z_i)-u(z_{i+1}))^2.
\end{aligned}
$$
Hence
$$
\begin{aligned}
&\sum_{w\in W_n}
{\sum_{\mbox{\tiny
$
\begin{subarray}{c}
\tilde{w}\in W_n\\
K_{\tilde{w}}\cap K_w\ne\emptyset
\end{subarray}
$
}}}
\int_{K_w}\int_{K_{\tilde{w}}}(u(x)-u(y))^2\nu_m(\md y)\nu_m(\md x)\\
&\lesssim3^{-\alpha n}\sum\limits_{w\in W_n}\int\limits_{K_w}(u(x)-u(P_w))^2\nu_m(\md x)+3^{-2\alpha n}\sum_{w\in W_{n-1}}
{\sum_{\mbox{\tiny
$
\begin{subarray}{c}
p,q\in V_w\\
|p-q|=2^{-1}\cdot3^{-(n-1)}
\end{subarray}
$
}}}
(u(p)-u(q))^2\\
&\asymp3^{-\alpha(m+n)}\sum\limits_{w\in W_n}\sum\limits_{x\in K_w\cap V_m}(u(x)-u(P_w))^2\\
&+3^{-2\alpha n}\sum_{w\in W_{n-1}}
{\sum_{\mbox{\tiny
$
\begin{subarray}{c}
p,q\in V_w\\
|p-q|=2^{-1}\cdot3^{-(n-1)}
\end{subarray}
$
}}}
(u(p)-u(q))^2.
\end{aligned}
$$
Let us estimate $(u(x)-u(P_w))^2$ for $x\in K_w\cap V_m$. We construct a finite sequence
$$p_1,\ldots,p_{4(m-n+1)},p_{4(m-n+1)+1}$$
such that $p_1=P_w$, $p_{4(m-n+1)+1}=x$ and for all $k=0,\ldots,m-n$, we have
$$p_{4k+1},p_{4k+2},p_{4k+3},p_{4k+4},p_{4(k+1)+1}\in V_{n+k},$$
and for all $i=1,2,3,4$, we have
$$\lvert p_{4k+i}-p_{4k+i+1}\rvert=0\text{ or }2^{-1}\cdot3^{-(n+k)}.$$
Then
$$
\begin{aligned}
\left(u(x)-u(P_w)\right)^2\lesssim\sum_{k=0}^{m-n}4^{k}&\left[(u(p_{4k+1})-u(p_{4k+2}))^2+(u(p_{4k+2})-u(p_{4k+3}))^2\right.\\
&\left.+(u(p_{4k+3})-u(p_{4k+4}))^2+(u(p_{4k+4})-u(p_{4(k+1)+1}))^2\right].
\end{aligned}
$$
For all $k=n,\ldots,m$, for all $p,q\in V_k\cap K_w$ with $|p-q|=2^{-1}\cdot 3^{-k}$, the term $(u(p)-u(q))^2$ occurs in the sum with times of the order $8^{m-k}=3^{\alpha(m-k)}$, hence
$$
\begin{aligned}
&3^{-\alpha(m+n)}\sum\limits_{w\in W_n}\sum\limits_{x\in K_w\cap V_m}(u(x)-u(P_w))^2\\
&\lesssim3^{-\alpha(m+n)}\sum_{k=n}^{m}4^{k-n}\cdot3^{\alpha(m-k)}\sum_{w\in W_k}
{\sum_{\mbox{\tiny
$
\begin{subarray}{c}
p,q\in V_w\\
|p-q|=2^{-1}\cdot3^{-k}
\end{subarray}
$
}}}
(u(p)-u(q))^2\\
&=\sum_{k=n}^{m}4^{k-n}\cdot3^{-\alpha(n+k)}\sum_{w\in W_k}
{\sum_{\mbox{\tiny
$
\begin{subarray}{c}
p,q\in V_w\\
|p-q|=2^{-1}\cdot3^{-k}
\end{subarray}
$
}}}
(u(p)-u(q))^2.
\end{aligned}
$$
Hence
$$
\begin{aligned}
&\int_K\int_{B(x,c3^{-n})}(u(x)-u(y))^2\nu_m(\md y)\nu_m(\md x)\\
&\lesssim\sum_{k=n}^{m}4^{k-n}\cdot3^{-\alpha(n+k)}\sum_{w\in W_k}
{\sum_{\mbox{\tiny
$
\begin{subarray}{c}
p,q\in V_w\\
|p-q|=2^{-1}\cdot3^{-k}
\end{subarray}
$
}}}
(u(p)-u(q))^2\\
&+3^{-2\alpha n}\sum_{w\in W_{n-1}}
{\sum_{\mbox{\tiny
$
\begin{subarray}{c}
p,q\in V_w\\
|p-q|=2^{-1}\cdot3^{-(n-1)}
\end{subarray}
$
}}}
(u(p)-u(q))^2.
\end{aligned}
$$
Letting $m\to+\infty$, we have
\begin{equation}\label{eqn_equiv2_3}
\begin{aligned}
&\int_K\int_{B(x,c3^{-n})}(u(x)-u(y))^2\nu(\md y)\nu(\md x)\\
&\lesssim\sum_{k=n}^\infty4^{k-n}\cdot3^{-\alpha(n+k)}\sum_{w\in W_k}
{\sum_{\mbox{\tiny
$
\begin{subarray}{c}
p,q\in V_w\\
|p-q|=2^{-1}\cdot3^{-k}
\end{subarray}
$
}}}
(u(p)-u(q))^2\\
&+3^{-2\alpha n}\sum_{w\in W_{n-1}}
{\sum_{\mbox{\tiny
$
\begin{subarray}{c}
p,q\in V_w\\
|p-q|=2^{-1}\cdot3^{-(n-1)}
\end{subarray}
$
}}}
(u(p)-u(q))^2.
\end{aligned}
\end{equation}
Hence
$$
\begin{aligned}
&\sum_{n=2}^\infty3^{(\alpha+\beta)n}\int_K\int_{B(x,c3^{-n})}(u(x)-u(y))^2\nu(\md y)\nu(\md x)\\
&\lesssim\sum_{n=2}^\infty\sum_{k=n}^\infty4^{k-n}\cdot3^{\beta n-\alpha k}\sum_{w\in W_k}
{\sum_{\mbox{\tiny
$
\begin{subarray}{c}
p,q\in V_w\\
|p-q|=2^{-1}\cdot3^{-k}
\end{subarray}
$
}}}
(u(p)-u(q))^2\\
&+\sum_{n=2}^\infty3^{(\beta-\alpha)n}\sum_{w\in W_{n-1}}
{\sum_{\mbox{\tiny
$
\begin{subarray}{c}
p,q\in V_w\\
|p-q|=2^{-1}\cdot3^{-(n-1)}
\end{subarray}
$
}}}
(u(p)-u(q))^2\\
&\lesssim\sum_{k=2}^\infty\sum_{n=2}^k4^{k-n}\cdot3^{\beta n-\alpha k}\sum_{w\in W_k}
{\sum_{\mbox{\tiny
$
\begin{subarray}{c}
p,q\in V_w\\
|p-q|=2^{-1}\cdot3^{-k}
\end{subarray}
$
}}}
(u(p)-u(q))^2\\
&+\sum_{n=1}^\infty3^{(\beta-\alpha)n}\sum_{w\in W_{n}}
{\sum_{\mbox{\tiny
$
\begin{subarray}{c}
p,q\in V_w\\
|p-q|=2^{-1}\cdot3^{-n}
\end{subarray}
$
}}}
(u(p)-u(q))^2\\
&\lesssim\sum_{n=1}^\infty3^{(\beta-\alpha)n}\sum_{w\in W_{n}}
{\sum_{\mbox{\tiny
$
\begin{subarray}{c}
p,q\in V_w\\
|p-q|=2^{-1}\cdot3^{-n}
\end{subarray}
$
}}}
(u(p)-u(q))^2.
\end{aligned}
$$
\end{proof}

\section{Proof of Theorem \ref{thm_bound}}\label{sec_bound}

First, we consider lower bound. We need some preparation.

\begin{myprop}\label{prop_lower}
Assume that $\beta\in(\alpha,+\infty)$. Let $f:[0,1]\to\R$ be a strictly increasing continuous function. Assume that the function $U(x,y)=f(x)$, $(x,y)\in K$ satisfies $\Ee_{\beta}(U,U)<+\infty$. Then $(\Ee_\beta,\Ff_\beta)$ is a regular Dirichlet form on $L^2(K;\nu)$.
\end{myprop}

\begin{myrmk}
Above proposition means that only \emph{one} good enough function contained in the domain can ensure that the domain is large enough.
\end{myrmk}

\begin{proof}
We only need to show that $\Ff_\beta$ is uniformly dense in $C(K)$. Then $\Ff_\beta$ is dense in $L^2(K;\nu)$. Using Fatou's lemma, we have $\Ff_\beta$ is complete under $(\Ee_\beta)_1$ metric. It is obvious that $\Ee_\beta$ has Markovian property. Hence $(\Ee_\beta,\Ff_\beta)$ is a Dirichlet form on $L^2(K;\nu)$. Moreover, $\Ff_\beta\cap C(K)=\Ff_\beta$ is trivially $(\Ee_\beta)_1$ dense in $\Ff_\beta$ and uniformly dense in $C(K)$. Hence $(\Ee_\beta,\Ff_\beta)$ on $L^2(K;\nu)$ is regular.

Indeed, by assumption, $U\in\Ff_\beta$, $\Ff_\beta\ne\emptyset$. It is obvious that $\Ff_\beta$ is a sub-algebra of $C(K)$, that is, for all $u,v\in\Ff_\beta$, $c\in\R$, we have $u+v,cu,uv\in\Ff_\beta$. We show that $\Ff_\beta$ separates points. For all distinct $(x^{(1)},y^{(1)}),(x^{(2)},y^{(2)})\in K$, we have $x^{(1)}\ne x^{(2)}$ or $y^{(1)}\ne y^{(2)}$.

If $x^{(1)}\ne x^{(2)}$, then since $f$ is strictly increasing, we have
$$U(x^{(1)},y^{(1)})=f(x^{(1)})\ne f(x^{(2)})=U(x^{(2)},y^{(2)}).$$
If $y^{(1)}\ne y^{(2)}$, then let $V(x,y)=f(y)$, $(x,y)\in K$, we have $V\in\Ff_\beta$ and
$$V(x^{(1)},y^{(1)})=f(y^{(1)})\ne f(y^{(2)})=V(x^{(2)},y^{(2)}).$$
By Stone-Weierstrass theorem, $\Ff_\beta$ is uniformly dense in $C(K)$.
\end{proof}

Now, we give lower bound.

\begin{proof}[Proof of Lower Bound]
The point is to construct an explicit function. We define $f:[0,1]\to\R$ as follows. Let $f(0)=0$ and $f(1)=1$. First, we determine the values of $f$ at $1/3$ and $2/3$. We consider the minimum of the following function
$$\vphi(x,y)=3x^2+2(x-y)^2+3(1-y)^2,x,y\in\R.$$
By elementary calculation, $\vphi$ attains minimum $6/7$ at $(x,y)=(2/7,5/7)$. Assume that we have defined $f$ on $i/3^n$, $i=0,1,\ldots,3^n$. Then, for $n+1$, for all $i=0,1,\ldots,3^{n}-1$, we define
$$
f(\frac{3i+1}{3^{n+1}})=\frac{5}{7}f(\frac{i}{3^n})+\frac{2}{7}f(\frac{i+1}{3^n}),f(\frac{3i+2}{3^{n+1}})=\frac{2}{7}f(\frac{i}{3^n})+\frac{5}{7}f(\frac{i+1}{3^n}).
$$
By induction principle, we have the definition of $f$ on all triadic points. It is obvious that $f$ is uniformly continuous on the set of all triadic points. We extend $f$ to be continuous on $[0,1]$. It is obvious that $f$ is increasing. For all $x,y\in[0,1]$ with $x<y$, there exist triadic points $i/3^n,(i+1)/3^n\in(x,y)$, then $f(x)\le f(i/3^n)<f((i+1)/3^n)\le f(y)$, hence $f$ is strictly increasing.

Let $U(x,y)=f(x)$, $(x,y)\in K$. By induction, we have
$$\sum_{w\in W_{n+1}}
{\sum_{\mbox{\tiny
$
\begin{subarray}{c}
p,q\in V_w\\
|p-q|=2^{-1}\cdot3^{-(n+1)}
\end{subarray}
$
}}}
(U(p)-U(q))^2=\frac{6}{7}\sum_{w\in W_{n}}
{\sum_{\mbox{\tiny
$
\begin{subarray}{c}
p,q\in V_w\\
|p-q|=2^{-1}\cdot3^{-n}
\end{subarray}
$
}}}
(U(p)-U(q))^2\text{ for all }n\ge1.$$
Hence
\begin{equation}\label{eqn_lower}
\sum_{w\in W_{n}}
{\sum_{\mbox{\tiny
$
\begin{subarray}{c}
p,q\in V_w\\
|p-q|=2^{-1}\cdot3^{-n}
\end{subarray}
$
}}}
(U(p)-U(q))^2=\left(\frac{6}{7}\right)^n\text{ for all }n\ge1.
\end{equation}
For all $\beta\in(\log8/\log3,\log(8\cdot7/6)/\log3)$, we have $3^{\beta-\alpha}<7/6$. By Equation (\ref{eqn_lower}), we have
$$\sum_{n=1}^\infty3^{(\beta-\alpha)n}\sum_{w\in W_n}
{\sum_{\mbox{\tiny
$
\begin{subarray}{c}
p,q\in V_w\\
|p-q|=2^{-1}\cdot3^{-n}
\end{subarray}
$
}}}
(U(p)-U(q))^2<+\infty.$$
By Lemma \ref{lem_equiv}, $\Ee_\beta(U,U)<+\infty$. By Proposition \ref{prop_lower}, $(\Ee_\beta,\Ff_\beta)$ is a regular Dirichlet form on $L^2(K;\nu)$ for all $\beta\in(\log8/\log3,\log(8\cdot7/6)/\log3)$. Hence
$$\beta_*\ge\frac{\log(8\cdot\frac{7}{6})}{\log3}.$$
\end{proof}

\begin{myrmk}
The construction of above function is similar to that given in the proof of \cite[Theorem 2.6]{Bar13}. Indeed, above function is constructed in a self-similar way. Let $f_n:[0,1]\to\R$ be given by $f_0(x)=x$, $x\in[0,1]$ and for all $n\ge0$
$$
f_{n+1}(x)=
\begin{cases}
\frac{2}{7}f_n(3x),&\text{if }0\le x\le\frac{1}{3},\\
\frac{3}{7}f_n(3x-1)+\frac{2}{7},&\text{if }\frac{1}{3}<x\le\frac{2}{3},\\
\frac{2}{7}f_n(3x-2)+\frac{5}{7},&\text{if }\frac{2}{3}<x\le1.
\end{cases}
$$
It is obvious that
$$f_n(\frac{i}{3^n})=f(\frac{i}{3^n})\text{ for all }i=0,\ldots,3^n,n\ge0,$$
and
$$\max_{x\in[0,1]}\lvert f_{n+1}(x)-f_n(x)\rvert\le\frac{3}{7}\max_{x\in[0,1]}\lvert f_{n}(x)-f_{n-1}(x)\rvert\text{ for all }n\ge1,$$
hence $f_n$ converges uniformly to $f$ on $[0,1]$. Let $g_1,g_2,g_3:\R^2\to\R^2$ be given by
$$g_1(x,y)=\left(\frac{1}{3}x,\frac{2}{7}y\right),g_2(x,y)=\left(\frac{1}{3}x+\frac{1}{3},\frac{3}{7}y+\frac{2}{7}\right),g_3(x,y)=\left(\frac{1}{3}x+\frac{2}{3},\frac{2}{7}y+\frac{5}{7}\right).$$
Then $\myset{(x,f(x)):x\in[0,1]}$ is the unique non-empty compact set $G$ in $\R^2$ satisfying 
$$G=g_1(G)\cup g_2(G)\cup g_3(G).$$
\end{myrmk}

Second, we consider upper bound. We shrink SC to another fractal. Denote $\mathcal{C}$ as Cantor ternary set in $[0,1]$. Then $[0,1]\times\mathcal{C}$ is the unique non-empty compact set $\tilde{K}$ in $\R^2$ satisfying
$$\tilde{K}=\cup_{i=0,1,2,4,5,6}f_i(\tilde{K}).$$
Let
$$\tilde{V}_0=\myset{p_0,p_1,p_2,p_4,p_5,p_6},\tilde{V}_{n+1}=\cup_{i=0,1,2,4,5,6}f_i(\tilde{V}_n)\text{ for all }n\ge0.$$
Then $\myset{\tilde{V}_n}$ is an increasing sequence of finite sets and $[0,1]\times\mathcal{C}$ is the closure of $\cup_{n=0}^\infty\tilde{V}_n$. Let $\tilde{W}_0=\myset{\emptyset}$ and
$$\tilde{W}_n=\myset{w=w_1\ldots w_n:w_i=0,1,2,4,5,6,i=1,\ldots,n}\text{ for all }n\ge1.$$
For all $w=w_1\ldots w_n\in\tilde{W}_n$, let
$$\tilde{V}_w=f_{w_1}\circ\ldots\circ f_{w_n}(\tilde{V}_0).$$

\begin{proof}[Proof of Upper Bound]
Assume that $(\Ee_\beta,\Ff_\beta)$ is a regular Dirichlet form on $L^2(K;\nu)$, then there exists $u\in\Ff_\beta$ such that $u|_{\myset{0}\times[0,1]}=0$ and $u|_{\myset{1}\times[0,1]}=1$. By Lemma \ref{lem_equiv}, we have
\begin{equation}\label{eqn_upper}
\begin{aligned}
+\infty&>\sum_{n=1}^\infty3^{(\beta-\alpha)n}\sum_{w\in W_n}
{\sum_{\mbox{\tiny
$
\begin{subarray}{c}
p,q\in V_w\\
|p-q|=2^{-1}\cdot3^{-n}
\end{subarray}
$
}}}
(u(p)-u(q))^2\\
&\ge\sum_{n=1}^\infty3^{(\beta-\alpha)n}\sum_{w\in\tilde{W}_n}
{\sum_{\mbox{\tiny
$
\begin{subarray}{c}
p,q\in\tilde{V}_w\\
|p-q|=2^{-1}\cdot3^{-n}
\end{subarray}
$
}}}
(u(p)-u(q))^2\\
&=\sum_{n=1}^\infty3^{(\beta-\alpha)n}\sum_{w\in\tilde{W}_n}
{\sum_{\mbox{\tiny
$
\begin{subarray}{c}
p,q\in\tilde{V}_w\\
|p-q|=2^{-1}\cdot3^{-n}
\end{subarray}
$
}}}
((u|_{[0,1]\times\mathcal{C}})(p)-(u|_{[0,1]\times\mathcal{C}})(q))^2\\
&\ge\sum_{n=1}^\infty3^{(\beta-\alpha)n}\sum_{w\in\tilde{W}_n}
{\sum_{\mbox{\tiny
$
\begin{subarray}{c}
p,q\in\tilde{V}_w\\
|p-q|=2^{-1}\cdot3^{-n}
\end{subarray}
$
}}}
(\tilde{u}(p)-\tilde{u}(q))^2,
\end{aligned}
\end{equation}
where $\tilde{u}$ is the function on $[0,1]\times\mathcal{C}$ that is the minimizer of 
$$\sum_{n=1}^\infty3^{(\beta-\alpha)n}\sum_{w\in\tilde{W}_n}
{\sum_{\mbox{\tiny
$
\begin{subarray}{c}
p,q\in\tilde{V}_w\\
|p-q|=2^{-1}\cdot3^{-n}
\end{subarray}
$
}}}
(\tilde{u}(p)-\tilde{u}(q))^2:\tilde{u}|_{\myset{0}\times\mathcal{C}}=0,\tilde{u}|_{\myset{1}\times\mathcal{C}}=1,\tilde{u}\in C([0,1]\times\mathcal{C}).$$
By symmetry of $[0,1]\times\mathcal{C}$, $\tilde{u}(x,y)=x,(x,y)\in [0,1]\times\mathcal{C}$. By induction, we have
$$\sum_{w\in\tilde{W}_{n+1}}
{\sum_{\mbox{\tiny
$
\begin{subarray}{c}
p,q\in\tilde{V}_w\\
|p-q|=2^{-1}\cdot3^{-(n+1)}
\end{subarray}
$
}}}
(\tilde{u}(p)-\tilde{u}(q))^2=\frac{2}{3}\sum_{w\in\tilde{W}_n}
{\sum_{\mbox{\tiny
$
\begin{subarray}{c}
p,q\in\tilde{V}_w\\
|p-q|=2^{-1}\cdot3^{-n}
\end{subarray}
$
}}}
(\tilde{u}(p)-\tilde{u}(q))^2\text{ for all }n\ge1,$$
hence
$$\sum_{w\in\tilde{W}_{n}}
{\sum_{\mbox{\tiny
$
\begin{subarray}{c}
p,q\in\tilde{V}_w\\
|p-q|=2^{-1}\cdot3^{-n}
\end{subarray}
$
}}}
(\tilde{u}(p)-\tilde{u}(q))^2=\left(\frac{2}{3}\right)^n\text{ for all }n\ge1.$$
By Equation (\ref{eqn_upper}), we have
$$\sum_{n=1}^\infty3^{(\beta-\alpha)n}\left(\frac{2}{3}\right)^n<+\infty,$$
hence, $\beta<\log(8\cdot3/2)/\log3$. Hence
$$\beta_*\le\frac{\log(8\cdot\frac{3}{2})}{\log3}.$$
\end{proof}

\section{Resistance Estimates}\label{sec_resistance}

In this section, we give resistance estimates using electrical network techniques.

We consider two sequences of finite graphs related to $V_n$ and $W_n$, respectively.

For all $n\ge1$. Let $\calV_n$ be the graph with vertex set $V_n$ and edge set given by
$$\myset{(p,q):p,q\in V_n,|p-q|=2^{-1}\cdot3^{-n}}.$$
For example, we have the figure of $\calV_2$ in Figure \ref{fig_V2}.

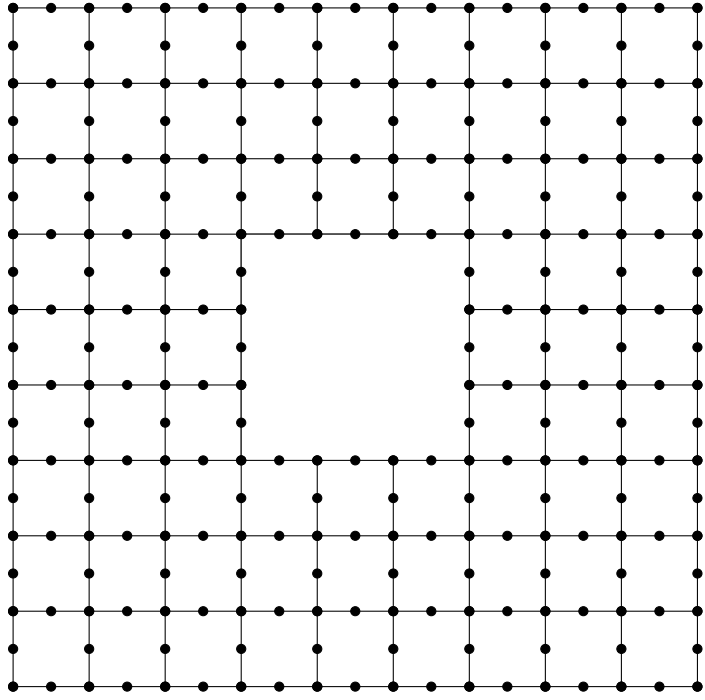
\begin{figure}[ht]
\centering
\begin{tikzpicture}

\foreach \x in {0,1,...,9}
\draw (\x,0)--(\x,9);

\foreach \y in {0,1,...,9}
\draw (0,\y)--(9,\y);

\draw[fill=white] (3,3)--(6,3)--(6,6)--(3,6)--cycle;

\foreach \x in {0,1,...,9}
\foreach \y in {0,0.5,1,...,9}
\draw[fill=black] (\x,\y) circle (0.06);

\foreach \y in {0,1,...,9}
\foreach \x in {0,0.5,1,...,9}
\draw[fill=black] (\x,\y) circle (0.06);

\draw[fill=white,draw=white] (3.25,3.25)--(5.75,3.25)--(5.75,5.75)--(3.25,5.75)--cycle;

\end{tikzpicture}
\caption{$\mathcal{V}_2$}\label{fig_V2}
\end{figure}

Let $\calW_n$ be the graph with vertex set $W_n$ and edge set given by
$$\myset{(w^{(1)},w^{(2)}):w^{(1)},w^{(2)}\in W_n,\mathrm{dim}_{\mathcal{H}}\left(K_{w^{(1)}}\cap K_{w^{(2)}}\right)=1}.$$
For example, we have the figure of $\calW_2$ in Figure \ref{fig_W2}.




\begin{figure}[ht]
\centering
\begin{tikzpicture}
\draw (0,0)--(8,0)--(8,8)--(0,8)--cycle;
\draw (2,0)--(2,8);
\draw (6,0)--(6,8);
\draw (0,2)--(8,2);
\draw (0,6)--(8,6);
\draw (3,0)--(3,2);
\draw (5,0)--(5,2);
\draw (0,3)--(2,3);
\draw (0,5)--(2,5);
\draw (6,3)--(8,3);
\draw (6,5)--(8,5);
\draw (3,6)--(3,8);
\draw (5,6)--(5,8);
\draw (2,1)--(3,1);
\draw (5,1)--(6,1);
\draw (1,2)--(1,3);
\draw (7,2)--(7,3);
\draw (1,5)--(1,6);
\draw (7,5)--(7,6);
\draw (2,7)--(3,7);
\draw (5,7)--(6,7);

\draw[fill=black] (0,0) circle (0.06);
\draw[fill=black] (1,0) circle (0.06);
\draw[fill=black] (2,0) circle (0.06);
\draw[fill=black] (3,0) circle (0.06);
\draw[fill=black] (4,0) circle (0.06);
\draw[fill=black] (5,0) circle (0.06);
\draw[fill=black] (6,0) circle (0.06);
\draw[fill=black] (7,0) circle (0.06);
\draw[fill=black] (8,0) circle (0.06);

\draw[fill=black] (0,1) circle (0.06);
\draw[fill=black] (2,1) circle (0.06);
\draw[fill=black] (3,1) circle (0.06);
\draw[fill=black] (5,1) circle (0.06);
\draw[fill=black] (6,1) circle (0.06);
\draw[fill=black] (8,1) circle (0.06);

\draw[fill=black] (0,2) circle (0.06);
\draw[fill=black] (1,2) circle (0.06);
\draw[fill=black] (2,2) circle (0.06);
\draw[fill=black] (3,2) circle (0.06);
\draw[fill=black] (4,2) circle (0.06);
\draw[fill=black] (5,2) circle (0.06);
\draw[fill=black] (6,2) circle (0.06);
\draw[fill=black] (7,2) circle (0.06);
\draw[fill=black] (8,2) circle (0.06);

\draw[fill=black] (0,3) circle (0.06);
\draw[fill=black] (1,3) circle (0.06);
\draw[fill=black] (2,3) circle (0.06);
\draw[fill=black] (6,3) circle (0.06);
\draw[fill=black] (7,3) circle (0.06);
\draw[fill=black] (8,3) circle (0.06);

\draw[fill=black] (0,4) circle (0.06);
\draw[fill=black] (2,4) circle (0.06);
\draw[fill=black] (6,4) circle (0.06);
\draw[fill=black] (8,4) circle (0.06);

\draw[fill=black] (0,5) circle (0.06);
\draw[fill=black] (1,5) circle (0.06);
\draw[fill=black] (2,5) circle (0.06);
\draw[fill=black] (6,5) circle (0.06);
\draw[fill=black] (7,5) circle (0.06);
\draw[fill=black] (8,5) circle (0.06);

\draw[fill=black] (0,6) circle (0.06);
\draw[fill=black] (1,6) circle (0.06);
\draw[fill=black] (2,6) circle (0.06);
\draw[fill=black] (3,6) circle (0.06);
\draw[fill=black] (4,6) circle (0.06);
\draw[fill=black] (5,6) circle (0.06);
\draw[fill=black] (6,6) circle (0.06);
\draw[fill=black] (7,6) circle (0.06);
\draw[fill=black] (8,6) circle (0.06);

\draw[fill=black] (0,7) circle (0.06);
\draw[fill=black] (2,7) circle (0.06);
\draw[fill=black] (3,7) circle (0.06);
\draw[fill=black] (5,7) circle (0.06);
\draw[fill=black] (6,7) circle (0.06);
\draw[fill=black] (8,7) circle (0.06);

\draw[fill=black] (0,8) circle (0.06);
\draw[fill=black] (1,8) circle (0.06);
\draw[fill=black] (2,8) circle (0.06);
\draw[fill=black] (3,8) circle (0.06);
\draw[fill=black] (4,8) circle (0.06);
\draw[fill=black] (5,8) circle (0.06);
\draw[fill=black] (6,8) circle (0.06);
\draw[fill=black] (7,8) circle (0.06);
\draw[fill=black] (8,8) circle (0.06);
\end{tikzpicture}
\caption{$\mathcal{W}_2$}\label{fig_W2}
\end{figure}
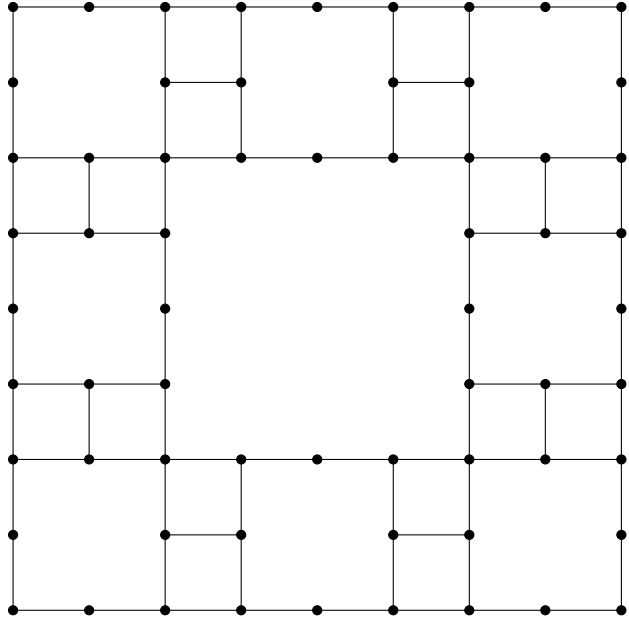


On $\calV_n$, the energy
$$
{\sum_{\mbox{\tiny
$
\begin{subarray}{c}
p,q\in V_n\\
|p-q|=2^{-1}\cdot3^{-n}
\end{subarray}
$
}}}
(u(p)-u(q))^2,u\in l(V_n),$$
is related to a weighted graph with the conductances of all edges equal to $1$. While the energy
$$\sum_{w\in W_n}
{\sum_{\mbox{\tiny
$
\begin{subarray}{c}
p,q\in V_w\\
|p-q|=2^{-1}\cdot3^{-n}
\end{subarray}
$
}}}
(u(p)-u(q))^2,u\in l(V_n),$$
is related to a weighted graph with the conductances of some edges equal to $1$ and the conductances of other edges equal to $2$, since the term $(u(p)-u(q))^2$ is added either once or twice.

Since
$$
\begin{aligned}
{\sum_{\mbox{\tiny
$
\begin{subarray}{c}
p,q\in V_n\\
|p-q|=2^{-1}\cdot3^{-n}
\end{subarray}
$
}}}
(u(p)-u(q))^2&\le\sum_{w\in W_n}
{\sum_{\mbox{\tiny
$
\begin{subarray}{c}
p,q\in V_w\\
|p-q|=2^{-1}\cdot3^{-n}
\end{subarray}
$
}}}
(u(p)-u(q))^2\\
&\le 2
{\sum_{\mbox{\tiny
$
\begin{subarray}{c}
p,q\in V_n\\
|p-q|=2^{-1}\cdot3^{-n}
\end{subarray}
$
}}}
(u(p)-u(q))^2,
\end{aligned}
$$
we use
$$D_n(u,u):=\sum_{w\in W_n}
{\sum_{\mbox{\tiny
$
\begin{subarray}{c}
p,q\in V_w\\
|p-q|=2^{-1}\cdot3^{-n}
\end{subarray}
$
}}}
(u(p)-u(q))^2,u\in l(V_n),$$
as the energy on $\calV_n$. Assume that $A,B$ are two disjoint subsets of $V_n$. Let
$$R_n(A,B)=\inf\myset{D_n(u,u):u|_A=0,u|_B=1,u\in l(V_n)}^{-1}.$$
Denote
$$R_n^V=R_n(V_n\cap\myset{0}\times[0,1],V_n\cap\myset{1}\times[0,1]),$$
$$R_n(x,y)=R_n(\myset{x},\myset{y}),x,y\in V_n.$$
It is obvious that $R_n$ is a metric on $V_n$, hence
$$R_n(x,y)\le R_n(x,z)+R_n(z,y)\text{ for all }x,y,z\in V_n.$$

On $\calW_n$, the energy
$$\frakD_n(u,u):=\sum_{w^{(1)}\sim_nw^{(2)}}(u(w^{(1)})-u(w^{(2)}))^2,u\in l(W_n),$$
is related to a weighted graph with the conductances of all edges equal to $1$. Assume that $A,B$ are two disjoint subsets of $W_n$. Let
$$\frakR_n(A,B)=\inf\myset{\frakD_n(u,u):u|_A=0,u|_B=1,u\in l(W_n)}^{-1}.$$
Denote
$$\frakR_n(w^{(1)},w^{(2)})=\frakR_n(\myset{w^{(1)}},\myset{w^{(2)}}),w^{(1)},w^{(2)}\in W_n.$$
It is obvious that $\frakR_n$ is a metric on $W_n$, hence
$$\frakR_n(w^{(1)},w^{(2)})\le\frakR_n(w^{(1)},w^{(3)})+\frakR_n(w^{(3)},w^{(2)})\text{ for all }w^{(1)},w^{(2)},w^{(3)}\in W_n.$$

The main result of this section is as follows.

\begin{mythm}\label{thm_resist}
There exists some positive constant $\rho\in\left[7/6,3/2\right]$ such that for all $n\ge1$
$$R_n^V\asymp\rho^n,$$
$$R_n(p_0,p_1)=\ldots=R_n(p_6,p_7)=R_n(p_7,p_0)\asymp\rho^n,$$
$$\frakR_n(0^n,1^n)=\ldots=\frakR_n(6^n,7^n)=\frakR_n(7^n,0^n)\asymp\rho^n.$$
\end{mythm}

\begin{myrmk}
By triangle inequality, for all $i,j=0,\ldots,7,n\ge1$
$$R_{n}(p_i,p_j)\lesssim\rho^n,$$
$$\frakR_{n}(i^n,j^n)\lesssim\rho^n.$$
\end{myrmk}

We have a direct corollary as follows.

\begin{mycor}\label{cor_resist_upper}
For all $n\ge1,p,q\in V_n,w^{(1)},w^{(2)}\in W_n$
$$R_n(p,q)\lesssim\rho^n,$$
$$\frakR_n(w^{(1)},w^{(2)})\lesssim\rho^n.$$
\end{mycor}

\begin{proof}
We only need to show that $\frakR_n(w,0^n)\lesssim\rho^n$ for all $w\in W_n,n\ge1$. Then for all $w^{(1)},w^{(2)}\in W_n$
$$\frakR_n(w^{(1)},w^{(2)})\le\frakR_n(w^{(1)},0^n)+\frakR_n(w^{(2)},0^n)\lesssim\rho^n.$$
Similarly, we have the proof of $R_n(p,q)\lesssim\rho^n$ for all $p,q\in V_n,n\ge1$.

Indeed, for all $n\ge1,w=w_1\ldots w_n\in W_n$, we construct a finite sequence as follows.
$$
\begin{aligned}
w^{(1)}&=w_1\ldots w_{n-2}w_{n-1}w_n=w,\\
w^{(2)}&=w_1\ldots w_{n-2}w_{n-1}w_{n-1},\\
w^{(3)}&=w_1\ldots w_{n-2}w_{n-2}w_{n-2},\\
&\ldots\\
w^{(n)}&=w_1\ldots w_1w_1w_1,\\
w^{(n+1)}&=0\ldots 000=0^n.
\end{aligned}
$$
For all $i=1,\ldots,n-1$, by cutting technique
$$
\begin{aligned}
&\frakR_n(w^{(i)},w^{(i+1)})=\frakR_n(w_1\ldots w_{n-i}w_{n-i+1}\ldots w_{n-i+1},w_1\ldots w_{n-i}w_{n-i}\ldots w_{n-i})\\
&\le\frakR_i(w_{n-i+1}\ldots w_{n-i+1},w_{n-i}\ldots w_{n-i})=\frakR_i(w_{n-i+1}^i,w_{n-i}^i)\lesssim\rho^i.
\end{aligned}
$$
Since $\frakR_n(w^{(n)},w^{(n+1)})=\frakR_n(w_1^n,0^n)\lesssim\rho^n$, we have
$$\frakR_n(w,0^n)=\frakR_n(w^{(1)},w^{(n+1)})\le\sum_{i=1}^n\frakR_n(w^{(i)},w^{(i+1)})\lesssim\sum_{i=1}^n\rho^i\lesssim\rho^n.$$
\end{proof}

We need the following results for preparation.

First, we have resistance estimates for some symmetric cases.

\begin{mythm}\label{thm_resist1}
There exists some positive constant $\rho\in[7/6,3/2]$ such that for all $n\ge1$
$$R_n^V\asymp\rho^n,$$
$$R_n(p_1,p_5)=R_n(p_3,p_7)\asymp\rho^n,$$
$$R_n(p_0,p_4)=R_n(p_2,p_6)\asymp\rho^n.$$
\end{mythm}

\begin{proof}
The proof is similar to \cite[Theorem 5.1]{BB90} and \cite[Theorem 6.1]{McG02} where flow technique and potential technique are used. We need discrete version instead of continuous version.

Hence there exists some positive constant $C$ such that
$$\frac{1}{C}x_nx_m\le x_{n+m}\le Cx_nx_m\text{ for all }n,m\ge1,$$
where $x$ is any of above resistances. Since above resistances share the same complexity, there exists \emph{one} positive constant $\rho$ such that they are equivalent to $\rho^n$ for all $n\ge1$.

By shorting and cutting technique, we have $\rho\in[7/6,3/2]$, see \cite[Equation (2.6)]{Bar13} or \cite[Remarks 5.4]{BB99a}.
\end{proof}

Second, by symmetry and shorting technique, we have the following relations.

\begin{myprop}\label{prop_resist2}
For all $n\ge1$
$$R_n(p_0,p_1)\le\frakR_n(0^n,1^n),$$
$$R_n^V\le R_n(p_1,p_5)=R_n(p_3,p_7)\le\frakR_n(1^n,5^n)=\frakR_n(3^n,7^n),$$
$$R_n^V\le R_n(p_0,p_4)=R_n(p_2,p_6)\le\frakR_n(0^n,4^n)=\frakR_n(2^n,6^n).$$
\end{myprop}

Third, we have the following relations.

\begin{myprop}\label{prop_resist3}
For all $n\ge1$
$$\frakR_n(0^n,1^n)\lesssim R_n(p_0,p_1),$$
$$\frakR_n(1^n,5^n)=\frakR_n(3^n,7^n)\lesssim R_n(p_1,p_5)=R_n(p_3,p_7),$$
$$\frakR_n(0^n,4^n)=\frakR_n(2^n,6^n)\lesssim R_n(p_0,p_4)=R_n(p_2,p_6).$$
\end{myprop}

\begin{proof}
The idea is to use electrical network transformations to \emph{increase} resistances to transform weighted graph $\calW_n$ to weighted graph $\calV_{n-1}$.

First, we do the transformation in Figure \ref{fig_trans1} where the resistances of the resistors in the new network only depend on the shape of the networks in Figure \ref{fig_trans1} such that we obtain the weighted graph in Figure \ref{fig_trans2} where the resistances between any two points are larger than those in the weighted graph $\calW_n$. For $\frakR_n(i^n,j^n)$, we have the equivalent weighted graph in Figure \ref{fig_trans3}.




\begin{figure}[ht]
\centering
\begin{tikzpicture}
\draw (0,0)--(2,0)--(2,1)--(0,1)--cycle;
\draw (1,0)--(1,1);

\draw[fill=black] (0,0) circle (0.06);
\draw[fill=black] (1,0) circle (0.06);
\draw[fill=black] (2,0) circle (0.06);
\draw[fill=black] (0,1) circle (0.06);
\draw[fill=black] (1,1) circle (0.06);
\draw[fill=black] (2,1) circle (0.06);

\draw (3,0)--(3,1);
\draw (4,0)--(4,1);
\draw (5,0)--(5,1);
\draw (3,0.5)--(5,0.5);

\draw[fill=black] (3,0) circle (0.06);
\draw[fill=black] (4,0) circle (0.06);
\draw[fill=black] (5,0) circle (0.06);
\draw[fill=black] (3,1) circle (0.06);
\draw[fill=black] (4,1) circle (0.06);
\draw[fill=black] (5,1) circle (0.06);

\draw (2.5,0.5) node {$\Rightarrow$};

\end{tikzpicture}
\caption{First Transformation}\label{fig_trans1}
\end{figure}




\begin{figure}[ht]
\centering
\begin{tikzpicture}

\draw[fill=black] (0,0) circle (0.06);
\draw[fill=black] (0.5,0) circle (0.06);
\draw[fill=black] (1,0) circle (0.06);
\draw[fill=black] (1.5,0) circle (0.06);
\draw[fill=black] (2,0) circle (0.06);
\draw[fill=black] (2.5,0) circle (0.06);
\draw[fill=black] (3,0) circle (0.06);
\draw[fill=black] (3.5,0) circle (0.06);
\draw[fill=black] (4,0) circle (0.06);
\draw[fill=black] (4.5,0) circle (0.06);
\draw[fill=black] (5,0) circle (0.06);
\draw[fill=black] (5.5,0) circle (0.06);
\draw[fill=black] (6,0) circle (0.06);
\draw[fill=black] (6.5,0) circle (0.06);
\draw[fill=black] (7,0) circle (0.06);
\draw[fill=black] (7.5,0) circle (0.06);
\draw[fill=black] (8,0) circle (0.06);
\draw[fill=black] (8.5,0) circle (0.06);

\draw[fill=black] (0,0.5) circle (0.06);
\draw[fill=black] (1,0.5) circle (0.06);
\draw[fill=black] (1.5,0.5) circle (0.06);
\draw[fill=black] (2.5,0.5) circle (0.06);
\draw[fill=black] (3,0.5) circle (0.06);
\draw[fill=black] (4,0.5) circle (0.06);
\draw[fill=black] (4.5,0.5) circle (0.06);
\draw[fill=black] (5.5,0.5) circle (0.06);
\draw[fill=black] (6,0.5) circle (0.06);
\draw[fill=black] (7,0.5) circle (0.06);
\draw[fill=black] (7.5,0.5) circle (0.06);
\draw[fill=black] (8.5,0.5) circle (0.06);

\draw[fill=black] (0,1) circle (0.06);
\draw[fill=black] (0.5,1) circle (0.06);
\draw[fill=black] (1,1) circle (0.06);
\draw[fill=black] (1.5,1) circle (0.06);
\draw[fill=black] (2,1) circle (0.06);
\draw[fill=black] (2.5,1) circle (0.06);
\draw[fill=black] (3,1) circle (0.06);
\draw[fill=black] (3.5,1) circle (0.06);
\draw[fill=black] (4,1) circle (0.06);
\draw[fill=black] (4.5,1) circle (0.06);
\draw[fill=black] (5,1) circle (0.06);
\draw[fill=black] (5.5,1) circle (0.06);
\draw[fill=black] (6,1) circle (0.06);
\draw[fill=black] (6.5,1) circle (0.06);
\draw[fill=black] (7,1) circle (0.06);
\draw[fill=black] (7.5,1) circle (0.06);
\draw[fill=black] (8,1) circle (0.06);
\draw[fill=black] (8.5,1) circle (0.06);

\draw[fill=black] (0,1.5) circle (0.06);
\draw[fill=black] (0.5,1.5) circle (0.06);
\draw[fill=black] (1,1.5) circle (0.06);
\draw[fill=black] (3,1.5) circle (0.06);
\draw[fill=black] (3.5,1.5) circle (0.06);
\draw[fill=black] (4,1.5) circle (0.06);
\draw[fill=black] (4.5,1.5) circle (0.06);
\draw[fill=black] (5,1.5) circle (0.06);
\draw[fill=black] (5.5,1.5) circle (0.06);
\draw[fill=black] (7.5,1.5) circle (0.06);
\draw[fill=black] (8,1.5) circle (0.06);
\draw[fill=black] (8.5,1.5) circle (0.06);

\draw[fill=black] (0,2) circle (0.06);
\draw[fill=black] (1,2) circle (0.06);
\draw[fill=black] (3,2) circle (0.06);
\draw[fill=black] (4,2) circle (0.06);
\draw[fill=black] (4.5,2) circle (0.06);
\draw[fill=black] (5.5,2) circle (0.06);
\draw[fill=black] (7.5,2) circle (0.06);
\draw[fill=black] (8.5,2) circle (0.06);

\draw[fill=black] (0,2.5) circle (0.06);
\draw[fill=black] (0.5,2.5) circle (0.06);
\draw[fill=black] (1,2.5) circle (0.06);
\draw[fill=black] (3,2.5) circle (0.06);
\draw[fill=black] (3.5,2.5) circle (0.06);
\draw[fill=black] (4,2.5) circle (0.06);
\draw[fill=black] (4.5,2.5) circle (0.06);
\draw[fill=black] (5,2.5) circle (0.06);
\draw[fill=black] (5.5,2.5) circle (0.06);
\draw[fill=black] (7.5,2.5) circle (0.06);
\draw[fill=black] (8,2.5) circle (0.06);
\draw[fill=black] (8.5,2.5) circle (0.06);

\draw[fill=black] (0,3) circle (0.06);
\draw[fill=black] (0.5,3) circle (0.06);
\draw[fill=black] (1,3) circle (0.06);
\draw[fill=black] (1.5,3) circle (0.06);
\draw[fill=black] (2,3) circle (0.06);
\draw[fill=black] (2.5,3) circle (0.06);
\draw[fill=black] (3,3) circle (0.06);
\draw[fill=black] (3.5,3) circle (0.06);
\draw[fill=black] (4,3) circle (0.06);
\draw[fill=black] (4.5,3) circle (0.06);
\draw[fill=black] (5,3) circle (0.06);
\draw[fill=black] (5.5,3) circle (0.06);
\draw[fill=black] (6,3) circle (0.06);
\draw[fill=black] (6.5,3) circle (0.06);
\draw[fill=black] (7,3) circle (0.06);
\draw[fill=black] (7.5,3) circle (0.06);
\draw[fill=black] (8,3) circle (0.06);
\draw[fill=black] (8.5,3) circle (0.06);

\draw[fill=black] (0,3.5) circle (0.06);
\draw[fill=black] (1,3.5) circle (0.06);
\draw[fill=black] (1.5,3.5) circle (0.06);
\draw[fill=black] (2.5,3.5) circle (0.06);
\draw[fill=black] (3,3.5) circle (0.06);
\draw[fill=black] (4,3.5) circle (0.06);
\draw[fill=black] (4.5,3.5) circle (0.06);
\draw[fill=black] (5.5,3.5) circle (0.06);
\draw[fill=black] (6,3.5) circle (0.06);
\draw[fill=black] (7,3.5) circle (0.06);
\draw[fill=black] (7.5,3.5) circle (0.06);
\draw[fill=black] (8.5,3.5) circle (0.06);

\draw[fill=black] (0,4) circle (0.06);
\draw[fill=black] (0.5,4) circle (0.06);
\draw[fill=black] (1,4) circle (0.06);
\draw[fill=black] (1.5,4) circle (0.06);
\draw[fill=black] (2,4) circle (0.06);
\draw[fill=black] (2.5,4) circle (0.06);
\draw[fill=black] (3,4) circle (0.06);
\draw[fill=black] (3.5,4) circle (0.06);
\draw[fill=black] (4,4) circle (0.06);
\draw[fill=black] (4.5,4) circle (0.06);
\draw[fill=black] (5,4) circle (0.06);
\draw[fill=black] (5.5,4) circle (0.06);
\draw[fill=black] (6,4) circle (0.06);
\draw[fill=black] (6.5,4) circle (0.06);
\draw[fill=black] (7,4) circle (0.06);
\draw[fill=black] (7.5,4) circle (0.06);
\draw[fill=black] (8,4) circle (0.06);
\draw[fill=black] (8.5,4) circle (0.06);

\draw[fill=black] (0,4.5) circle (0.06);
\draw[fill=black] (0.5,4.5) circle (0.06);
\draw[fill=black] (1,4.5) circle (0.06);
\draw[fill=black] (1.5,4.5) circle (0.06);
\draw[fill=black] (2,4.5) circle (0.06);
\draw[fill=black] (2.5,4.5) circle (0.06);
\draw[fill=black] (3,4.5) circle (0.06);
\draw[fill=black] (3.5,4.5) circle (0.06);
\draw[fill=black] (4,4.5) circle (0.06);

\draw[fill=black] (0,5) circle (0.06);
\draw[fill=black] (1,5) circle (0.06);
\draw[fill=black] (1.5,5) circle (0.06);
\draw[fill=black] (2.5,5) circle (0.06);
\draw[fill=black] (3,5) circle (0.06);
\draw[fill=black] (4,5) circle (0.06);

\draw[fill=black] (0,5.5) circle (0.06);
\draw[fill=black] (0.5,5.5) circle (0.06);
\draw[fill=black] (1,5.5) circle (0.06);
\draw[fill=black] (1.5,5.5) circle (0.06);
\draw[fill=black] (2,5.5) circle (0.06);
\draw[fill=black] (2.5,5.5) circle (0.06);
\draw[fill=black] (3,5.5) circle (0.06);
\draw[fill=black] (3.5,5.5) circle (0.06);
\draw[fill=black] (4,5.5) circle (0.06);

\draw[fill=black] (0,6) circle (0.06);
\draw[fill=black] (0.5,6) circle (0.06);
\draw[fill=black] (1,6) circle (0.06);
\draw[fill=black] (3,6) circle (0.06);
\draw[fill=black] (3.5,6) circle (0.06);
\draw[fill=black] (4,6) circle (0.06);

\draw[fill=black] (0,6.5) circle (0.06);
\draw[fill=black] (1,6.5) circle (0.06);
\draw[fill=black] (3,6.5) circle (0.06);
\draw[fill=black] (4,6.5) circle (0.06);

\draw[fill=black] (0,7) circle (0.06);
\draw[fill=black] (0.5,7) circle (0.06);
\draw[fill=black] (1,7) circle (0.06);
\draw[fill=black] (3,7) circle (0.06);
\draw[fill=black] (3.5,7) circle (0.06);
\draw[fill=black] (4,7) circle (0.06);

\draw[fill=black] (0,7.5) circle (0.06);
\draw[fill=black] (0.5,7.5) circle (0.06);
\draw[fill=black] (1,7.5) circle (0.06);
\draw[fill=black] (1.5,7.5) circle (0.06);
\draw[fill=black] (2,7.5) circle (0.06);
\draw[fill=black] (2.5,7.5) circle (0.06);
\draw[fill=black] (3,7.5) circle (0.06);
\draw[fill=black] (3.5,7.5) circle (0.06);
\draw[fill=black] (4,7.5) circle (0.06);

\draw[fill=black] (0,8) circle (0.06);
\draw[fill=black] (1,8) circle (0.06);
\draw[fill=black] (1.5,8) circle (0.06);
\draw[fill=black] (2.5,8) circle (0.06);
\draw[fill=black] (3,8) circle (0.06);
\draw[fill=black] (4,8) circle (0.06);

\draw[fill=black] (0,8.5) circle (0.06);
\draw[fill=black] (0.5,8.5) circle (0.06);
\draw[fill=black] (1,8.5) circle (0.06);
\draw[fill=black] (1.5,8.5) circle (0.06);
\draw[fill=black] (2,8.5) circle (0.06);
\draw[fill=black] (2.5,8.5) circle (0.06);
\draw[fill=black] (3,8.5) circle (0.06);
\draw[fill=black] (3.5,8.5) circle (0.06);
\draw[fill=black] (4,8.5) circle (0.06);

\draw (0,0)--(8.5,0);
\draw (0,0)--(0,8.5);

\draw (1,1)--(1.5,1);
\draw (1,1)--(1,1.5);
\draw (1.25,0)--(1.25,1);
\draw (0,1.25)--(1,1.25);
\draw (1,0.5)--(1.5,0.5);
\draw (0.5,1)--(0.5,1.5);

\draw (1.5,1)--(3,1);
\draw (2.75,1)--(2.75,0);
\draw (2.5,0.5)--(3,0.5);

\draw (3,1)--(3,1.5);
\draw (4,1)--(4,1.5);
\draw (3,1.25)--(4,1.25);
\draw (3.5,1)--(3.5,1.5);
\draw (4,1)--(4.5,1);
\draw (4,0.5)--(4.5,0.5);
\draw (4.25,0)--(4.25,1);

\draw (4.5,1)--(4.5,1.5);
\draw (5,1)--(5,1.5);
\draw (5.5,1)--(5.5,1.5);
\draw (4.5,1.25)--(5.5,1.25);
\draw (5.5,1)--(6,1);
\draw (5.5,0.5)--(6,0.5);
\draw (5.75,0)--(5.75,1);
\draw (6,1)--(7,1);
\draw (7,1)--(7.5,1);
\draw (7,0.5)--(7.5,0.5);
\draw (7.25,0)--(7.25,1);
\draw (7.5,1)--(7.5,1.5);
\draw (8,1)--(8,1.5);
\draw (8.5,1)--(8.5,1.5);
\draw (7.5,1.25)--(8.5,1.25);

\draw (8.5,0)--(9,0);
\draw (8.5,0.5)--(9,0.5);
\draw (8.5,1)--(9,1);
\draw (8.75,0)--(8.75,1);

\draw (1,1.5)--(1,3);
\draw (0.5,2.5)--(0.5,3);
\draw (0,2.75)--(1,2.75);

\draw (1,3)--(3,3);
\draw (3,3)--(3,1);

\draw (3.5,2.5)--(3.5,3);
\draw (4,2.5)--(4,3);
\draw (3,2.75)--(4,2.75);
\draw (4,1.5)--(4.5,1.5);
\draw (4,2)--(4.5,2);
\draw (4,2.5)--(4.5,2.5);
\draw (4.25,1.5)--(4.25,2.5);

\draw (4.5,2.5)--(4.5,3);
\draw (5,2.5)--(5,3);
\draw (5.5,2.5)--(5.5,3);
\draw (4.5,2.75)--(5.5,2.75);
\draw (5.5,1.5)--(5.5,2.5);
\draw (5.5,3)--(7.5,3);
\draw (7.5,3)--(7.5,1.5);

\draw (8,2.5)--(8,3);
\draw (8.5,2.5)--(8.5,3);
\draw (7.5,2.75)--(8.5,2.75);

\draw (8.5,1.5)--(9,1.5);
\draw (8.5,2)--(9,2);
\draw (8.5,2.5)--(9,2.5);
\draw (8.75,1.5)--(8.75,2.5);

\draw (8.5,3)--(9,3);
\draw (8.5,3.5)--(9,3.5);
\draw (8.5,4)--(9,4);
\draw (8.75,3)--(8.75,4);

\draw (7,3.5)--(7.5,3.5);
\draw (7,4)--(7.5,4);
\draw (7.25,3)--(7.25,4);

\draw (5.5,3.5)--(6,3.5);
\draw (5.5,4)--(6,4);
\draw (5.75,3)--(5.75,4);

\draw (6,4)--(7,4);
\draw (4.5,4)--(5.5,4);

\draw (4,3)--(4.5,3);
\draw (4,3.5)--(4.5,3.5);
\draw (4,4)--(4.5,4);
\draw (4.25,3)--(4.25,4);

\draw (4,4)--(4,4.5);
\draw (3.5,4)--(3.5,4.5);
\draw (3,4)--(3,4.5);
\draw (3,4.25)--(4,4.25);

\draw (2.5,4)--(3,4);
\draw (2.5,3.5)--(3,3.5);
\draw (2.75,3)--(2.75,4);

\draw (1,4)--(1.5,4);
\draw (1,3.5)--(1.5,3.5);
\draw (1.25,3)--(1.25,4);

\draw (7.5,4)--(8.5,4);

\draw (0.5,4)--(0.5,4.5);
\draw (1,4)--(1,4.5);
\draw (0,4.25)--(1,4.25);

\draw (1.5,4)--(1.5,4.5);
\draw (2,4)--(2,4.5);
\draw (2.5,4)--(2.5,4.5);
\draw (1.5,4.25)--(2.5,4.25);

\draw (1,4.5)--(1.5,4.5);
\draw (1,5)--(1.5,5);
\draw (1,5.5)--(1.5,5.5);
\draw (1.25,4.5)--(1.25,5.5);

\draw (2.5,4.5)--(3,4.5);
\draw (2.5,5)--(3,5);
\draw (2.5,5.5)--(3,5.5);
\draw (2.75,4.5)--(2.75,5.5);

\draw (4,4.5)--(4,5.5);
\draw (1.5,5.5)--(2.5,5.5);

\draw (0.5,5.5)--(0.5,6);
\draw (1,5.5)--(1,6);
\draw (0,5.75)--(1,5.75);

\draw (3,5.5)--(3,6);
\draw (3.5,5.5)--(3.5,6);
\draw (4,5.5)--(4,6);
\draw (3,5.75)--(4,5.75);

\draw (1,6)--(1,7);
\draw (3,6)--(3,7);
\draw (4,6)--(4,7);

\draw (0.5,7)--(0.5,7.5);
\draw (1,7)--(1,7.5);
\draw (0,7.25)--(1,7.25);

\draw (3,7)--(3,7.5);
\draw (3.5,7)--(3.5,7.5);
\draw (4,7)--(4,7.5);
\draw (3,7.25)--(4,7.25);

\draw (1,7.5)--(3,7.5);

\draw (1,8)--(1.5,8);
\draw (1,8.5)--(1.5,8.5);
\draw (1.25,7.5)--(1.25,8.5);

\draw (2.5,8)--(3,8);
\draw (2.5,8.5)--(3,8.5);
\draw (2.75,7.5)--(2.75,8.5);

\draw (0,8.5)--(0,9);
\draw (0.5,8.5)--(0.5,9);
\draw (1,8.5)--(1,9);
\draw (1.5,8.5)--(1.5,9);
\draw (2,8.5)--(2,9);
\draw (2.5,8.5)--(2.5,9);
\draw (3,8.5)--(3,9);
\draw (3.5,8.5)--(3.5,9);
\draw (4,8.5)--(4,9);

\draw (0,8.75)--(1,8.75);
\draw (1.5,8.75)--(2.5,8.75);
\draw (3,8.75)--(4,8.75);

\draw (4,8.5)--(4,7.5);

\draw (2,9.5) node {\ldots};
\draw (9.5,2) node {\vdots};
\end{tikzpicture}
\caption{First Transformation}\label{fig_trans2}
\end{figure}





\begin{figure}[ht]
\centering
\begin{tikzpicture}

\draw[fill=black] (0,0) circle (0.06);
\draw[fill=black] (0.5,0) circle (0.06);
\draw[fill=black] (1,0) circle (0.06);
\draw[fill=black] (1.5,0) circle (0.06);
\draw[fill=black] (2,0) circle (0.06);
\draw[fill=black] (2.5,0) circle (0.06);
\draw[fill=black] (3,0) circle (0.06);
\draw[fill=black] (3.5,0) circle (0.06);
\draw[fill=black] (4,0) circle (0.06);
\draw[fill=black] (4.5,0) circle (0.06);
\draw[fill=black] (5,0) circle (0.06);
\draw[fill=black] (5.5,0) circle (0.06);
\draw[fill=black] (6,0) circle (0.06);
\draw[fill=black] (6.5,0) circle (0.06);
\draw[fill=black] (7,0) circle (0.06);
\draw[fill=black] (7.5,0) circle (0.06);
\draw[fill=black] (8,0) circle (0.06);
\draw[fill=black] (8.5,0) circle (0.06);

\draw[fill=black] (0,0.5) circle (0.06);
\draw[fill=black] (1,0.5) circle (0.06);
\draw[fill=black] (1.5,0.5) circle (0.06);
\draw[fill=black] (2.5,0.5) circle (0.06);
\draw[fill=black] (3,0.5) circle (0.06);
\draw[fill=black] (4,0.5) circle (0.06);
\draw[fill=black] (4.5,0.5) circle (0.06);
\draw[fill=black] (5.5,0.5) circle (0.06);
\draw[fill=black] (6,0.5) circle (0.06);
\draw[fill=black] (7,0.5) circle (0.06);
\draw[fill=black] (7.5,0.5) circle (0.06);
\draw[fill=black] (8.5,0.5) circle (0.06);

\draw[fill=black] (0,1) circle (0.06);
\draw[fill=black] (0.5,1) circle (0.06);
\draw[fill=black] (1,1) circle (0.06);
\draw[fill=black] (1.5,1) circle (0.06);
\draw[fill=black] (2,1) circle (0.06);
\draw[fill=black] (2.5,1) circle (0.06);
\draw[fill=black] (3,1) circle (0.06);
\draw[fill=black] (3.5,1) circle (0.06);
\draw[fill=black] (4,1) circle (0.06);
\draw[fill=black] (4.5,1) circle (0.06);
\draw[fill=black] (5,1) circle (0.06);
\draw[fill=black] (5.5,1) circle (0.06);
\draw[fill=black] (6,1) circle (0.06);
\draw[fill=black] (6.5,1) circle (0.06);
\draw[fill=black] (7,1) circle (0.06);
\draw[fill=black] (7.5,1) circle (0.06);
\draw[fill=black] (8,1) circle (0.06);
\draw[fill=black] (8.5,1) circle (0.06);

\draw[fill=black] (0,1.5) circle (0.06);
\draw[fill=black] (0.5,1.5) circle (0.06);
\draw[fill=black] (1,1.5) circle (0.06);
\draw[fill=black] (3,1.5) circle (0.06);
\draw[fill=black] (3.5,1.5) circle (0.06);
\draw[fill=black] (4,1.5) circle (0.06);
\draw[fill=black] (4.5,1.5) circle (0.06);
\draw[fill=black] (5,1.5) circle (0.06);
\draw[fill=black] (5.5,1.5) circle (0.06);
\draw[fill=black] (7.5,1.5) circle (0.06);
\draw[fill=black] (8,1.5) circle (0.06);
\draw[fill=black] (8.5,1.5) circle (0.06);

\draw[fill=black] (0,2) circle (0.06);
\draw[fill=black] (1,2) circle (0.06);
\draw[fill=black] (3,2) circle (0.06);
\draw[fill=black] (4,2) circle (0.06);
\draw[fill=black] (4.5,2) circle (0.06);
\draw[fill=black] (5.5,2) circle (0.06);
\draw[fill=black] (7.5,2) circle (0.06);
\draw[fill=black] (8.5,2) circle (0.06);

\draw[fill=black] (0,2.5) circle (0.06);
\draw[fill=black] (0.5,2.5) circle (0.06);
\draw[fill=black] (1,2.5) circle (0.06);
\draw[fill=black] (3,2.5) circle (0.06);
\draw[fill=black] (3.5,2.5) circle (0.06);
\draw[fill=black] (4,2.5) circle (0.06);
\draw[fill=black] (4.5,2.5) circle (0.06);
\draw[fill=black] (5,2.5) circle (0.06);
\draw[fill=black] (5.5,2.5) circle (0.06);
\draw[fill=black] (7.5,2.5) circle (0.06);
\draw[fill=black] (8,2.5) circle (0.06);
\draw[fill=black] (8.5,2.5) circle (0.06);

\draw[fill=black] (0,3) circle (0.06);
\draw[fill=black] (0.5,3) circle (0.06);
\draw[fill=black] (1,3) circle (0.06);
\draw[fill=black] (1.5,3) circle (0.06);
\draw[fill=black] (2,3) circle (0.06);
\draw[fill=black] (2.5,3) circle (0.06);
\draw[fill=black] (3,3) circle (0.06);
\draw[fill=black] (3.5,3) circle (0.06);
\draw[fill=black] (4,3) circle (0.06);
\draw[fill=black] (4.5,3) circle (0.06);
\draw[fill=black] (5,3) circle (0.06);
\draw[fill=black] (5.5,3) circle (0.06);
\draw[fill=black] (6,3) circle (0.06);
\draw[fill=black] (6.5,3) circle (0.06);
\draw[fill=black] (7,3) circle (0.06);
\draw[fill=black] (7.5,3) circle (0.06);
\draw[fill=black] (8,3) circle (0.06);
\draw[fill=black] (8.5,3) circle (0.06);

\draw[fill=black] (0,3.5) circle (0.06);
\draw[fill=black] (1,3.5) circle (0.06);
\draw[fill=black] (1.5,3.5) circle (0.06);
\draw[fill=black] (2.5,3.5) circle (0.06);
\draw[fill=black] (3,3.5) circle (0.06);
\draw[fill=black] (4,3.5) circle (0.06);
\draw[fill=black] (4.5,3.5) circle (0.06);
\draw[fill=black] (5.5,3.5) circle (0.06);
\draw[fill=black] (6,3.5) circle (0.06);
\draw[fill=black] (7,3.5) circle (0.06);
\draw[fill=black] (7.5,3.5) circle (0.06);
\draw[fill=black] (8.5,3.5) circle (0.06);

\draw[fill=black] (0,4) circle (0.06);
\draw[fill=black] (0.5,4) circle (0.06);
\draw[fill=black] (1,4) circle (0.06);
\draw[fill=black] (1.5,4) circle (0.06);
\draw[fill=black] (2,4) circle (0.06);
\draw[fill=black] (2.5,4) circle (0.06);
\draw[fill=black] (3,4) circle (0.06);
\draw[fill=black] (3.5,4) circle (0.06);
\draw[fill=black] (4,4) circle (0.06);
\draw[fill=black] (4.5,4) circle (0.06);
\draw[fill=black] (5,4) circle (0.06);
\draw[fill=black] (5.5,4) circle (0.06);
\draw[fill=black] (6,4) circle (0.06);
\draw[fill=black] (6.5,4) circle (0.06);
\draw[fill=black] (7,4) circle (0.06);
\draw[fill=black] (7.5,4) circle (0.06);
\draw[fill=black] (8,4) circle (0.06);
\draw[fill=black] (8.5,4) circle (0.06);

\draw[fill=black] (0,4.5) circle (0.06);
\draw[fill=black] (0.5,4.5) circle (0.06);
\draw[fill=black] (1,4.5) circle (0.06);
\draw[fill=black] (1.5,4.5) circle (0.06);
\draw[fill=black] (2,4.5) circle (0.06);
\draw[fill=black] (2.5,4.5) circle (0.06);
\draw[fill=black] (3,4.5) circle (0.06);
\draw[fill=black] (3.5,4.5) circle (0.06);
\draw[fill=black] (4,4.5) circle (0.06);

\draw[fill=black] (0,5) circle (0.06);
\draw[fill=black] (1,5) circle (0.06);
\draw[fill=black] (1.5,5) circle (0.06);
\draw[fill=black] (2.5,5) circle (0.06);
\draw[fill=black] (3,5) circle (0.06);
\draw[fill=black] (4,5) circle (0.06);

\draw[fill=black] (0,5.5) circle (0.06);
\draw[fill=black] (0.5,5.5) circle (0.06);
\draw[fill=black] (1,5.5) circle (0.06);
\draw[fill=black] (1.5,5.5) circle (0.06);
\draw[fill=black] (2,5.5) circle (0.06);
\draw[fill=black] (2.5,5.5) circle (0.06);
\draw[fill=black] (3,5.5) circle (0.06);
\draw[fill=black] (3.5,5.5) circle (0.06);
\draw[fill=black] (4,5.5) circle (0.06);

\draw[fill=black] (0,6) circle (0.06);
\draw[fill=black] (0.5,6) circle (0.06);
\draw[fill=black] (1,6) circle (0.06);
\draw[fill=black] (3,6) circle (0.06);
\draw[fill=black] (3.5,6) circle (0.06);
\draw[fill=black] (4,6) circle (0.06);

\draw[fill=black] (0,6.5) circle (0.06);
\draw[fill=black] (1,6.5) circle (0.06);
\draw[fill=black] (3,6.5) circle (0.06);
\draw[fill=black] (4,6.5) circle (0.06);

\draw[fill=black] (0,7) circle (0.06);
\draw[fill=black] (0.5,7) circle (0.06);
\draw[fill=black] (1,7) circle (0.06);
\draw[fill=black] (3,7) circle (0.06);
\draw[fill=black] (3.5,7) circle (0.06);
\draw[fill=black] (4,7) circle (0.06);

\draw[fill=black] (0,7.5) circle (0.06);
\draw[fill=black] (0.5,7.5) circle (0.06);
\draw[fill=black] (1,7.5) circle (0.06);
\draw[fill=black] (1.5,7.5) circle (0.06);
\draw[fill=black] (2,7.5) circle (0.06);
\draw[fill=black] (2.5,7.5) circle (0.06);
\draw[fill=black] (3,7.5) circle (0.06);
\draw[fill=black] (3.5,7.5) circle (0.06);
\draw[fill=black] (4,7.5) circle (0.06);

\draw[fill=black] (0,8) circle (0.06);
\draw[fill=black] (1,8) circle (0.06);
\draw[fill=black] (1.5,8) circle (0.06);
\draw[fill=black] (2.5,8) circle (0.06);
\draw[fill=black] (3,8) circle (0.06);
\draw[fill=black] (4,8) circle (0.06);

\draw[fill=black] (0,8.5) circle (0.06);
\draw[fill=black] (0.5,8.5) circle (0.06);
\draw[fill=black] (1,8.5) circle (0.06);
\draw[fill=black] (1.5,8.5) circle (0.06);
\draw[fill=black] (2,8.5) circle (0.06);
\draw[fill=black] (2.5,8.5) circle (0.06);
\draw[fill=black] (3,8.5) circle (0.06);
\draw[fill=black] (3.5,8.5) circle (0.06);
\draw[fill=black] (4,8.5) circle (0.06);

\draw (0,0)--(8.5,0);
\draw (0,0)--(0,8.5);

\draw (1,1)--(1.5,1);
\draw (1,1)--(1,1.5);
\draw (1.25,0)--(1.25,1);
\draw (0,1.25)--(1,1.25);

\draw (1.5,1)--(3,1);
\draw (2.75,1)--(2.75,0);

\draw (3,1)--(3,1.5);
\draw (4,1)--(4,1.5);
\draw (3,1.25)--(4,1.25);
\draw (4,1)--(4.5,1);
\draw (4.25,0)--(4.25,1);

\draw (4.5,1)--(4.5,1.5);
\draw (5.5,1)--(5.5,1.5);
\draw (4.5,1.25)--(5.5,1.25);
\draw (5.5,1)--(6,1);
\draw (5.75,0)--(5.75,1);
\draw (6,1)--(7,1);
\draw (7,1)--(7.5,1);
\draw (7.25,0)--(7.25,1);
\draw (7.5,1)--(7.5,1.5);
\draw (8.5,1)--(8.5,1.5);
\draw (7.5,1.25)--(8.5,1.25);

\draw (8.5,0)--(9,0);
\draw (8.5,1)--(9,1);
\draw (8.75,0)--(8.75,1);

\draw (1,1.5)--(1,3);
\draw (0,2.75)--(1,2.75);

\draw (1,3)--(3,3);
\draw (3,3)--(3,1);

\draw (4,2.5)--(4,3);
\draw (3,2.75)--(4,2.75);
\draw (4,1.5)--(4.5,1.5);
\draw (4,2.5)--(4.5,2.5);
\draw (4.25,1.5)--(4.25,2.5);

\draw (4.5,2.5)--(4.5,3);
\draw (5.5,2.5)--(5.5,3);
\draw (4.5,2.75)--(5.5,2.75);
\draw (5.5,1.5)--(5.5,2.5);
\draw (5.5,3)--(7.5,3);
\draw (7.5,3)--(7.5,1.5);

\draw (8.5,2.5)--(8.5,3);
\draw (7.5,2.75)--(8.5,2.75);

\draw (8.5,1.5)--(9,1.5);
\draw (8.5,2.5)--(9,2.5);
\draw (8.75,1.5)--(8.75,2.5);

\draw (8.5,3)--(9,3);
\draw (8.5,4)--(9,4);
\draw (8.75,3)--(8.75,4);

\draw (7,4)--(7.5,4);
\draw (7.25,3)--(7.25,4);

\draw (5.5,4)--(6,4);
\draw (5.75,3)--(5.75,4);

\draw (6,4)--(7,4);
\draw (4.5,4)--(5.5,4);

\draw (4,3)--(4.5,3);
\draw (4,4)--(4.5,4);
\draw (4.25,3)--(4.25,4);

\draw (4,4)--(4,4.5);
\draw (3,4)--(3,4.5);
\draw (3,4.25)--(4,4.25);

\draw (2.5,4)--(3,4);
\draw (2.75,3)--(2.75,4);

\draw (1,4)--(1.5,4);
\draw (1.25,3)--(1.25,4);

\draw (7.5,4)--(8.5,4);

\draw (1,4)--(1,4.5);
\draw (0,4.25)--(1,4.25);

\draw (1.5,4)--(1.5,4.5);
\draw (2.5,4)--(2.5,4.5);
\draw (1.5,4.25)--(2.5,4.25);

\draw (1,4.5)--(1.5,4.5);
\draw (1,5.5)--(1.5,5.5);
\draw (1.25,4.5)--(1.25,5.5);

\draw (2.5,4.5)--(3,4.5);
\draw (2.5,5.5)--(3,5.5);
\draw (2.75,4.5)--(2.75,5.5);

\draw (4,4.5)--(4,5.5);
\draw (1.5,5.5)--(2.5,5.5);

\draw (1,5.5)--(1,6);
\draw (0,5.75)--(1,5.75);

\draw (3,5.5)--(3,6);
\draw (4,5.5)--(4,6);
\draw (3,5.75)--(4,5.75);

\draw (1,6)--(1,7);
\draw (3,6)--(3,7);
\draw (4,6)--(4,7);

\draw (1,7)--(1,7.5);
\draw (0,7.25)--(1,7.25);

\draw (3,7)--(3,7.5);
\draw (4,7)--(4,7.5);
\draw (3,7.25)--(4,7.25);

\draw (1,7.5)--(3,7.5);

\draw (1,8.5)--(1.5,8.5);
\draw (1.25,7.5)--(1.25,8.5);

\draw (2.5,8.5)--(3,8.5);
\draw (2.75,7.5)--(2.75,8.5);

\draw (0,8.5)--(0,9);
\draw (1,8.5)--(1,9);
\draw (1.5,8.5)--(1.5,9);
\draw (2.5,8.5)--(2.5,9);
\draw (3,8.5)--(3,9);
\draw (4,8.5)--(4,9);

\draw (0,8.75)--(1,8.75);
\draw (1.5,8.75)--(2.5,8.75);
\draw (3,8.75)--(4,8.75);

\draw (4,8.5)--(4,7.5);

\draw (2,9.5) node {\ldots};
\draw (9.5,2) node {\vdots};
\end{tikzpicture}
\caption{First Transformation}\label{fig_trans3}
\end{figure}





\begin{figure}[ht]
\centering
\begin{tikzpicture}

\draw (0.25-1,0.25)--(0.25-1,-0.25)--(-0.25-1,-0.25)--(-0.25-1,0.25)--cycle;
\draw (0-1,0.25)--(0-1,0.75);
\draw (0-1,-0.25)--(0-1,-0.75);
\draw (0.25-1,0)--(0.75-1,0);
\draw (-0.25-1,0)--(-0.75-1,0);

\draw[fill=black] (0.25-1,0.25) circle (0.06);
\draw[fill=black] (0.25-1,-0.25) circle (0.06);
\draw[fill=black] (-0.25-1,-0.25) circle (0.06);
\draw[fill=black] (-0.25-1,0.25) circle (0.06);
\draw[fill=white] (-1,0) circle (0.06);

\draw (0,0) node {$\Rightarrow$};

\draw (0.25,0)--(1.75,0);
\draw (1,-0.75)--(1,0.75);

\draw[fill=black] (0.25+1,0.25) circle (0.06);
\draw[fill=black] (0.25+1,-0.25) circle (0.06);
\draw[fill=black] (-0.25+1,-0.25) circle (0.06);
\draw[fill=black] (-0.25+1,0.25) circle (0.06);
\draw[fill=white] (1,0) circle (0.06);

\draw (-5,-1.5)--(-5,-2)--(-4.5,-2);
\draw[fill=black] (-5,-1.5) circle (0.06);
\draw[fill=black] (-5,-2) circle (0.06);
\draw[fill=black] (-4.5,-2) circle (0.06);
\draw[fill=white] (-5,-1.75) circle (0.06);
\draw[fill=white] (-4.75,-2) circle (0.06);
\draw[fill=white] (-4.75,-1.75) circle (0.06);

\draw (-4,-2) node {$\Rightarrow$};

\draw (-3,-1.75)--(-2.5,-1.75)--(-2.5,-2);
\draw (-2.75,-2)--(-2.75,-1.5)--(-3,-1.5);
\draw[fill=black] (-5+2,-1.5) circle (0.06);
\draw[fill=black] (-5+2,-2) circle (0.06);
\draw[fill=black] (-4.5+2,-2) circle (0.06);
\draw[fill=white] (-5+2,-1.75) circle (0.06);
\draw[fill=white] (-4.75+2,-2) circle (0.06);
\draw[fill=white] (-4.75+2,-1.75) circle (0.06);

\draw (3,-2)--(3,-1.5)--(3.5,-1.5);
\draw[fill=black] (3,-2) circle (0.06);
\draw[fill=black] (3,-1.5) circle (0.06);
\draw[fill=black] (3.5,-1.5) circle (0.06);
\draw[fill=white] (3,-1.75) circle (0.06);
\draw[fill=white] (3.25,-1.5) circle (0.06);
\draw[fill=white] (3.25,-1.75) circle (0.06);

\draw (4,-2) node {$\Rightarrow$};

\draw (5,-1.75)--(5.5,-1.75)--(5.5,-1.5);
\draw (5.25,-1.5)--(5.25,-2)--(5,-2);
\draw[fill=black] (5,-2) circle (0.06);
\draw[fill=black] (5,-1.5) circle (0.06);
\draw[fill=black] (5.5,-1.5) circle (0.06);
\draw[fill=white] (5,-1.75) circle (0.06);
\draw[fill=white] (5.25,-1.5) circle (0.06);
\draw[fill=white] (5.25,-1.75) circle (0.06);

\draw (-5,-2.5)--(-4.5,-2.5)--(-4.5,-3);
\draw[fill=black] (-5,-2.5) circle (0.06);
\draw[fill=black] (-4.5,-2.5) circle (0.06);
\draw[fill=black] (-4.5,-3) circle (0.06);
\draw[fill=white] (-4.75,-2.5) circle (0.06);
\draw[fill=white] (-4.5,-2.75) circle (0.06);
\draw[fill=white] (-4.75,-2.75) circle (0.06);

\draw (-4,-3) node {$\Rightarrow$};

\draw (-2.75,-2.5)--(-2.75,-3)--(-2.5,-3);
\draw (-3,-2.5)--(-3,-2.75)--(-2.5,-2.75);
\draw[fill=black] (-5+2,-2.5) circle (0.06);
\draw[fill=black] (-4.5+2,-2.5) circle (0.06);
\draw[fill=black] (-4.5+2,-3) circle (0.06);
\draw[fill=white] (-4.75+2,-2.5) circle (0.06);
\draw[fill=white] (-4.5+2,-2.75) circle (0.06);
\draw[fill=white] (-4.75+2,-2.75) circle (0.06);

\draw (3,-3)--(3.5,-3)--(3.5,-2.5);
\draw[fill=black] (3,-3) circle (0.06);
\draw[fill=black] (3.5,-3) circle (0.06);
\draw[fill=black] (3.5,-2.5) circle (0.06);
\draw[fill=white] (3.25,-3) circle (0.06);
\draw[fill=white] (3.5,-2.75) circle (0.06);
\draw[fill=white] (3.25,-2.75) circle (0.06);

\draw (4,-3) node {$\Rightarrow$};

\draw (5,-3)--(5,-2.75)--(5.5,-2.75);
\draw (5.25,-3)--(5.25,-2.5)--(5.5,-2.5);
\draw[fill=black] (3+2,-3) circle (0.06);
\draw[fill=black] (3.5+2,-3) circle (0.06);
\draw[fill=black] (3.5+2,-2.5) circle (0.06);
\draw[fill=white] (3.25+2,-3) circle (0.06);
\draw[fill=white] (3.5+2,-2.75) circle (0.06);
\draw[fill=white] (3.25+2,-2.75) circle (0.06);
\end{tikzpicture}
\caption{Second Transformation}\label{fig_trans4}
\end{figure}
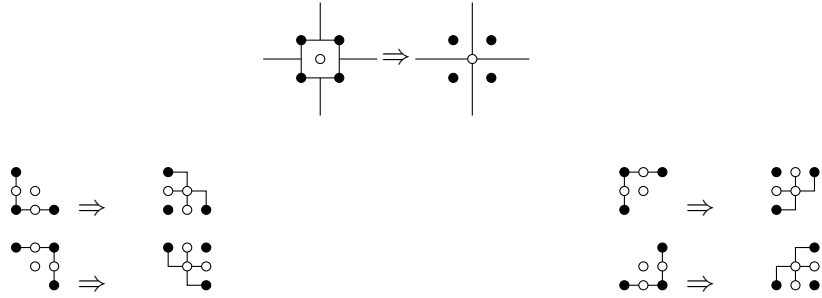


Second, we do the transformations in Figure \ref{fig_trans4} where the resistances of the resistors in the new networks only depend on the shape of the networks in Figure \ref{fig_trans4} such that we obtain a weighted graph with vertex set $V_{n-1}$ and all conductances equivalent to $1$. Moreover, the resistances between any two points are larger than those in the weighted graph $\calW_n$, hence we obtain the desired result.
\end{proof}

Now we estimate $R_n(p_0,p_1)$ and $\frakR_n(0^n,1^n)$ as follows.

\begin{proof}[Proof of Theorem \ref{thm_resist}]
The idea is that replacing one point by one network should increase resistances by multiplying the resistance of an individual network.

By Proposition \ref{prop_resist2} and Proposition \ref{prop_resist3}, we have for all $n\ge1$
$$R_n(p_0,p_1)\asymp\frakR_n(0^n,1^n).$$
By Theorem \ref{thm_resist1} and Proposition \ref{prop_resist2}, we have for all $n\ge1$
$$\frakR_n(0^n,1^n)\ge R_n(p_0,p_1)\ge\frac{1}{4}R_n(p_1,p_5)\asymp\rho^n.$$
We only need to show that for all $n\ge1$
$$\frakR_n(0^n,1^n)\lesssim\rho^n.$$

First, we estimate $\frakR_{n+1}(0^{n+1},12^n)$. Cutting certain edges in $\calW_{n+1}$, we obtain the electrical network in Figure \ref{fig_resist1} which is equivalent to the electrical networks in Figure \ref{fig_resist2}.




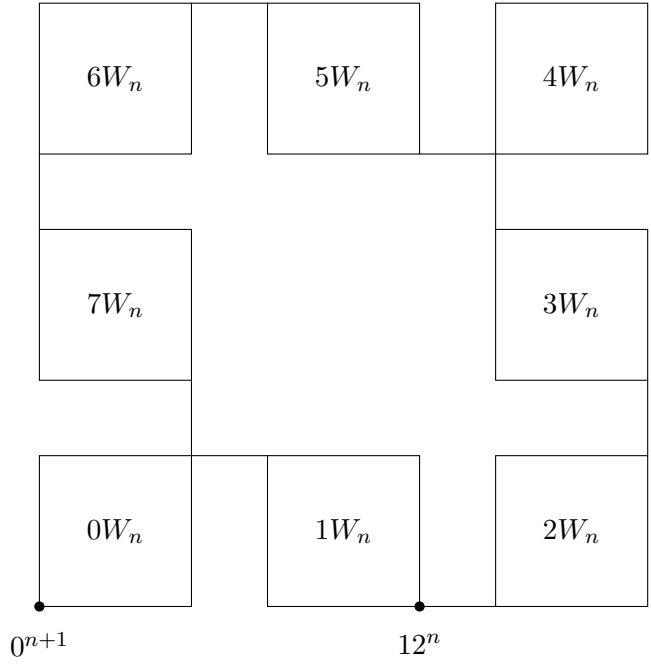
\begin{figure}[ht]
\centering
\begin{tikzpicture}
\draw (0,0)--(2,0)--(2,2)--(0,2)--cycle;
\draw (3,0)--(5,0)--(5,2)--(3,2)--cycle;
\draw (6,0)--(8,0)--(8,2)--(6,2)--cycle;
\draw (0,3)--(2,3)--(2,5)--(0,5)--cycle;
\draw (6,3)--(8,3)--(8,5)--(6,5)--cycle;
\draw (0,6)--(2,6)--(2,8)--(0,8)--cycle;
\draw (3,6)--(5,6)--(5,8)--(3,8)--cycle;
\draw (6,6)--(8,6)--(8,8)--(6,8)--cycle;

\draw (2,2)--(2,3);
\draw (2,2)--(3,2);
\draw (0,5)--(0,6);
\draw (2,8)--(3,8);
\draw (5,6)--(6,6);
\draw (6,6)--(6,5);
\draw (8,3)--(8,2);
\draw (5,0)--(6,0);

\draw[fill=black] (0,0) circle (0.06);
\draw[fill=black] (5,0) circle (0.06);

\draw (0,-0.5) node {$0^{n+1}$};
\draw (5,-0.5) node {$12^n$};

\draw (1,1) node {$0W_n$};
\draw (4,1) node {$1W_n$};
\draw (7,1) node {$2W_n$};
\draw (7,4) node {$3W_n$};
\draw (7,7) node {$4W_n$};
\draw (4,7) node {$5W_n$};
\draw (1,7) node {$6W_n$};
\draw (1,4) node {$7W_n$};

\end{tikzpicture}
\caption{An Equivalent Electrical Network}\label{fig_resist1}
\end{figure}





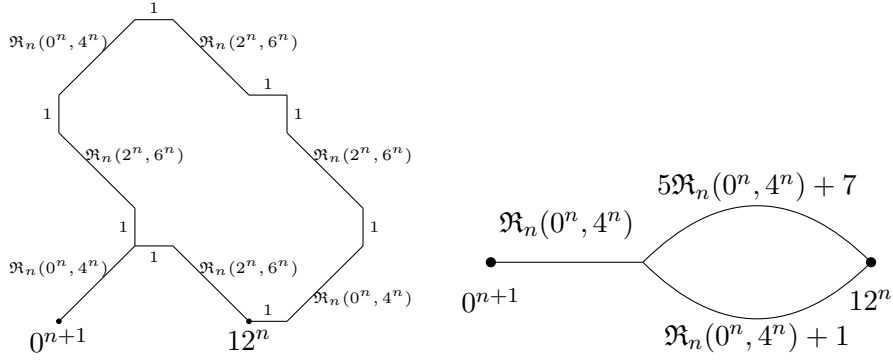
\begin{figure}[ht]
\centering
\subfigure{
\begin{tikzpicture}[scale=0.5]
\draw (0,0)--(2,2);
\draw (5,0)--(3,2);
\draw (6,0)--(8,2);
\draw (2,3)--(0,5);
\draw (8,3)--(6,5);
\draw (0,6)--(2,8);
\draw (5,6)--(3,8);

\draw (2,2)--(2,3);
\draw (2,2)--(3,2);
\draw (0,5)--(0,6);
\draw (2,8)--(3,8);
\draw (5,6)--(6,6);
\draw (6,6)--(6,5);
\draw (8,3)--(8,2);
\draw (5,0)--(6,0);

\draw[fill=black] (0,0) circle (0.06);
\draw[fill=black] (5,0) circle (0.06);

\draw (0,-0.5) node {$0^{n+1}$};
\draw (5,-0.5) node {$12^n$};

\draw (0,1.4) node {\tiny{$\mathfrak{R}_n(0^n,4^n)$}};
\draw (2,4.4) node {\tiny{$\mathfrak{R}_n(2^n,6^n)$}};
\draw (0,7.4) node {\tiny{$\mathfrak{R}_n(0^n,4^n)$}};
\draw (5,7.4) node {\tiny{$\mathfrak{R}_n(2^n,6^n)$}};
\draw (8,4.4) node {\tiny{$\mathfrak{R}_n(2^n,6^n)$}};
\draw (5,1.4) node {\tiny{$\mathfrak{R}_n(2^n,6^n)$}};
\draw (8,0.6) node {\tiny{$\mathfrak{R}_n(0^n,4^n)$}};

\draw (2.5,1.7) node {\tiny{$1$}};
\draw (1.7,2.5) node {\tiny{$1$}};
\draw (-0.3,5.5) node {\tiny{$1$}};
\draw (2.5,8.3) node {\tiny{$1$}};
\draw (5.5,6.3) node {\tiny{$1$}};
\draw (6.3,5.5) node {\tiny{$1$}};
\draw (8.3,2.5) node {\tiny{$1$}};
\draw (5.5,0.3) node {\tiny{$1$}};

\end{tikzpicture}}
\hspace{0.01pt}
\subfigure{
\begin{tikzpicture}
\draw (0,0)--(2,0);
\draw (2,0)..controls (3,1) and (4,1) ..(5,0);
\draw (2,0)..controls (3,-1) and (4,-1) ..(5,0);

\draw[fill=black] (0,0) circle (0.06);
\draw[fill=black] (5,0) circle (0.06);

\draw (0,-0.5) node {$0^{n+1}$};
\draw (5,-0.5) node {$12^n$};

\draw (1,0.5) node {$\mathfrak{R}_n(0^n,4^n)$};
\draw (3.5,1) node {$5\mathfrak{R}_n(0^n,4^n)+7$};
\draw (3.5,-1) node {$\mathfrak{R}_n(0^n,4^n)+1$};

\end{tikzpicture}
}
\caption{Equivalent Electrical Networks}\label{fig_resist2}
\end{figure}


Hence
$$
\begin{aligned}
\frakR_{n+1}(0^{n+1},12^n)&\le\frakR_n(0^n,4^n)+\frac{\left(5\frakR_n(0^n,4^n)+7\right)\left(\frakR_n(0^n,4^n)+1\right)}{\left(5\frakR_n(0^n,4^n)+7\right)+\left(\frakR_n(0^n,4^n)+1\right)}\\
&\lesssim\frakR_n(0^n,4^n)+\frac{5}{6}\frakR_n(0^n,4^n)=\frac{11}{6}\frakR_n(0^n,4^n)\lesssim\rho^{n+1}.
\end{aligned}
$$
Second, from $0^{n+1}$ to $1^{n+1}$, we construct a finite sequence as follows. For $i=1,\ldots,n+2$,
$$
w^{(i)}=
\begin{cases}
1^{i-1}0^{n+2-i},\text{ if }i\text{ is an odd number},\\
1^{i-1}2^{n+2-i},\text{ if }i\text{ is an even number}.\\
\end{cases}
$$
By cutting technique, if $i$ is an odd number, then
$$
\begin{aligned}
&\frakR_{n+1}(w^{(i)},w^{(i+1)})=\frakR_{n+1}(1^{i-1}0^{n+2-i},1^{i}2^{n+1-i})\\
&\le\frakR_{n+2-i}(0^{n+2-i},12^{n+1-i})\lesssim\rho^{n+2-i}.
\end{aligned}
$$
If $i$ is an even number, then
$$
\begin{aligned}
&\frakR_{n+1}(w^{(i)},w^{(i+1)})=\frakR_{n+1}(1^{i-1}2^{n+2-i},1^{i}0^{n+1-i})\\
&\le\frakR_{n+2-i}(2^{n+2-i},10^{n+1-i})=\frakR_{n+2-i}(0^{n+2-i},12^{n+1-i})\lesssim\rho^{n+2-i}.
\end{aligned}
$$
Hence
$$
\begin{aligned}
&\frakR_{n+1}(0^{n+1},1^{n+1})=\frakR_{n+1}(w^{(1)},w^{(n+2)})\\
&\le\sum_{i=1}^{n+1}\frakR_{n+1}(w^{(i)},w^{(i+1)})\lesssim\sum_{i=1}^{n+1}\rho^{n+2-i}=\sum_{i=1}^{n+1}\rho^{i}\lesssim\rho^{n+1}.
\end{aligned}
$$
\end{proof}

\section{Uniform Harnack Inequality}\label{sec_harnack}

In this section, we give uniform Harnack inequality as follows.

\begin{mythm}\label{thm_harnack}
There exist some constants $C\in(0,+\infty),\delta\in(0,1)$ such that for all $n\ge1,x\in K,r>0$, for all nonnegative harmonic function $u$ on $V_n\cap B(x,r)$, we have
$$\max_{V_n\cap B(x,\delta r)}u\le C\min_{V_n\cap B(x,\delta r)}u.$$
\end{mythm}

\begin{myrmk}
The point of above theorem is that the constant $C$ is \emph{uniform} in $n$.
\end{myrmk}

The idea is as follows. First, we use resistance estimates in finite graphs $V_n$ to obtain resistance estimates in an infinite graph $V_\infty$. Second, we obtain Green function estimates in $V_\infty$. Third, we obtain elliptic Harnack inequality in $V_\infty$. Finally, we transfer elliptic Harnack inequality in $V_\infty$ to uniform Harnack inequality in $V_n$.

Let $\calV_\infty$ be the graph with vertex set $V_\infty=\cup_{n=0}^\infty3^nV_n$ and edge set given by
$$\myset{(p,q):p,q\in V_\infty,|p-q|=2^{-1}}.$$
We have the figure of $\calV_\infty$ in Figure \ref{fig_graphSC}.

Locally, $\calV_\infty$ is like $\calV_n$. Let the conductances of all edges be $1$. Let $d$ be the graph distance, that is, $d(p,q)$ is the minimum of the lengths of all paths connecting $p$ and $q$. It is obvious that
$$d(p,q)\asymp|p-q|\text{ for all }p,q\in V_\infty.$$

By shorting and cutting technique, we reduce $\calV_\infty$ to $\calV_n$ to obtain resistance estimates as follows.
$$R(x,y)\asymp\rho^{\frac{\log d(x,y)}{\log 3}}=d(x,y)^{\frac{\log\rho}{\log3}}=d(x,y)^\gamma\text{ for all }x,y\in V_\infty,$$
where $\gamma=\log\rho/\log3$.

Let $g_B$ be the Green function in a ball $B$. We have Green function estimates as follows.

\begin{mythm}(\cite[Proposition 6.11]{GHL14})\label{thm_green}
There exist some constants $C\in(0,+\infty),\eta\in(0,1)$ such that for all $z\in V_\infty,r>0$, we have
$$g_{B(z,r)}(x,y)\le Cr^\gamma\text{ for all }x,y\in B(z,r),$$
$$g_{B(z,r)}(z,y)\ge\frac{1}{C}r^\gamma\text{ for all }y\in B(z,\eta r).$$
\end{mythm}

We obtain elliptic Harnack inequality in $V_\infty$ as follows.

\begin{mythm}(\cite[Lemma 10.2]{GT01},\cite[Theorem 3.12]{GH14a})\label{thm_harnack_infinite}
There exist some constants $C\in(0,+\infty)$, $\delta\in(0,1)$ such that for all $z\in V_\infty,r>0$, for all nonnegative harmonic function $u$ on $V_\infty\cap B(z,r)$, we have
$$\max_{B(z,\delta r)}u\le C\min_{B(z,\delta r)}u.$$
\end{mythm}

\begin{myrmk}
We give an alternative approach as follows. It was proved in \cite{BCK05} that sub-Gaussian heat kernel estimates are equivalent to resistance estimates for random walks on fractal graph under strongly recurrent condition. Hence we obtain sub-Gaussian heat kernel estimates, see \cite[Example 4]{BCK05}. It was proved in \cite[Theorem 3.1]{GT02} that sub-Gaussian heat kernel estimates imply elliptic Harnack inequality. Hence we obtain elliptic Harnack inequality in $V_\infty$.
\end{myrmk}

Now we obtain Theorem \ref{thm_harnack} directly.

\section{Weak Monotonicity Results}\label{sec_monotone}

In this section, we give two weak monotonicity results.

For all $n\ge1$, let
$$a_n(u)=\rho^n\sum_{w\in W_n}
{\sum_{\mbox{\tiny
$
\begin{subarray}{c}
p,q\in V_w\\
|p-q|=2^{-1}\cdot3^{-n}
\end{subarray}
$
}}}
(u(p)-u(q))^2,u\in l(V_n).$$

We have one weak monotonicity result as follows.

\begin{mythm}\label{thm_monotone1}
There exists some positive constant $C$ such that for all $n,m\ge1,u\in l(V_{n+m})$, we have
$$a_n(u)\le Ca_{n+m}(u).$$
\end{mythm}
\begin{proof}
For all $w\in W_n,p,q\in V_w$ with $|p-q|=2^{-1}\cdot3^{-n}$, by cutting technique and Corollary \ref{cor_resist_upper}
$$
\begin{aligned}
\left(u(p)-u(q)\right)^2&\le R_m(f_w^{-1}(p),f_w^{-1}(q))
\sum_{v\in W_m}
{\sum_{\mbox{\tiny
$
\begin{subarray}{c}
x,y\in V_{wv}\\
|x-y|=2^{-1}\cdot3^{-(n+m)}
\end{subarray}
$
}}}
(u(x)-u(y))^2\\
&\le C\rho^m\sum_{v\in W_m}
{\sum_{\mbox{\tiny
$
\begin{subarray}{c}
x,y\in V_{wv}\\
|x-y|=2^{-1}\cdot3^{-(n+m)}
\end{subarray}
$
}}}
(u(x)-u(y))^2.
\end{aligned}
$$
Hence
$$
\begin{aligned}
a_n(u)&=\rho^n\sum_{w\in W_n}
{\sum_{\mbox{\tiny
$
\begin{subarray}{c}
p,q\in V_w\\
|p-q|=2^{-1}\cdot3^{-n}
\end{subarray}
$
}}}
(u(p)-u(q))^2\\
&\le\rho^n\sum_{w\in W_n}
{\sum_{\mbox{\tiny
$
\begin{subarray}{c}
p,q\in V_w\\
|p-q|=2^{-1}\cdot3^{-n}
\end{subarray}
$
}}}
\left(C\rho^m\sum_{v\in W_m}
{\sum_{\mbox{\tiny
$
\begin{subarray}{c}
x,y\in V_{wv}\\
|x-y|=2^{-1}\cdot3^{-(n+m)}
\end{subarray}
$
}}}
(u(x)-u(y))^2\right)\\
&=C\rho^{n+m}\sum_{w\in W_{n+m}}
{\sum_{\mbox{\tiny
$
\begin{subarray}{c}
p,q\in V_w\\
|p-q|=2^{-1}\cdot3^{-(n+m)}
\end{subarray}
$
}}}
(u(p)-u(q))^2=Ca_{n+m}(u).
\end{aligned}
$$
\end{proof}

For all $n\ge1$, let
$$b_n(u)=\rho^n
{\sum_{\mbox{\tiny
$
\begin{subarray}{c}
w^{(1)}\sim_nw^{(2)}
\end{subarray}
$
}}}
(P_nu(w^{(1)})-P_nu(w^{(2)}))^2,u\in L^2(K;\nu).$$

We have another weak monotonicity result as follows.
\begin{mythm}\label{thm_monotone2}
There exists some positive constant $C$ such that for all $n,m\ge1,u\in L^2(K;\nu)$, we have
$$b_n(u)\le Cb_{n+m}(u).$$
\end{mythm}

\begin{myrmk}
This result was also obtained in \cite[Proposition 5.2]{KZ92}. Here we give a direct proof using resistance estimates.
\end{myrmk}

This result can be reduced as follows.

For all $n\ge1$, let
$$B_n(u)=\rho^n\sum_{w^{(1)}\sim_nw^{(2)}}(u(w^{(1)})-u(w^{(2)}))^2,u\in l(W_n).$$

For all $n,m\ge1$, let $M_{n,m}:l(W_{n+m})\to l(W_n)$ be a mean value operator given by
$$(M_{n,m}u)(w)=\frac{1}{8^m}\sum_{v\in W_m}u(wv),w\in W_n,u\in l(W_{n+m}).$$

\begin{mythm}\label{thm_monotonicity2}
There exists some positive constant $C$ such that for all $n,m\ge1,u\in l(W_{n+m})$, we have
$$B_n(M_{n,m}u)\le CB_{n+m}(u).$$
\end{mythm}

\begin{proof}[Proof of Theorem \ref{thm_monotone2} using Theorem \ref{thm_monotonicity2}]
For all $u\in L^2(K;\nu)$, note that
$$P_nu=M_{n,m}(P_{n+m}u),$$
hence
$$
\begin{aligned}
b_n(u)&=\rho^n\sum_{w^{(1)}\sim_nw^{(2)}}(P_nu(w^{(1)})-P_nu(w^{(2)}))^2=B_n(P_nu)\\
&=B_n(M_{n,m}(P_{n+m}u))\le CB_{n+m}(P_{n+m}u)\\
&=C\rho^{n+m}\sum_{w^{(1)}\sim_{n+m}w^{(2)}}(P_{n+m}u(w^{(1)})-P_{n+m}u(w^{(2)}))^2=Cb_{n+m}(u).
\end{aligned}
$$
\end{proof}

\begin{proof}[Proof of Theorem \ref{thm_monotonicity2}]
Fix $n\ge1$. Assume that $W\subseteq W_n$ is connected, that is, for all $w^{(1)},w^{(2)}\in W$, there exists a finite sequence $\myset{v^{(1)},\ldots,v^{(k)}}\subseteq W$ such that $v^{(1)}=w^{(1)},v^{(k)}=w^{(2)}$ and $v^{(i)}\sim_nv^{(i+1)}$ for all $i=1,\ldots,k-1$. Let
$$\frakD_W(u,u):=
{\sum_{\mbox{\tiny
$
\begin{subarray}{c}
w^{(1)},w^{(2)}\in W\\
w^{(1)}\sim_nw^{(2)}
\end{subarray}
$
}}}
(u(w^{(1)})-u(w^{(2)}))^2,u\in l(W).$$
For all $w^{(1)},w^{(2)}\in W$, let
$$
\begin{aligned}
\frakR_W(w^{(1)},w^{(2)})&=\inf\myset{\frakD_W(u,u):u(w^{(1)})=0,u(w^{(2)})=1,u\in l(W)}^{-1}\\
&=\sup\myset{\frac{(u(w^{(1)})-u(w^{(2)}))^2}{\frakD_W(u,u)}:\frakD_W(u,u)\ne0,u\in l(W)}.
\end{aligned}
$$
It is obvious that
$$(u(w^{(1)})-u(w^{(2)}))^2\le\frakR_W(w^{(1)},w^{(2)})\frakD_W(u,u)\text{ for all }w^{(1)},w^{(2)}\in W,u\in l(W),$$
and $\frakR_W$ is a metric on $W$, hence
$$\frakR_W(w^{(1)},w^{(2)})\le\frakR_W(w^{(1)},w^{(3)})+\frakR_W(w^{(3)},w^{(2)})\text{ for all }w^{(1)},w^{(2)},w^{(3)}\in W.$$

Fix $w^{(1)}\sim_nw^{(2)}$, there exist $i,j=0,\ldots,7$ such that $w^{(1)}i^m\sim_{n+m}w^{(2)}j^m$, see Figure \ref{fig_monotonicity}.




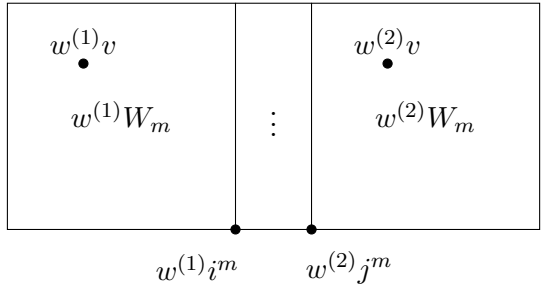
\begin{figure}[ht]
\centering
\begin{tikzpicture}

\draw (0,0)--(3,0)--(3,3)--(0,3)--cycle;
\draw (4,0)--(7,0)--(7,3)--(4,3)--cycle;
\draw (3,0)--(4,0);
\draw (3,3)--(4,3);

\draw (3.5,1.5) node {$\vdots$};

\draw[fill=black] (3,0) circle (0.06);
\draw[fill=black] (4,0) circle (0.06);

\draw (1.5,1.5) node {$w^{(1)}W_m$};
\draw (5.5,1.5) node {$w^{(2)}W_m$};
\draw (2.5,-0.5) node {$w^{(1)}i^m$};
\draw (4.5,-0.5) node {$w^{(2)}j^m$};

\draw (1,2.5) node {$w^{(1)}v$};
\draw[fill=black] (1,2.2) circle (0.06);
\draw (5,2.5) node {$w^{(2)}v$};
\draw[fill=black] (5,2.2) circle (0.06);

\end{tikzpicture}
\caption{$w^{(1)}W_m$ and $w^{(2)}W_m$}\label{fig_monotonicity}
\end{figure}


Fix $v\in W_m$
$$(u(w^{(1)}v)-u(w^{(2)}v))^2\le\frakR_{w^{(1)}W_m\cup w^{(2)}W_m}(w^{(1)}v,w^{(2)}v)\frakD_{w^{(1)}W_m\cup w^{(2)}W_m}(u,u).$$
By cutting technique and Corollary \ref{cor_resist_upper}
$$
\begin{aligned}
&\frakR_{w^{(1)}W_m\cup w^{(2)}W_m}(w^{(1)}v,w^{(2)}v)\\
&\le\frakR_{w^{(1)}W_m\cup w^{(2)}W_m}(w^{(1)}v,w^{(1)}i^m)+\frakR_{w^{(1)}W_m\cup w^{(2)}W_m}(w^{(1)}i^m,w^{(2)}j^m)\\
&+\frakR_{w^{(1)}W_m\cup w^{(2)}W_m}(w^{(2)}j^m,w^{(2)}v)\\
&\le\frakR_m(v,i^m)+1+\frakR_m(v,j^m)\lesssim\rho^m.
\end{aligned}
$$
Hence
$$
\begin{aligned}
&(u(w^{(1)}v)-u(w^{(2)}v))^2\lesssim\rho^m\frakD_{w^{(1)}W_m\cup w^{(2)}W_m}(u,u)\\
&=\rho^m\left(\frakD_{w^{(1)}W_m}(u,u)+\frakD_{w^{(2)}W_m}(u,u)\right.\\
&\left.+
{\sum_{\mbox{\tiny
$
\begin{subarray}{c}
v^{(1)},v^{(2)}\in W_m\\
w^{(1)}v^{(1)}\sim_{n+m}w^{(2)}v^{(2)}
\end{subarray}
$
}}}
(u(w^{(1)}v^{(1)})-u(w^{(2)}v^{(2)}))^2\right).
\end{aligned}
$$
Hence
$$
\begin{aligned}
&\left(M_{n,m}u(w^{(1)})-M_{n,m}u(w^{(2)})\right)^2=\left(\frac{1}{8^m}\sum_{v\in W_m}\left(u(w^{(1)}v)-u(w^{(2)}v)\right)\right)^2\\
&\le\frac{1}{8^m}\sum_{v\in W_m}\left(u(w^{(1)}v)-u(w^{(2)}v)\right)^2\\
&\lesssim\rho^m\left(\frakD_{w^{(1)}W_m}(u,u)+\frakD_{w^{(2)}W_m}(u,u)\right.\\
&\left.+
{\sum_{\mbox{\tiny
$
\begin{subarray}{c}
v^{(1)},v^{(2)}\in W_m\\
w^{(1)}v^{(1)}\sim_{n+m}w^{(2)}v^{(2)}
\end{subarray}
$
}}}
(u(w^{(1)}v^{(1)})-u(w^{(2)}v^{(2)}))^2\right).
\end{aligned}
$$
In the summation with respect to $w^{(1)}\sim_nw^{(2)}$, the terms $\frakD_{w^{(1)}W_m}(u,u),\frakD_{w^{(2)}W_m}(u,u)$ are summed at most $8$ times, hence
$$
\begin{aligned}
B_n(M_{n,m}u)&=\rho^n\sum_{w^{(1)}\sim_nw^{(2)}}\left(M_{n,m}u(w^{(1)})-M_{n,m}u(w^{(2)})\right)^2\\
&\lesssim\rho^n\sum_{w^{(1)}\sim_nw^{(2)}}\rho^m\left(\frakD_{w^{(1)}W_m}(u,u)+\frakD_{w^{(2)}W_m}(u,u)\right.\\
&\left.+
{\sum_{\mbox{\tiny
$
\begin{subarray}{c}
v^{(1)},v^{(2)}\in W_m\\
w^{(1)}v^{(1)}\sim_{n+m}w^{(2)}v^{(2)}
\end{subarray}
$
}}}
(u(w^{(1)}v^{(1)})-u(w^{(2)}v^{(2)}))^2\right)\\
&\le8\rho^{n+m}\sum_{w^{(1)}\sim_{n+m}w^{(2)}}\left(u(w^{(1)})-u(w^{(2)})\right)^2=8B_{n+m}(u).
\end{aligned}
$$
\end{proof}

\section{One Good Function}\label{sec_good}

In this section, we construct \emph{one} good function with energy property and separation property.

By standard argument, we have H\"older continuity from Harnack inequality as follows.

\begin{mythm}\label{thm_holder}
For all $0\le\delta_1<\veps_1<\veps_2<\delta_2\le1$, there exist some positive constants $\theta=\theta(\delta_1,\delta_2,\veps_1,\veps_2)$, $C=C(\delta_1,\delta_2,\veps_1,\veps_2)$ such that for all $n\ge1$, for all bounded harmonic function $u$ on $V_n\cap(\delta_1,\delta_2)\times[0,1]$, we have
$$|u(x)-u(y)|\le C|x-y|^\theta\left(\max_{V_n\cap[\delta_1,\delta_2]\times[0,1]}|u|\right)\text{ for all }x,y\in V_n\cap[\veps_1,\veps_2]\times[0,1].$$
\end{mythm}
\begin{proof}
The proof is similar to \cite[Theorem 3.9]{BB89}.
\end{proof}

For all $n\ge1$. Let $u_n\in l(V_n)$ satisfy $u_n|_{V_n\cap\myset{0}\times[0,1]}=0,u_n|_{V_n\cap\myset{1}\times[0,1]}=1$ and
$$D_n(u_n,u_n)=\sum_{w\in W_n}
{\sum_{\mbox{\tiny
$
\begin{subarray}{c}
p,q\in V_w\\
|p-q|=2^{-1}\cdot3^{-n}
\end{subarray}
$
}}}
(u_n(p)-u_n(q))^2=(R_n^V)^{-1}.$$
Then $u_n$ is harmonic on $V_n\cap(0,1)\times[0,1]$, $u_n(x,y)=1-u_n(1-x,y)=u_n(x,1-y)$ for all $(x,y)\in V_n$ and
$$u_n|_{V_n\cap\myset{\frac{1}{2}}\times[0,1]}=\frac{1}{2},u_n|_{V_n\cap[0,\frac{1}{2})\times[0,1]}<\frac{1}{2},u_n|_{V_n\cap(\frac{1}{2},1]\times[0,1]}>\frac{1}{2}.$$
By Arzel\`a-Ascoli theorem, Theorem \ref{thm_holder} and diagonal argument, there exist some subsequence still denoted by $\myset{u_n}$ and some function $u$ on $K$ with $u|_{\myset{0}\times[0,1]}=0$ and $u|_{\myset{1}\times[0,1]}=1$ such that $u_n$ converges uniformly to $u$ on $K\cap[\veps_1,\veps_2]\times[0,1]$ for all $0<\veps_1<\veps_2<1$. Hence $u$ is continuous on $K\cap(0,1)\times[0,1]$, $u_n(x)\to u(x)$ for all $x\in K$ and $u(x,y)=1-u(1-x,y)=u(x,1-y)$ for all $(x,y)\in K$.

\begin{myprop}\label{prop_u}
The function $u$ given above has the following properties.
\begin{enumerate}[(1)]
\item There exists some positive constant $C$ such that
$$a_n(u)\le C\text{ for all }n\ge1.$$
\item For all $\beta\in(\alpha,\log(8\rho)/\log3)$, we have
$$E_{\beta}(u,u)<+\infty.$$
Hence $u\in C^{\frac{\beta-\alpha}{2}}(K)$.
\item
$$u|_{K\cap\myset{\frac{1}{2}}\times[0,1]}=\frac{1}{2},u|_{K\cap[0,\frac{1}{2})\times[0,1]}<\frac{1}{2},u|_{K\cap(\frac{1}{2},1]\times[0,1]}>\frac{1}{2}.$$
\end{enumerate}
\end{myprop}

\begin{proof}
(1) By Theorem \ref{thm_resist1} and Theorem \ref{thm_monotone1}, for all $n\ge1$, we have
$$
\begin{aligned}
&a_n(u)=\lim_{m\to+\infty}a_{n}(u_{n+m})\le C\varliminf_{m\to+\infty}a_{n+m}(u_{n+m})\\
&=C\varliminf_{m\to+\infty}\rho^{n+m}D_{n+m}(u_{n+m},u_{n+m})=C\varliminf_{m\to+\infty}\rho^{n+m}\left(R_{n+m}^V\right)^{-1}\le C.
\end{aligned}
$$

(2) By (1), for all $\beta\in(\alpha,\log(8\rho)/\log3)$, we have
$$E_{\beta}(u,u)=\sum_{n=1}^\infty\left(3^{\beta-\alpha}\rho^{-1}\right)^na_n(u)\le C\sum_{n=1}^\infty\left(3^{\beta-\alpha}\rho^{-1}\right)^n<+\infty.$$
By Lemma \ref{lem_equiv} and Lemma \ref{lem_holder}, we have $u\in C^{\frac{\beta-\alpha}{2}}(K)$.

(3) It is obvious that
$$u|_{K\cap\myset{\frac{1}{2}}\times[0,1]}=\frac{1}{2},u|_{K\cap[0,\frac{1}{2})\times[0,1]}\le\frac{1}{2},u|_{K\cap(\frac{1}{2},1]\times[0,1]}\ge\frac{1}{2}.$$
By symmetry, we only need to show that
$$u|_{K\cap(\frac{1}{2},1]\times[0,1]}>\frac{1}{2}.$$
Suppose there exists $(x,y)\in K\cap(1/2,1)\times[0,1]$ such that $u(x,y)=1/2$. Since $u_n-\frac{1}{2}$ is a nonnegative harmonic function on $V_n\cap(\frac{1}{2},1)\times[0,1]$, by Theorem \ref{thm_harnack}, for all $1/2<\veps_1<x<\veps_2<1$, there exists some positive constant $C=C(\veps_1,\veps_2)$ such that for all $n\ge1$
$$\max_{V_n\cap[\veps_1,\veps_2]\times[0,1]}\left(u_n-\frac{1}{2}\right)\le C\min_{V_n\cap[\veps_1,\veps_2]\times[0,1]}\left(u_n-\frac{1}{2}\right).$$
Since $u_n$ converges uniformly to $u$ on $K\cap[\veps_1,\veps_2]\times[0,1]$, we have
$$\sup_{K\cap[\veps_1,\veps_2]\times[0,1]}\left(u-\frac{1}{2}\right)\le C\inf_{K\cap[\veps_1,\veps_2]\times[0,1]}\left(u-\frac{1}{2}\right)=0.$$
Hence
$$u-\frac{1}{2}=0\text{ on }K\cap[\veps_1,\veps_2]\times[0,1]\text{ for all }\frac{1}{2}<\veps_1<x<\veps_2<1.$$
Hence
$$u=\frac{1}{2}\text{ on }K\cap(\frac{1}{2},1)\times[0,1].$$
By continuity, we have
$$u=\frac{1}{2}\text{ on }K\cap[\frac{1}{2},1]\times[0,1],$$
contradiction!
\end{proof}

\section{Proof of Theorem \ref{thm_walk}}\label{sec_walk}
First, we consider upper bound. Assume that $(\Ee_\beta,\Ff_\beta)$ is a regular Dirichlet form on $L^2(K;\nu)$, then there exists $u\in\Ff_\beta$ such that $u|_{\myset{0}\times[0,1]}=0$ and $u|_{\myset{1}\times[0,1]}=1$. Hence
$$
\begin{aligned}
+\infty&>E_\beta(u,u)=\sum_{n=1}^\infty3^{(\beta-\alpha)n}D_n(u,u)\ge\sum_{n=1}^\infty3^{(\beta-\alpha)n}D_n(u_n,u_n)\\
&=\sum_{n=1}^\infty3^{(\beta-\alpha)n}\left(R_n^V\right)^{-1}\ge C\sum_{n=1}^\infty\left(3^{\beta-\alpha}\rho^{-1}\right)^n.
\end{aligned}
$$
Hence $3^{\beta-\alpha}\rho^{-1}<1$, that is, $\beta<{\log\left(8\rho\right)}/{\log3}=\beta^*$. Hence $\beta_*\le\beta^*$.

Second, we consider lower bound. Similar to the proof of Proposition \ref{prop_lower}, to show that $(\calE_\beta,\calF_\beta)$ is a regular Dirichlet form on $L^2(K;\nu)$ for all $\beta\in(\alpha,\beta^*)$, we only need to show that $\Ff_\beta$ separates points.

Let $u\in C(K)$ be the function in Proposition \ref{prop_u}. By Proposition \ref{prop_u} (2), we have $E_{\beta}(u,u)<+\infty$, hence $u\in\Ff_\beta$. 

For all distinct $z_1=(x_1,y_1),z_2=(x_2,y_2)\in K$, without lose of generality, we may assume that $x_1<x_2$. Replacing $z_i$ by $f_w^{-1}(z_i)$ with some $w\in W_n$ and some $n\ge1$, we only have the following cases.
\begin{enumerate}[(1)]
\item $x_1\in[0,\frac{1}{2}),x_2\in[\frac{1}{2},1]$.
\item $x_1\in[0,\frac{1}{2}],x_2\in(\frac{1}{2},1]$.
\item $x_1,x_2\in[0,\frac{1}{2})$, there exist distinct $w_1,w_2\in\myset{0,1,5,6,7}$ such that
$$z_1\in K_{w_1}\backslash K_{w_2}\text{ and }z_2\in K_{w_2}\backslash K_{w_1}.$$
\item $x_1,x_2\in(\frac{1}{2},1]$, there exist distinct $w_1,w_2\in\myset{1,2,3,4,5}$ such that
$$z_1\in K_{w_1}\backslash K_{w_2}\text{ and }z_2\in K_{w_2}\backslash K_{w_1}.$$
\end{enumerate}
For the first case, $u(z_1)<{1}/{2}\le u(z_2)$. For the second case, $u(z_1)\le{1}/{2}<u(z_2)$.




\begin{figure}[ht]
\centering
\begin{tikzpicture}[scale=0.5]

\draw (0,0)--(6,0)--(6,6)--(0,6)--cycle;
\draw (2,0)--(2,6);
\draw (4,0)--(4,6);
\draw (0,2)--(6,2);
\draw (0,4)--(6,4);

\draw (1,1) node {$K_0$};
\draw (3,1) node {$K_1$};
\draw (5,1) node {$K_2$};
\draw (5,3) node {$K_3$};
\draw (5,5) node {$K_4$};
\draw (3,5) node {$K_5$};
\draw (1,5) node {$K_6$};
\draw (1,3) node {$K_7$};

\end{tikzpicture}
\caption{The Location of $z_1,z_2$}\label{fig_characterization}
\end{figure}


For the third case. If $w_1,w_2$ do \emph{not} belong to the same one of the following sets
$$\myset{0,1},\myset{7},\myset{5,6},$$
then we construct a function $w$ as follows. Let $v(x,y)=u(y,x)$ for all $(x,y)\in K$, then
$$v|_{[0,1]\times\myset{0}}=0,v|_{[0,1]\times\myset{1}}=1,$$
$$v(x,y)=v(1-x,y)=1-v(x,1-y)\text{ for all }(x,y)\in K,$$
$$E_\beta(v,v)=E_\beta(u,u)<+\infty.$$
Let
$$w=
\begin{cases}
v\circ f_i^{-1}-1,&\text{on }K_i,i=0,1,2,\\
v\circ f_i^{-1},&\text{on }K_i,i=3,7,\\
v\circ f_i^{-1}+1,&\text{on }K_i,i=4,5,6,\\
\end{cases}
$$
then $w\in C(K)$ is well-defined and $E_\beta(w,w)<+\infty$, hence $w\in\Ff_\beta$. Moreover, $w(z_1)\ne w(z_2)$, $w|_{[0,1]\times\myset{0}}=-1,w|_{[0,1]\times\myset{1}}=2,w(x,y)=w(1-x,y)=1-w(x,1-y)$ for all $(x,y)\in K$.

If $w_1,w_2$ \emph{do} belong to the same one of the following sets
$$\myset{0,1},\myset{7},\myset{5,6},$$
then it can only happen that $w_1,w_2\in\myset{0,1}$ or $w_1,w_2\in\myset{5,6}$, without lose of generality, we may assume that $w_1=0$ and $w_2=1$, then $z_1\in K_0\backslash K_1$ and $z_2\in K_1\backslash K_0$.

Let
$$
w=
\begin{cases}
u\circ f_i^{-1}-1,&\text{on }K_i,i=0,6,7,\\
u\circ f_i^{-1},&\text{on }K_i,i=1,5,\\
u\circ f_i^{-1}+1,&\text{on }K_i,i=2,3,4,\\
\end{cases}
$$
then $w\in C(K)$ is well-defined and $E_{\beta}(w,w)<+\infty$, hence $w\in\Ff_\beta$. Moreover $w(z_1)\ne w(z_2)$, $w|_{\myset{0}\times[0,1]}=-1,w|_{\myset{1}\times[0,1]}=2,w(x,y)=w(x,1-y)=1-w(1-x,y)$ for all $(x,y)\in K$.

For the forth case, by reflection about $\myset{\frac{1}{2}}\times[0,1]$, we reduce to the third case.

Hence $\Ff_\beta$ separates points, hence $(\Ee_\beta,\Ff_\beta)$ is a regular Dirichlet form on $L^2(K;\nu)$ for all $\beta\in(\alpha,\beta^*)$, hence $\beta_*\ge\beta^*$.

In conclusion, $\beta_*=\beta^*$.

\section{Proof of Theorem \ref{thm_BM}}\label{sec_BM}

In this section, we use $\Gamma$-convergence technique to construct a local regular Dirichlet form on $L^2(K;\nu)$ which corresponds to the BM. The idea of this construction is from \cite{KS05}.

The construction of local Dirichlet forms on p.c.f. self-similar sets relies heavily on some monotonicity result which is ensured by some compatibility condition, see \cite{Kig93,Kig01}. Our key observation is that even with some weak monotonicity results, we still apply $\Gamma$-convergence technique to obtain some limit.

We need some preparation about $\Gamma$-convergence.

In what follows, $K$ is a locally compact separable metric space and $\nu$ is a Radon measure on $K$ with full support. We say that $(\Ee,\Ff)$ is a \emph{closed form on $L^2(K;\nu)$ in the wide sense} if $\mathcal{F}$ is complete under the inner product $\Ee_1$ but $\Ff$ is not necessary to be dense in $L^2(K;\nu)$. If $(\Ee,\Ff)$ is a closed form on $L^2(K;\nu)$ in the wide sense, we extend $\Ee$ to be $+\infty$ outside $\Ff$, hence the information of $\Ff$ is encoded in $\Ee$.

\begin{mydef}\label{def_gamma}
Let $\Ee^n,\Ee$ be closed forms on $L^2(K;\nu)$ in the wide sense. We say that $\Ee^n$ is $\Gamma$-convergent to $\Ee$ if the following conditions are satisfied.
\begin{enumerate}[(1)]
\item For all $\myset{u_n}\subseteq L^2(K;\nu)$ that converges strongly to $u\in L^2(K;\nu)$, we have
$$\varliminf_{n\to+\infty}\Ee^n(u_n,u_n)\ge\Ee(u,u).$$
\item For all $u\in L^2(K;\nu)$, there exists a sequence $\myset{u_n}\subseteq L^2(K;\nu)$ converging strongly to $u$ in $L^2(K;\nu)$ such that
$$\varlimsup_{n\to+\infty}\Ee^n(u_n,u_n)\le\Ee(u,u).$$
\end{enumerate}
\end{mydef}

We have the following result about $\Gamma$-convergence.

\begin{myprop}\label{prop_gamma}(\cite[Proposition 6.8, Theorem 8.5, Theorem 11.10, Proposition 12.16]{Dal93})
Let $\myset{(\Ee^n,\Ff^n)}$ be a sequence of closed forms on $L^2(K;\nu)$ in the wide sense, then there exist some subsequence $\myset{(\Ee^{n_k},\Ff^{n_k})}$ and some closed form $(\Ee,\Ff)$ on $L^2(K;\nu)$ in the wide sense such that $\Ee^{n_k}$ is $\Gamma$-convergent to $\Ee$.
\end{myprop}

In what follows, $K$ is SC and $\nu$ is Hausdorff measure.

We need an elementary result as follows.

\begin{myprop}\label{prop_ele}
Let $\myset{x_n}$ be a sequence of nonnegative real numbers.
\begin{enumerate}[(1)]
\item $$\varliminf_{n\to+\infty}x_n\le\varliminf_{\lambda\uparrow1}(1-\lambda)\sum_{n=1}^\infty\lambda^nx_n\le\varlimsup_{\lambda\uparrow1}(1-\lambda)\sum_{n=1}^\infty\lambda^nx_n\le\varlimsup_{n\to+\infty}x_n\le\sup_{n\ge1}x_n.$$
\item If there exists some positive constant $C$ such that
$$x_n\le Cx_{n+m}\text{ for all }n,m\ge1,$$
then
$$\sup_{n\ge1}x_n\le C\varliminf_{n\to+\infty}x_n.$$
\end{enumerate}
\end{myprop}
\begin{proof}
The proof is elementary using $\veps$-$N$ argument.
\end{proof}

Take $\myset{\beta_n}\subseteq(\alpha,\beta^*)$ with $\beta_n\uparrow\beta^*$. By Proposition \ref{prop_gamma}, there exist some subsequence still denoted by $\myset{\beta_n}$ and some closed form $(\Ee,\Ff)$ on $L^2(K;\nu)$ in the wide sense such that $(\beta^*-\beta_n)\frakE_{\beta_n}$ is $\Gamma$-convergent to $\Ee$. Without lose of generality, we may assume that
$$0<\beta^*-\beta_n<\frac{1}{n+1}\text{ for all }n\ge1.$$

We have the characterization of $(\Ee,\Ff)$ on $L^2(K;\nu)$ as follows.

\begin{mythm}\label{thm_gamma}
$$
\begin{aligned}
&\Ee(u,u)\asymp\sup_{n\ge1}3^{(\beta^*-\alpha)n}\sum_{w\in W_n}
{\sum_{\mbox{\tiny
$
\begin{subarray}{c}
p,q\in V_w\\
|p-q|=2^{-1}\cdot3^{-n}
\end{subarray}
$
}}}
(u(p)-u(q))^2,\\
&\Ff=\myset{u\in C(K):\sup_{n\ge1}3^{(\beta^*-\alpha)n}\sum_{w\in W_n}
{\sum_{\mbox{\tiny
$
\begin{subarray}{c}
p,q\in V_w\\
|p-q|=2^{-1}\cdot3^{-n}
\end{subarray}
$
}}}
(u(p)-u(q))^2<+\infty}.
\end{aligned}
$$
Moreover, $(\calE,\calF)$ is a regular closed form on $L^2(K;\nu)$.
\end{mythm}

\begin{proof}
Recall that $\rho=3^{\beta^*-\alpha}$, then
$$
\begin{aligned}
E_{\beta}(u,u)&=\sum_{n=1}^\infty3^{(\beta-\alpha)n}\sum_{w\in W_n}
{\sum_{\mbox{\tiny
$
\begin{subarray}{c}
p,q\in V_w\\
|p-q|=2^{-1}\cdot3^{-n}
\end{subarray}
$
}}}
(u(p)-u(q))^2=\sum_{n=1}^\infty3^{(\beta-\beta^*)n}a_n(u),\\
\frakE_\beta(u,u)&=\sum_{n=1}^\infty3^{(\beta-\alpha)n}
{\sum_{\mbox{\tiny
$
\begin{subarray}{c}
w^{(1)}\sim_nw^{(2)}
\end{subarray}
$
}}}
\left(P_nu(w^{(1)})-P_nu(w^{(2)})\right)^2=\sum_{n=1}^\infty3^{(\beta-\beta^*)n}b_n(u).
\end{aligned}
$$

We use weak monotonicity results Theorem \ref{thm_monotone1}, Theorem \ref{thm_monotone2} and elementary result Proposition \ref{prop_ele}.

For all $u\in L^2(K;\nu)$, there exists $\myset{u_n}\subseteq L^2(K;\nu)$ converging strongly to $u$ in $L^2(K;\nu)$ such that
$$
\begin{aligned}
&\Ee(u,u)\ge\varlimsup_{n\to+\infty}(\beta^*-\beta_n)\frakE_{\beta_n}(u_n,u_n)=\varlimsup_{n\to+\infty}(\beta^*-\beta_n)\sum_{k=1}^\infty3^{(\beta_n-\beta^*)k}b_k(u_n)\\
&\ge\varlimsup_{n\to+\infty}(\beta^*-\beta_n)\sum_{k=n+1}^\infty3^{(\beta_n-\beta^*)k}b_k(u_n)\ge C\varlimsup_{n\to+\infty}(\beta^*-\beta_n)\sum_{k=n+1}^\infty3^{(\beta_n-\beta^*)k}b_n(u_n)\\
&=C\varlimsup_{n\to+\infty}\left\{b_n(u_n)\left[(\beta^*-\beta_n)\frac{3^{(\beta_n-\beta^*)(n+1)}}{1-3^{\beta_n-\beta^*}}\right]\right\}.
\end{aligned}
$$
Since $0<\beta^*-\beta_n<1/(n+1)$, we have $3^{(\beta_n-\beta^*)(n+1)}>1/3$. Since
$$\lim_{n\to+\infty}\frac{\beta^*-\beta_n}{1-3^{\beta_n-\beta^*}}=\frac{1}{\log3},$$
there exists some positive constant $C$ such that 
$$(\beta^*-\beta_n)\frac{3^{(\beta_n-\beta^*)(n+1)}}{1-3^{\beta_n-\beta^*}}\ge C\text{ for all }n\ge1.$$
Hence
$$\Ee(u,u)\ge C\varlimsup_{n\to+\infty}b_n(u_n).$$
Since $u_n\to u$ in $L^2(K;\nu)$, for all $k\ge1$, we have
$$b_k(u)=\lim_{n\to+\infty}b_k(u_n)=\lim_{k\le n\to+\infty}b_k(u_n)\le C\varliminf_{n\to+\infty}b_n(u_n).$$
For all $m\ge1$, we have
$$
\begin{aligned}
(\beta^*-\beta_m)\sum_{k=1}^\infty3^{(\beta_m-\beta^*)k}b_k(u)&\le C(\beta^*-\beta_m)\sum_{k=1}^\infty3^{(\beta_m-\beta^*)k}\varliminf_{n\to+\infty}b_n(u_n)\\
&=C(\beta^*-\beta_m)\frac{3^{\beta_m-\beta^*}}{1-3^{\beta_m-\beta^*}}\varliminf_{n\to+\infty}b_n(u_n).
\end{aligned}
$$
Hence $\calE(u,u)<+\infty$ implies $\frakE_{\beta_m}(u,u)<+\infty$, by Lemma \ref{lem_holder}, we have $\calF\subseteq C(K)$. Hence
$$\varliminf_{m\to+\infty}(\beta^*-\beta_m)\sum_{k=1}^\infty3^{(\beta_m-\beta^*)k}b_k(u)\le C\varliminf_{n\to+\infty}b_n(u_n).$$
Hence for all $u\in\calF\subseteq C(K)$, we have
$$
\begin{aligned}
\Ee(u,u)&\ge C\varlimsup_{n\to+\infty}b_n(u_n)\ge C\varliminf_{n\to+\infty}b_n(u_n)\ge C\varliminf_{m\to+\infty}(\beta^*-\beta_m)\sum_{k=1}^\infty3^{(\beta_m-\beta^*)k}b_k(u)\\
&\ge C\varliminf_{m\to+\infty}(\beta^*-\beta_m)\sum_{k=1}^\infty3^{(\beta_m-\beta^*)k}a_k(u)\ge C\sup_{n\ge1}a_n(u).
\end{aligned}
$$

On the other hand, for all $u\in\calF\subseteq C(K)$, we have
$$
\begin{aligned}
&\Ee(u,u)\le\varliminf_{n\to+\infty}(\beta^*-\beta_n)\frakE_{\beta_n}(u,u)\\
&\le C\varliminf_{n\to+\infty}(\beta^*-\beta_n)E_{\beta_n}(u,u)=C\varliminf_{n\to+\infty}(\beta^*-\beta_n)\sum_{k=1}^\infty3^{(\beta_n-\beta^*)k}a_k(u)\\
&=C\varliminf_{n\to+\infty}\frac{\beta^*-\beta_n}{1-3^{\beta_n-\beta^*}}(1-3^{\beta_n-\beta^*})\sum_{k=1}^\infty3^{(\beta_n-\beta^*)k}a_k(u)\le C\sup_{n\ge1}a_n(u).
\end{aligned}
$$
Therefore, for all $u\in\calF\subseteq C(K)$, we have
$$\Ee(u,u)\asymp\sup_{n\ge1}a_n(u)=\sup_{n\ge1}3^{(\beta^*-\alpha)n}\sum_{w\in W_n}
{\sum_{\mbox{\tiny
$
\begin{subarray}{c}
p,q\in V_w\\
|p-q|=2^{-1}\cdot3^{-n}
\end{subarray}
$
}}}
(u(p)-u(q))^2,$$
and
$$\Ff=\myset{u\in C(K):\sup_{n\ge1}3^{(\beta^*-\alpha)n}\sum_{w\in W_n}
{\sum_{\mbox{\tiny
$
\begin{subarray}{c}
p,q\in V_w\\
|p-q|=2^{-1}\cdot3^{-n}
\end{subarray}
$
}}}
(u(p)-u(q))^2<+\infty}.$$

It is obvious that the function $u\in C(K)$ in Proposition \ref{prop_u} is in $\Ff$. Similar to the proof of Theorem \ref{thm_walk}, we have $\Ff$ is uniformly dense in $C(K)$. Hence $(\Ee,\Ff)$ is a regular closed form on $L^2(K;\nu)$.
\end{proof}

Now we prove Theorem \ref{thm_BM} as follows.

\begin{proof}[Proof of Theorem \ref{thm_BM}]
For all $n\ge1,u\in l(V_{n+1})$, we have
$$
\begin{aligned}
&\rho\sum_{i=0}^7a_n(u\circ f_i)=\rho\sum_{i=0}^7\rho^n\sum_{w\in W_n}
{\sum_{\mbox{\tiny
$
\begin{subarray}{c}
p,q\in V_w\\
|p-q|=2^{-1}\cdot3^{-n}
\end{subarray}
$
}}}
(u\circ f_i(p)-u\circ f_i(q))^2\\
&=\rho^{n+1}\sum_{w\in W_{n+1}}
{\sum_{\mbox{\tiny
$
\begin{subarray}{c}
p,q\in V_w\\
|p-q|=2^{-1}\cdot3^{-(n+1)}
\end{subarray}
$
}}}
(u(p)-u(q))^2=a_{n+1}(u).
\end{aligned}
$$
Hence for all $n,m\ge1,u\in l(V_{n+m})$, we have
$$\rho^m\sum_{w\in W_m}a_n(u\circ f_w)=a_{n+m}(u).$$
For all $u\in\calF,n\ge1,w\in W_n$, we have
$$\sup_{k\ge1}a_k(u\circ f_w)\le\sup_{k\ge1}\sum_{w\in W_n}a_k(u\circ f_w)=\rho^{-n}\sup_{k\ge1}a_{n+k}(u)\le\rho^{-n}\sup_{k\ge1}a_{k}(u)<+\infty,$$
hence $u\circ f_w\in\calF$.

Let
$$\mybar{\Ee}^{(n)}(u,u)=\rho^n\sum_{w\in W_n}\Ee(u\circ f_w,u\circ f_w),u\in\Ff,n\ge1.$$
Then
$$
\begin{aligned}
\mybar{\Ee}^{(n)}(u,u)&\ge C\rho^n\sum_{w\in W_n}\varlimsup_{k\to+\infty}a_k(u\circ f_w)\ge C\rho^n\varlimsup_{k\to+\infty}\sum_{w\in W_n}a_k(u\circ f_w)\\
&=C\varlimsup_{k\to+\infty}a_{n+k}(u)\ge C\sup_{k\ge1}a_k(u).
\end{aligned}
$$
Similarly
$$
\begin{aligned}
\mybar{\Ee}^{(n)}(u,u)&\le C\rho^n\sum_{w\in W_n}\varliminf_{k\to+\infty}a_k(u\circ f_w)\le C\rho^n\varliminf_{k\to+\infty}\sum_{w\in W_n}a_k(u\circ f_w)\\
&=C\varliminf_{k\to+\infty}a_{n+k}(u)\le C\sup_{k\ge1}a_k(u).
\end{aligned}
$$
Hence
$$\mybar{\Ee}^{(n)}(u,u)\asymp\sup_{k\ge1}a_k(u)\text{ for all }u\in\Ff,n\ge1.$$
Moreover, for all $u\in\Ff$, $n\ge1$, we have
$$
\begin{aligned}
&\mybar{\Ee}^{(n+1)}(u,u)=\rho^{n+1}\sum_{w\in W_{n+1}}\Ee(u\circ f_w,u\circ f_w)=\rho^{n+1}\sum_{i=0}^7\sum_{w\in W_n}\Ee(u\circ f_i\circ f_w,u\circ f_i\circ f_w)\\
&=\rho\sum_{i=0}^7\left(\rho^n\sum_{w\in W_n}\Ee((u\circ f_i)\circ f_w,(u\circ f_i)\circ f_w)\right)=\rho\sum_{i=0}^7\mybar{\Ee}^{(n)}(u\circ f_i,u\circ f_i).
\end{aligned}
$$
Let
$$\tilde{\Ee}^{(n)}(u,u)=\frac{1}{n}\sum_{l=1}^n\mybar{\Ee}^{(l)}(u,u),u\in\calF,n\ge1.$$
It is obvious that
$$\tilde{\Ee}^{(n)}(u,u)\asymp\sup_{k\ge1}a_k(u)\text{ for all }u\in\Ff,n\ge1.$$
Since $(\Ee,\Ff)$ is a regular closed form on $L^2(K;\nu)$, by \cite[Definition 1.3.8, Remark 1.3.9, Definition 1.3.10, Remark 1.3.11]{CF12}, we have $(\Ff,\Ee_1)$ is a separable Hilbert space. Let $\myset{u_i}_{i\ge1}$ be a dense subset of $(\Ff,\Ee_1)$. For all $i\ge1$, $\myset{\tilde{\Ee}^{(n)}(u_i,u_i)}_{n\ge1}$ is a bounded sequence. By diagonal argument, there exists a subsequence $\myset{n_k}_{k\ge1}$ such that $\myset{\tilde{\Ee}^{(n_k)}(u_i,u_i)}_{k\ge1}$ converges for all $i\ge1$. Since
$$\tilde{\Ee}^{(n)}(u,u)\asymp\sup_{k\ge1}a_k(u)\asymp\Ee(u,u)\text{ for all }u\in\Ff,n\ge1,$$
we have $\myset{\tilde{\Ee}^{(n_k)}(u,u)}_{k\ge1}$ converges for all $u\in\Ff$. Let
$$\Ee_{\loc}(u,u)=\lim_{k\to+\infty}\tilde{\Ee}^{(n_k)}(u,u)\text{ for all }u\in\Ff_{\loc}:=\Ff.$$
Then
$$\Ee_{\loc}(u,u)\asymp\sup_{k\ge1}a_k(u)\asymp\Ee(u,u)\text{ for all }u\in\calF_\loc=\calF.$$
Hence $(\Ee_\loc,\Ff_\loc)$ is a regular closed form on $L^2(K;\nu)$. It is obvious that $1\in\Ff_\loc$ and $\Ee_\loc(1,1)=0$, by \cite[Lemma 1.6.5, Theorem 1.6.3]{FOT11}, we have $(\Ee_\loc,\Ff_\loc)$ on $L^2(K;\nu)$ is conservative.

For all $u\in\Ff_\loc=\calF$, we have $u\circ f_i\in\calF=\Ff_\loc$ for all $i=0,\ldots,7$ and
$$
\begin{aligned}
&\rho\sum_{i=0}^7\Ee_\loc(u\circ f_i,u\circ f_i)=\rho\sum_{i=0}^7\lim_{k\to+\infty}\tilde{\Ee}^{(n_k)}(u\circ f_i,u\circ f_i)\\
&=\rho\sum_{i=0}^7\lim_{k\to+\infty}\frac{1}{n_k}\sum_{l=1}^{n_k}\mybar{\Ee}^{(l)}(u\circ f_i,u\circ f_i)=\lim_{k\to+\infty}\frac{1}{n_k}\sum_{l=1}^{n_k}\left[\rho\sum_{i=0}^7\mybar{\Ee}^{(l)}(u\circ f_i,u\circ f_i)\right]\\
&=\lim_{k\to+\infty}\frac{1}{n_k}\sum_{l=1}^{n_k}\mybar{\Ee}^{(l+1)}(u,u)=\lim_{k\to+\infty}\frac{1}{n_k}\sum_{l=2}^{n_k+1}\mybar{\Ee}^{(l)}(u,u)\\
&=\lim_{k\to+\infty}\left[\frac{1}{n_k}\sum_{l=1}^{n_k}\mybar{\Ee}^{(l)}(u,u)+\frac{1}{n_k}\mybar{\Ee}^{(n_k+1)}(u,u)-\frac{1}{n_k}\mybar{\Ee}^{(1)}(u,u)\right]\\
&=\lim_{k\to+\infty}\tilde{\Ee}^{(n_k)}(u,u)=\Ee_\loc(u,u).
\end{aligned}
$$
Hence $(\Ee_\loc,\Ff_\loc)$ on $L^2(K;\nu)$ is self-similar.

For all $u,v\in\Ff_\loc$ satisfying $\mathrm{supp}(u),\mathrm{supp}(v)$ are compact and $v$ is constant in an open neighborhood $U$ of $\mathrm{supp}(u)$, we have $K\backslash U$ is compact and $\mathrm{supp}(u)\cap(K\backslash U)=\emptyset$, hence $\delta=\mathrm{dist}(\mathrm{supp}(u),K\backslash U)>0$. Taking sufficiently large $n\ge1$ such that $3^{1-n}<\delta$, by self-similarity, we have
$$\Ee_\loc(u,v)=\rho^n\sum_{w\in W_n}\Ee_\loc(u\circ f_w,v\circ f_w).$$
For all $w\in W_n$, we have $u\circ f_w=0$ or $v\circ f_w$ is constant, hence $\Ee_\loc(u\circ f_w,v\circ f_w)=0$, hence $\Ee_\loc(u,v)=0$, that is, $(\Ee_\loc,\Ff_\loc)$ on $L^2(K;\nu)$ is strongly local.

For all $u\in\calF_\loc$, it is obvious that $u^+,u^-,1-u,\mybar{u}=(0\vee u)\wedge1\in\calF_\loc$ and
$$\calE_\loc(u,u)=\calE_\loc(1-u,1-u).$$
Since $u^+u^-=0$ and $(\calE_\loc,\calF_\loc)$ on $L^2(K;\nu)$ is strongly local, we have $\calE_\loc(u^+,u^-)=0$. Hence
$$
\begin{aligned}
\calE_\loc(u,u)&=\calE_\loc(u^+-u^-,u^+-u^-)=\calE_\loc(u^+,u^+)+\calE_\loc(u^-,u^-)-2\calE_\loc(u^+,u^-)\\
&=\calE_\loc(u^+,u^+)+\calE_\loc(u^-,u^-)\ge\calE_\loc(u^+,u^+)=\calE_\loc(1-u^+,1-u^+)\\
&\ge\calE_\loc((1-u^+)^+,(1-u^+)^+)=\calE_\loc(1-(1-u^+)^+,1-(1-u^+)^+)=\calE_\loc(\mybar{u},\mybar{u}),
\end{aligned}
$$
that is, $(\calE_\loc,\calF_\loc)$ on $L^2(K;\nu)$ is Markovian. Hence $(\calE_\loc,\calF_\loc)$ is a self-similar strongly local regular Dirichlet form on $L^2(K;\nu)$.
\end{proof}

\begin{myrmk}
The idea of the construction of $\mybar{\calE}^{(n)},\tilde{\calE}^{(n)}$ is from \cite[Section 6]{KZ92}. The proof of Markovain property is from the proof of \cite[Theorem 2.1]{BBKT10}.
\end{myrmk}

\section{Proof of Theorem \ref{thm_Besov}}\label{sec_Besov}

Theorem \ref{thm_Besov} is a special case of the following result.

\begin{myprop}\label{prop_equiv_local}
For all $\beta\in(\alpha,+\infty),u\in C(K)$, we have
$$\sup_{n\ge1}3^{(\beta-\alpha)n}\sum_{w\in W_n}
{\sum_{\mbox{\tiny
$
\begin{subarray}{c}
p,q\in V_w\\
|p-q|=2^{-1}\cdot3^{-n}
\end{subarray}
$
}}}
(u(p)-u(q))^2\asymp[u]_{B^{2,\infty}_{\alpha,\beta}(K)}.$$
\end{myprop}

Similar to non-local case, we need the following preparation.

\begin{mylem}(\cite[Theorem 4.11 (\rmnum{3})]{GHL03})\label{lem_holder_local}
Let $u\in L^2(K;\nu)$ and
$$F(u):=\sup_{n\ge1}3^{(\alpha+\beta)n}\int_K\int_{B(x,3^{-n})}(u(x)-u(y))^2\nu(\md y)\nu(\md x),$$
then
$$|u(x)-u(y)|^2\le cF(u)|x-y|^{\beta-\alpha}\text{ for }\nu\text{-almost every }x,y\in K,$$
where $c$ is some positive constant.
\end{mylem}

\begin{myrmk}
If $F(u)<+\infty$, then $u\in C^{\frac{\beta-\alpha}{2}}(K)$.
\end{myrmk}

\begin{proof}[Proof of Proposition \ref{prop_equiv_local}]
The proof is very similar to that of Lemma \ref{lem_equiv}. We only point out the differences. To show that LHS$\lesssim$RHS, by the proof of Theorem \ref{thm_equiv1}, we still have Equation (\ref{eqn_equiv1_1}) where $E(u)$ is replaced by $F(u)$. Then
$$
\begin{aligned}
&3^{(\beta-\alpha)n}\sum_{w\in W_n}
{\sum_{\mbox{\tiny
$
\begin{subarray}{c}
p,q\in V_w\\
|p-q|=2^{-1}\cdot3^{-n}
\end{subarray}
$
}}}
(u(p)-u(q))^2\\
&\le128\cdot2^{(\beta-\alpha)/2}cF(u)3^{\beta n-(\beta-\alpha)(n+kl)}+32\cdot3^{\alpha k}\sum_{i=0}^{l-1}2^i\cdot3^{-(\beta-\alpha)ki}E_{n+ki}(u).
\end{aligned}
$$
Take $l=n$, then
$$
\begin{aligned}
&3^{(\beta-\alpha)n}\sum_{w\in W_n}
{\sum_{\mbox{\tiny
$
\begin{subarray}{c}
p,q\in V_w\\
|p-q|=2^{-1}\cdot3^{-n}
\end{subarray}
$
}}}
(u(p)-u(q))^2\\
&\le128\cdot2^{(\beta-\alpha)/2}cF(u)3^{[\beta-(\beta-\alpha)(k+1)]n}+32\cdot3^{\alpha k}\sum_{i=0}^{n-1}2^i\cdot3^{-(\beta-\alpha)ki}E_{n+ki}(u)\\
&\le128\cdot2^{(\beta-\alpha)/2}cF(u)3^{[\beta-(\beta-\alpha)(k+1)]n}+32\cdot3^{\alpha k}\sum_{i=0}^{\infty}3^{[1-(\beta-\alpha)k]i}\left(\sup_{n\ge1}E_{n}(u)\right).
\end{aligned}
$$
Take $k\ge1$ sufficiently large such that $\beta-(\beta-\alpha)(k+1)<0$ and $1-(\beta-\alpha)k<0$, then
$$
\begin{aligned}
&\sup_{n\ge1}3^{(\beta-\alpha)n}\sum_{w\in W_n}
{\sum_{\mbox{\tiny
$
\begin{subarray}{c}
p,q\in V_w\\
|p-q|=2^{-1}\cdot3^{-n}
\end{subarray}
$
}}}
(u(p)-u(q))^2\\
&\lesssim\sup_{n\ge1}3^{(\alpha+\beta)n}\int_K\int_{B(x,3^{-n})}(u(x)-u(y))^2\nu(\md y)\nu(\md x).
\end{aligned}
$$

To show that LHS$\gtrsim$RHS, by the proof of Theorem \ref{thm_equiv2}, we still have Equation (\ref{eqn_equiv2_3}). Then
$$
\begin{aligned}
&\sup_{n\ge2}3^{(\alpha+\beta)n}\int_K\int_{B(x,c3^{-n})}(u(x)-u(y))^2\nu(\md y)\nu(\md x)\\
&\lesssim\sup_{n\ge2}\sum_{k=n}^\infty4^{k-n}\cdot3^{\beta n-\alpha k}\sum_{w\in W_k}
{\sum_{\mbox{\tiny
$
\begin{subarray}{c}
p,q\in V_w\\
|p-q|=2^{-1}\cdot3^{-k}
\end{subarray}
$
}}}
(u(p)-u(q))^2\\
&+\sup_{n\ge2}3^{(\beta-\alpha)n}\sum_{w\in W_{n-1}}
{\sum_{\mbox{\tiny
$
\begin{subarray}{c}
p,q\in V_w\\
|p-q|=2^{-1}\cdot3^{-(n-1)}
\end{subarray}
$
}}}
(u(p)-u(q))^2\\
&\lesssim\sup_{n\ge2}\sum_{k=n}^\infty4^{k-n}\cdot3^{\beta(n-k)}\left(\sup_{k\ge1}3^{(\beta-\alpha)k}\sum_{w\in W_k}
{\sum_{\mbox{\tiny
$
\begin{subarray}{c}
p,q\in V_w\\
|p-q|=2^{-1}\cdot3^{-k}
\end{subarray}
$
}}}
(u(p)-u(q))^2\right)\\
&+\sup_{n\ge1}3^{(\beta-\alpha)n}\sum_{w\in W_n}
{\sum_{\mbox{\tiny
$
\begin{subarray}{c}
p,q\in V_w\\
|p-q|=2^{-1}\cdot3^{-n}
\end{subarray}
$
}}}
(u(p)-u(q))^2\\
&\lesssim\sup_{n\ge1}3^{(\beta-\alpha)n}\sum_{w\in W_n}
{\sum_{\mbox{\tiny
$
\begin{subarray}{c}
p,q\in V_w\\
|p-q|=2^{-1}\cdot3^{-n}
\end{subarray}
$
}}}
(u(p)-u(q))^2.
\end{aligned}
$$

\end{proof}

We have the following properties of Besov spaces for large exponents.

\begin{mycor}\label{cor_chara}
$B^{2,2}_{\alpha,\beta^*}(K)=\myset{\text{constant functions}}$, $B^{2,\infty}_{\alpha,\beta^*}(K)$ is uniformly dense in $C(K)$. $B^{2,2}_{\alpha,\beta}(K)=B^{2,\infty}_{\alpha,\beta}(K)=\myset{\text{constant functions}}$ for all $\beta\in(\beta^*,+\infty)$.
\end{mycor}

\begin{proof}
By Theorem \ref{thm_BM} and Theorem \ref{thm_Besov}, we have $B^{2,\infty}_{\alpha,\beta^*}(K)$ is uniformly dense in $C(K)$. Assume that $u\in C(K)$ is non-constant, then there exists $N\ge1$ such that $a_N(u)>0$. By Theorem \ref{thm_monotone1}, for all $\beta\in[\beta^*,+\infty)$, we have
$$\sum_{n=1}^\infty3^{(\beta-\beta^*)n}a_n(u)\ge\sum_{n=N+1}^\infty3^{(\beta-\beta^*)n}a_n(u)\ge C\sum_{n=N+1}^\infty3^{(\beta-\beta^*)n}a_N(u)=+\infty,$$
for all $\beta\in(\beta^*,+\infty)$, we have
$$\sup_{n\ge1}3^{(\beta-\beta^*)n}a_n(u)\ge\sup_{n\ge N+1}3^{(\beta-\beta^*)n}a_n(u)\ge C\sup_{n\ge N+1}3^{(\beta-\beta^*)n}a_N(u)=+\infty.$$

By Lemma \ref{lem_equiv} and Proposition \ref{prop_equiv_local}, we have $B^{2,2}_{\alpha,\beta}(K)=\myset{\text{constant functions}}$ for all $\beta\in[\beta^*,+\infty)$ and $B^{2,\infty}_{\alpha,\beta}(K)=\myset{\text{constant functions}}$ for all $\beta\in(\beta^*,+\infty)$.
\end{proof}

\section{Proof of Theorem \ref{thm_hk}}\label{sec_hk}

We use effective resistance as follows.

Let $(M,d,\mu)$ be a metric measure space and $(\calE,\calF)$ a regular Dirichlet form on $L^2(M;\mu)$. Assume that $A,B$ are two disjoint subsets of $M$. Define \emph{effective resistance} as
$$R(A,B)=\inf\myset{\calE(u,u):u|_A=0,u|_B=1,u\in\calF\cap C_0(M)}^{-1}.$$
Denote
$$R(x,B)=R(\myset{x},B),R(x,y)=R(\myset{x},\myset{y}),x,y\in M.$$
It is obvious that if $A_1\subseteq A_2$, $B_1\subseteq B_2$, then
$$R(A_1,B_1)\ge R(A_2,B_2).$$

\begin{proof}[Proof of Theorem \ref{thm_hk}]
First, we show that
$$R(x,y)\asymp|x-y|^{\beta^*-\alpha}\text{ for all }x,y\in K.$$
By Lemma \ref{lem_holder_local}, we have
$$(u(x)-u(y))^2\le c\calE_\loc(u,u)|x-y|^{\beta^*-\alpha}\text{ for all }x,y\in K,u\in\calF_\loc,$$
hence
$$R(x,y)\lesssim|x-y|^{\beta^*-\alpha}\text{ for all }x,y\in K.$$
On the other hand, we claim
$$R(x,B(x,r)^c)\asymp r^{\beta^*-\alpha}\text{ for all }x\in K,r>0\text{ with }B(x,r)^c\ne\emptyset.$$
Indeed, fix $C>0$. If $u\in\calF_\loc$ satisfies $u(x)=1$, $u|_{B(x,r)^c}=0$, then $\tilde{u}:y\mapsto u(x+C(y-x))$ satisfies $\tilde{u}\in\calF_\loc$, $\tilde{u}(x)=1$, $\tilde{u}|_{B(x,Cr)^c}=0$. By Theorem \ref{thm_BM}, it is obvious that
$$\calE_\loc(\tilde{u},\tilde{u})\asymp C^{-(\beta^*-\alpha)}\calE_\loc(u,u),$$
hence
$$R(x,B(x,Cr)^c)\asymp C^{\beta^*-\alpha}R(x,B(x,r)^c).$$
Hence
$$R(x,B(x,r)^c)\asymp r^{\beta^*-\alpha}.$$
For all $x,y\in K$, we have
$$R(x,y)\ge R(x,B(x,|x-y|)^c)\asymp|x-y|^{\beta^*-\alpha}.$$

Then, we follow a standard analytic approach as follows. First, we obtain Green function estimates as in \cite[Proposition 6.11]{GHL14}. Then, we obtain heat kernel estimates as in \cite[Theorem 3.14]{GH14a}. Note that we are dealing with compact set, the final estimates only hold for some finite time $t\in(0,1)$.
\end{proof}

\bibliographystyle{siam}

\def\cprime{$'$}

\end{document}